\documentclass[leqno,10pt]{amsart}
\usepackage{amsmath,amscd,amsthm,amsxtra}
\usepackage{epsfig,graphics,color,colortbl}
\usepackage{amssymb,latexsym}
\usepackage{mathabx}
\usepackage{mathrsfs,eucal,upgreek}
\usepackage[poly,all]{xy}
\usepackage{hyperref}
\usepackage{tikz-cd}
\usepackage{arydshln}
\usepackage{ytableau}
\usepackage{adjustbox}
\usepackage{listings}
\usepackage[draft]{todonotes}
\usepackage{datetime} 
\usepackage{tabularx}
\usepackage{lipsum}
\usepackage{dynkin-diagrams}
\usepackage{marginnote}
\usepackage[normalem]{ulem}  
\usepackage[nocompress]{cite}  
\usepackage{float}
\usepackage{enumerate}
\usepackage{comment}
\usepackage{longtable}

\newtheorem{thm}{\bf Theorem}[section]
\newtheorem{df}[thm]{\bf Definition}
\newtheorem{prop}[thm]{\bf Proposition}
\newtheorem{cor}[thm]{\bf Corollary}
\newtheorem{lem}[thm]{\bf Lemma}
\newtheorem{rem}[thm]{\bf Remark}
\newtheorem{ex}[thm]{\bf Example}

\numberwithin{equation}{section}

\newcommand{\mr}{\mathring}
\newcommand{\ot}{\otimes}
\newcommand{\wtd}{\widetilde}
\newcommand{\ov}{\overline}

\newcommand{\mc}{\mathcal}
\newcommand{\mf}{\mathfrak}
\newcommand{\ms}{\mathscr}
\newcommand{\mb}{\mathbf}

\newcommand{\Z}{\mathbb{Z}}
\newcommand{\Q}{\mathbb{Q}}
\newcommand{\C}{\mathbb{C}}

\newcommand{\de}{\delta}

\newcommand{\bk}{\mathbf{k}}
\newcommand{\g}{\mathfrak{g}}
\newcommand{\bo}{\mathfrak{b}}
\newcommand{\h}{\mathfrak{h}}
\newcommand{\hd}{\mathfrak{h}^*}
\newcommand{\cg}{\mathring{\mathfrak{g}}}
\newcommand{\ch}{\mathring{\mathfrak{h}}}
\newcommand{\chd}{\mathring{\mathfrak{h}}^*}
\newcommand{\qi}{{q_i}}
\newcommand{\cW}{\mathring{W}}
\newcommand{\eW}{\widehat{W}}
\newcommand{\xz}{\mathsf{x}_0}
\newcommand{\GL}{\mathcal{GL}}

\newcommand{\al}{{\alpha}}
\newcommand{\E}{\mathsf{E}} 
\newcommand{\F}{\mathsf{F}} 
\newcommand{\qbn}[2]{{\left[\!\!\begin{array}{c} #1 \\ #2 \end{array}\!\!\right]}}

\newcommand{\sfA}{\mathsf{A}} 
\newcommand{\sfP}{\mathsf{P}} 
\newcommand{\sfQ}{\mathsf{Q}} 
\newcommand{\sfDel}{\mathsf{\Delta}} 
\newcommand{\sfD}{\mathsf{D}} 
\newcommand{\sfM}{\mathsf{M}} 
\newcommand{\eM}{\widehat{\sfM}} 
\newcommand{\kaction}[1]{\mb{t}_{#1}}
\newcommand{\eaction}[1]{\mb{e}_{#1}}

\newcommand{\I}{I_0}
\newcommand{\lar}{\lambda_r}
\newcommand{\tr}{t_r}

\makeatletter
\newcommand*{\shifttext}[2]{%
  \settowidth{\@tempdima}{#2}%
  \makebox[\@tempdima]{\hspace*{#1}#2}%
}
\makeatother

\usepackage{todonotes}

\title[cominuscule prefundamental representations]
{Unipotent quantum coordinate ring and cominuscule prefundamental representations}

\author{IL-SEUNG JANG}
\address[I.-S. Jang]{Department of Mathematics, Incheon National University, Incheon 22012, Republic of Korea}
\email{ilseungjang@inu.ac.kr}

\author{JAE-HOON KWON}
\address[J.-H. Kwon]{Department of Mathematical Sciences and RIM, Seoul National University, Seoul 08826, Republic of Korea}
\email{jaehoonkw@snu.ac.kr}

\author{EUIYONG PARK}
\address[E. Park]{Department of Mathematics, University of Seoul, Seoul 02504, Republic of Korea}
\email{epark@uos.ac.kr}


\thanks{I.-S Jang  is supported by the National Research Foundation of Korea(NRF) grant funded by the Korea government (MSIT) (No.~RS-2023-00277388)}
\thanks{J.-H. Kwon is supported by the National Research Foundation of Korea(NRF) grant funded by the Korea government (MSIT) (No.~NRF-2020R1A5A1016126 and RS-2024-00342349)}
\thanks{E. Park is supported by the National Research Foundation of Korea (NRF) Grant funded by the Korea Government (MSIT) (No.~RS-2023-00273425)}

\begin{document}
\begin{abstract}
We continue the study of realization of the prefundamental modules $L_{r,a}^{\pm}$, introduced by Hernandez and Jimbo, in terms of unipotent quantum coordinate rings as in \cite{JKP23}. We show that the ordinary character of $L_{r,a}^{\pm}$ is equal to that of the unipotent quantum coordinate ring $U_q^-(w_r)$ associated to fundamental $r$-th coweight. When $r$ is cominuscule, we prove that there exists a $U_q(\bo)$-module structure on $U_q^-(w_r)$, which is isomorphic to $L_{r,a\eta_r}^\pm$ for some $\eta_r \in \mathbb{C}^\times$.
\end{abstract}

\maketitle
\setcounter{tocdepth}{1}

\section{Introduction}
Let $\mf{g}$ be the untwisted affine Kac-Moody algebra associated with the generalized Cartan matrix $\sfA = (a_{ij})_{i, j \in I}$ with $I=\{\,0,\dots,n\,\}$. Let $U_q(\g)$ be the associated Drinfeld-Jimbo quantum group, and let $U_q'(\g)$ be the subquotient of $U_q(\g)$ without the degree operator.

Let $U_q(\bo)$ be the Borel subalgebra of $U_q'(\g)$.
In \cite{HJ}, Hernandez and Jimbo introduced a category $\mc{O}$ of $U_q(\mf{b})$-modules partly motivated by the existence of a limit of normalized $q$-characters of Kirillov-Reshetikhin modules \cite{Her06, Nak03}. 
This category contains the finite-dimensional $U_q'(\mf{g})$-modules, and an irreducible module in $\mc{O}$, which is infinite-dimensional in general, is characterized in terms of tuples ${\bf \Psi} = (\Psi_i(z))_{i \in \I}$ of rational functions ($I_0=I\setminus \{0\}$), which are regular and non-zero at $z = 0$.

For $a\in \C(q)^\times$ and $r\in I_0$, the irreducible $U_q(\mf{b})$-modules $L_{r,a}^\pm$ corresponding to $(\Psi_i(z))_{i\in I_0}$ such that
\begin{equation*}
\Psi_r(z) = (1-az)^{\pm 1},\quad \Psi_i(z)=1\quad (i\neq r),
\end{equation*}
are called the prefundamental modules, where $L_{r,a}^-$ (resp.~$L_{r,a}^+$) is called the negative (resp.~positive) prefundamental module. They are the building blocks in $\mc{O}$ since each irreducible module in $\mc{O}$ is a subquotient of a tensor product of prefundamental representations and one-dimensional modules. 

Let $\mathring{\g}$ be the subalgebra of $\mf{g}$ associated to $(a_{ij})_{i,j\in I_0}$.
The ordinary $\mathring{\g}$-character ${\rm ch} (L^\pm_{r,a})$ of $L^\pm_{r,a}$ is given by a formal power series in $\Z[\![e^{-\alpha_i}\,|\,i\in I_0\,]\!]$, where $\alpha_i$ denotes the simple root of $\mathring{\g}$ for $i\in I_0$.
There is a product formula given by
\begin{equation}\label{eq:MY formula}
{\rm ch} (L^\pm_{r,a})=\frac{1}{\prod_{\beta\in \mathring{\sfDel}^+}(1-e^{-\beta})^{[\beta]_r}},
\end{equation}
which was conjectured in \cite{MY14} and proved in \cite{Lee, Neg}.
Here $\mathring{\sfDel}^+$ is the set of positive roots of $\mathring{\g}$ and $[\beta]_r$ is the coefficient of the simple root $\alpha_r$ in $\beta$.

In this paper, we continue the study of realization of $L_{r,a}^\pm$ in terms of unipotent quantum coordinate rings as in \cite{JKP23}.
Let $W$ (resp.~$\eW$) be the (resp.~extended) affine Weyl group associated with $\sfA$, where $\eW \simeq W \ltimes \mc{T}$ with $\mc{T}$ the group of automorphisms of the Dynkin diagram of $\sfA$.
For $r\in I_0$, let $\lambda_r$ be a weight in the dual Cartan subalgebra $\h^*$ of $\g$ corresponding to the $r$-th fundamental coweight of $\mathring{\g}$, and let $w_r\in W$ be such that $t_{-\lambda_r}=w_r\tau_r \in \eW$, where $t_{-\lambda_r}\in \eW$ denotes the translation on $\h^*$ associated to $-\lambda_r$ and $\tau_r \in\mc{T}$.
Let $U_q^-(w_r)$ denote the unipotent quantum coordinate ring generated by the root vectors associated to a reduced expression of $w_r$ corresponding to $\sfDel^+\cap w_r(-\sfDel^+)$, where $\sfDel^+$ is the set of positive roots of $\g$.

We first prove that $U_q^-(w_r)$ has the same $\mathring{\g}$-character as $L_{r,a}^\pm$, that is, 
\begin{equation}\label{eq:ch of Uw_r}
	{\rm ch} \left( U_q^-(w_r) \right)   = \frac{1}{\prod_{\beta\in \mathring{\sfDel}^+}(1-e^{-\beta})^{[\beta]_r}}.
\end{equation}
Combining \eqref{eq:MY formula} and \eqref{eq:ch of Uw_r}, one may expect to have  a $U_q(\mf{b})$-module structure on $U_q(w_r)$ which is isomorphic to $L^\pm_{r,a}$.
Indeed, we constructed such a $U_q(\mf{b})$-module structure when $\g$ is of type $A_n^{(1)}$ and $D_n^{(1)}$, and $r$ is minuscule in \cite{JKP23}. The action of the Chevalley generator $e_i\in U_q(\mf{b})$ for $i\in I_0$ is given by the usual $q$-derivation on $U_q^-(\g)$, while the action of $e_0$ is given by a left or right multiplication by the root vector corresponding to the negative maximal root of $\mathring{\g}$.

The main result in this paper is to extend the result in \cite{JKP23} to the case when $r$ is cominuscule, that is, $B_n^{(1)}$ with $r=1$, $C_n^{(1)}$ with $r=n$, $E_6^{(1)}$ with $r=1,6$, and $E_7^{(1)}$ with $r=7$ including the cases of $A_n^{(1)}$ and $D_n^{(1)}$ with $r$ minuscule, equivalently cominuscule (hence, $[\beta]_r=1$ for all $\beta\in \mathring{\Delta}^+$ in \eqref{eq:MY formula}). We prove the following.

\begin{thm}[Theorem \ref{thm:main3}, Theorem \ref{thm:main3-1}] \label{thm:positive case}
When $r$ is cominuscule, there exists a $U_q(\bo)$-module structure on $U_q^-(w_r)$, which is isomorphic to $L_{r,a\eta_r}^\pm$ for some $\eta_r \in \mathbb{C}^\times$.
\end{thm}

We remark that the parameter $\eta_r$ only depends on the rank of the underlying simple Lie algebra of $\g$ up to a sign associated with $r$. For example, for type $A_n^{(1)}$, we have $\eta_r = (-1)^n q^{-n-1} (q-q^{-1}) o(r)$, where $o : I \,\rightarrow\, \{ \pm 1 \}$ be a map such that $o(i) = -o(j)$ whenever $a_{ij} < 0$.

The $U_q(\mf{b})$-module structures on $U^-_q(w_r)$ are still defined in almost the same way as in \cite{JKP23}, while the method in this paper is more improved and unified compared to \cite{JKP23}.
In particular, when the action of $e_0$ is given by a left multiplication by the root vector corresponding to the maximal root in $\sfDel^+\cap w_r\left(-\sfDel^+\right)$ with respect to a reduced expression of $w_r$, the space $U^-_q(w_r)$ has a well-defined $U_q(\mf{b})$-module structure for any $r$. But it is not irreducible unless $r$ is cominuscule. 
Thus the $U_q(\bo)$-module $U^-_q(w_r)$ is not isomorphic to $L_{r,a}^\pm$ in general except for cominuscule $r$ (see Remark \ref{rem:Mw} for further discussion). So we expect a completely different description of $e_0$ on $U_q^-(w_r)$ to have an isomorphism to $L^\pm_{r,a}$ for non-cominuscule $r$.

When we prove Theorem \ref{thm:positive case}, it is crucial to compute the $\ell$-highest weight of $1 \in U_q^-(w_r)$ with respect to the $U_q(\bo)$-actions. This is more involved than in \cite{JKP23} beyond types $A_n^{(1)}$ and $D_n^{(1)}$. 
So we use the quantum shuffle approach following \cite{Le04} on root vectors in $U_q^-(w_r)$ to have an explicit description of the action of $e_i$ ($i\neq 0$) on them (see Appendix \ref{appendixA}) and the crystal structure developed in \cite{Kw13,JK19, Jang22} (see Remark \ref{rem:action of ei on dual vector} for more details).

Recently, a category of modules over Borel subalgebra of twisted affine Kac-Moody algebra is developed in \cite{Wang23} to establish twisted $Q\widetilde{Q}$-systems, and then the prefundamental modules in that category are also constructed following \cite{HJ}. It would be interesting to realize those modules by our approach for some cases.
Also, there is a functor $\mc{F}$ from a category of $U_{q^{-1}}(\bo)$-modules to that of $U_q(\bo)$-modules such that $\mc{F}(L_{r,a}^{\pm,q^{-1}})$ are also prefundamental modules \cite{P} (cf.~\cite{HL16}), where $L_{r,a}^{\pm,q^{-1}}$ are prefundamental modules over $U_{q^{-1}}(\bo)$. In fact, there exists $\gamma \in \mathbb{C}^\times$ such that $\mc{F}(L_{r,a}^{\pm,q^{-1}}) \cong L_{r,a\gamma}^\mp$ for all $i \in \I$ and $a \in \mathbb{C}^\times$. On the other hand, the parameters $\eta_r$'s in Theorem \ref{thm:positive case} for $L_{r,a}^+$ and $L_{r,a}^-$ are the same up to a sign. This similar phenomenon may suggest a connection between Theorem \ref{thm:positive case} and \cite{P}.

The paper is organized as follows.
In Section \ref{sec:braid group on boson}, 
we review the realization of the $q$-boson algebra following \cite{Lu10}, and then introduce the braid group symmetries on it. 
In Section \ref{sec:extended affine Weyl groups}, 
we recall some properties of $t_{-\lambda_r}$,
which play crucial roles throughout the paper.
In Section \ref{sec:pseudo-prefundamentals}, 
we show that $U_q^-(w_r)$ has a $U_q(\bo)$-module structure, which belongs to $\mc{O}$.
In Section \ref{sec:prefundamentals}, we prove that the character of $U_q^-(w_r)$ coincides with that of $L_{r,a}^\pm$ (Theorem \ref{thm:main2}), and when $r$ is cominuscule, it is isomorphic to $L_{r,a}^\pm$ up to shift of spectral parameter $a$ (Theorem \ref{thm:main3} and Theorem \ref{thm:main3-1}).

\vskip 3mm
\noindent
{\em \bf Acknowledgement.}
The proof of Lemma \ref{lem:l-highest weight} is inspired by a comment of an anonymous referee of \cite{JKP23}.
We sincerely thank the anonymous referee for their careful reading and many valuable comments.

\section{Braid group symmetries on $q$-boson algebra} \label{sec:braid group on boson}
\subsection{Cartan data and quantum group} \label{subsec:cartan data and quantum group}
Let $\sfA = (a_{ij})_{i,j \in I}$ be a symmetrizable generalized Cartan matrix with an index set $I$.
Following \cite{Kac}, we assume that
\begin{itemize}
	\item $\sfD = {\rm diag}(d_i)_{i \in I}$ is a diagonal matrix such that $d_i \in \Z$ and $\sfD\sfA$ is symmetric,
	\item $\sfP^\vee$ is a free abelian group of rank $2|I|-{\rm rank}A$, which is the dual weight lattice,
	\item $\Pi^\vee = \left\{\,h_i \in P^\vee \,|\, i\in I\,\right\}$ is the set of simple coroots, 
	\item $\h = \mathbb{C} \otimes_{\,\mathbb{Z}} \sfP^\vee$ is the $\mathbb{C}$-linear space spanned by $\sfP^\vee$, and $\h^*$ is the dual space of $\h$,
	\item $\sfP = \left\{\,\lambda\in\h^*\,|\,\lambda(P^\vee) \subset \Z\,\right\}$ is the weight lattice,
	\item $\Pi = \left\{\,\alpha_i\in\h^*\,|\,i\in I\,\right\}$ is the set of simple roots,
	\item $\sfDel = \sfDel^+ \cup \sfDel^-$ is the set of roots, where $\sfDel^+$ (resp.~$\sfDel^-$) is the set of positive (resp.~negative) roots,
\end{itemize}
\noindent
and let $\g$ be the Kac-Moody algebra over $\C$ associated with the Cartan datum $\left(\, \sfA, \sfP^\vee, \Pi^\vee, \sfP, \Pi \,\right)$, where $\h$ is the Cartan subalgebra of $\g$.
Take a symmetric bilinear $\C$-valued form $\left(\,\cdot\,,\,\cdot\,\right)$ on $\h^*$ such that $(\alpha_i, \alpha_i) \in 2\Z$ for $i \in I$ following \cite[\textsection\,2.1]{Kac}. 
In particular, for $i, j \in I$,
\begin{equation} \label{eq:aij}
a_{ij} = \langle h_i, \alpha_j \rangle = \frac{2(\alpha_i, \alpha_j)}{(\alpha_i, \alpha_i)},
\end{equation}
where $\langle \,\cdot\,,\,\cdot\, \rangle$ is a pairing between $\h$ and $\h^*$. Here $\h$ is regarded as a dual space of $\h^*$, so $\langle h_i, \alpha_j \rangle = \alpha_j\left(h_i\right) = \langle \alpha_j, h_i \rangle$ for $i, j \in I$.

Fix an indeterminate $q$. Set $\bk = \mathbb{C}(q)$ to be the base field.
We put $q_i = q^{d_i}$ and
{\allowdisplaybreaks
\begin{gather*}
[m]_{q_i}=\frac{q_i^m-q_i^{-m}}{q_i-q_i^{-1}}\quad (m\in \Z_+),\quad
[m]_{q_i}!=[m]_{q_i}[m-1]_{q_i}\cdots [1]_{q_i}\quad (m\geq 1),\quad [0]_{q_i}!=1, \\
\begin{bmatrix} m \\ k \end{bmatrix}_{q_i} = \frac{[m]_{q_i}[m-1]_{q_i}\cdots [m-k+1]_{q_i}}{[k]_{q_i}}\quad (0\leq k\leq m).
\end{gather*}}
\!\!If there is no confusion, then we often write $[m]_i$ and $\begin{bmatrix} m \\ k \end{bmatrix}_i$ instead of  $[m]_{q_i}$ and $\begin{bmatrix} m \\ k \end{bmatrix}_{q_i}$ for simplicity, respectively.

The {\it quantum group} corresponding to $\g$, denoted by $U_q(\g)$, is the $\bk$-algebra generated by the symbols $e_i$, $f_i$ $(i \in I)$, and $q^h$ $(h \in \sfP^\vee)$ subject to the following relations:
{\allowdisplaybreaks
\begin{gather*}
	q^0 = 1, \quad q^h q^{h'} = q^{h+h'}, \\
	q^h e_i q^{-h}= q^{\langle h, \alpha_i \rangle} e_i,\quad q^h f_i q^{-h}= q^{-\langle h, \alpha_i \rangle} f_i,\quad 
	e_if_j-f_je_i=\delta_{ij}\frac{k_i-k_i^{-1}}{\qi-\qi^{-1}},\\
	\sum_{m=0}^{1-a_{ij}}(-1)^{m} \qbn{1-a_{ij}}{m}_i e_i^{1-a_{ij}-m}e_je_i^{m}=0, \quad (i \neq j), \\
	\,\sum_{m=0}^{1-a_{ij}}(-1)^{m} \qbn{1-a_{ij}}{m}_i f_i^{1-a_{ij}-m}f_jf_i^{m}=0, \quad (i \neq j),
\end{gather*}}
\!\!where $k_i = q^{d_ih_i}$.
We set $e_i^{(m)}=e_i^m/[m]_i!$ and $f_i^{(m)}=f_i^m/[m]_i!$ for $i \in I$ and $m \in \Z_+$.

Let $U_q^0(\g)$ be the $\bk$-subalgebra of $U_q(\g)$ generated by $q^h$ for $h \in \sfP^\vee$.
Also, we denote by $U_q^+(\g)$ (resp.~$U_q^-(\g)$) the $\bk$-subalgebra of $U_q(\g)$ generated by $e_i$ (resp.~$f_i$) for $i \in I$.
Then $U_q(\g)$ has the {\it triangular decomposition} 
\begin{equation} \label{eq:triangular decomp}
	U_q(\g) \cong U_q^+(\g) \otimes U_q^0(\g) \otimes U_q^-(\g)
\end{equation}
as a $\bk$-vector space (e.g.~see \cite{HK}).

An element $x \in U_q(\g)$ is said to be homogeneous if there exists $\xi \in \sfP$ such that $q^h x q^{-h} = q^{\langle h, \xi \rangle} x$ for all $h \in \sfP^\vee$, where we denote $\xi$ by ${\rm wt}(x)$ which is called the {\it weight} of $x$.

Let us adopt a Hopf algebra structure on $U_q(\g)$, where the comultiplication $\Delta$ and the antipode $S$ are given by 
{\allowdisplaybreaks
\begin{gather}
	\Delta(q^h)= q^h \ot q^h, \quad
	\Delta(e_i)= e_i \ot 1 + k_i \ot e_i, \quad
	\Delta(f_i)= f_i \ot k_i^{-1} + 1 \ot f_i, \label{eq:comultiplication} \\
	S(q^h)=q^{-h}, \ \ S(e_i)=-k_i^{-1} e_i, \ \  S(f_i)=-f_i k_i, \label{eq:antipode}
\end{gather}}
\!\!for $i \in I$.
We remark that $\Delta$ in \eqref{eq:comultiplication} is denoted by $\Delta_+$ in \cite{Kas91}, which coincides with the choice of $\Delta$ in \cite{HJ}.

\subsection{$q$-Boson algebras}
Let $\ms{B}_q(\g)$ be the $\bk$-algebra generated by the symbols $e_i'$ and $f_i$ for $i \in I$ satisfying the following relations:
{\allowdisplaybreaks
\begin{gather*}
	e_i'f_j = q_i^{-\langle h_i, \alpha_j \rangle} f_j e_i' + \de_{ij}, \\
	\sum_{m=0}^{1-a_{ij}}(-1)^{m} \qbn{1-a_{ij}}{m}_i e_i'^{1-a_{ij}-m}e_j'e_i'^{m}=0, \quad (i \neq j), \label{eq:qsr for e'} \\
	\sum_{m=0}^{1-a_{ij}}(-1)^{m} \qbn{1-a_{ij}}{m}_i f_i^{1-a_{ij}-m}f_jf_i^{m}=0, \quad (i \neq j), \label{eq:qsr for f}
	\end{gather*}}
\!\!for $i, j \in I$ (see~\cite{Kas91}).
Let $\ms{B}_q^+(\g)$ (resp.~$\ms{B}_q^-(\g)$) be the $\bk$-subalgebra of $\ms{B}_q(\g)$ generated by $e_i'$ (resp.~$f_i$) for $i \in I$.
Note that $\ms{B}_q^\pm(\g) \cong U_q^\pm(\g)$ as a $\bk$-algebra.

\begin{lem}\text{\!\!\!\rm \cite[Corollary 3.4.9]{Kas91}} \label{lem:negative half is irreducible module over boson algebra}
The $\bk$-algebra $U_q^-(\g)$ is an irreducible $\ms{B}_q(\g)$-module, where $f_i$ acts on $U_q^-(\g)$ by the left multiplication of $f_i \in U_q^-(\g)$, and $e_i'$ acts on $U_q^-(\g)$ inductively by $e_i'(1) = 0$, $e_i'(f_j) = \de_{ij}$, and 
\begin{equation} \label{eq:derivation ei'}
	e_i'(xy) = e_i'(x)y +  q_i^{\langle h_i, {\rm wt}(x) \rangle} x e_i'(y)
\end{equation}
for homogeneous elements $x, y \in U_q^-(\g)$.
\end{lem}

\begin{rem} \label{rem:eistar}
{\em 
	There exists another $q$-derivation of $U_q^-(\g)$ (see \cite[Lemma 8.2.1]{KKKO18}), which is slightly different from a $q$-derivation $e_i''$ in \cite{Kas91}.
	More precisely, for $i \in I$, let $e_i^\star$ be the $q$-derivation of $U_q^-(\g)$, which is defined by $e_i^\star(f_j) = \delta_{ij}$ $(j \in I)$ and, for homogeneous $x, y \in U_q^-(\g)$, 
	$$
		e_i^\star(xy) = q_i^{\langle h_i, {\rm wt}(y) \rangle} e_i^\star(x)y + x e_i^\star(y).
	$$
	Following \cite{Kas91}, we have
	 $e_i' e_i^\star  = e_i^\star e_i'$ 
	and $e_i^\star = *\circ e_i'\circ *$,
	where $*$ is the antiautomorphism of $U_q(\g)$ as a $\bk$-algebra given by $e_i^* = e_i$, $f_i^* = f_i$, $(q^h)^* = q^{-h}$ for $i \in I$ and $h \in \sfP^\vee$.
}
\end{rem}

Put 
\begin{equation} \label{eq: def of mbei'}
\mb{e}'_i:=-(q_i-q_i^{-1})k_ie_i
\end{equation}
for $i \in I$.
Let $\tilde{U}_q(\g)$ be the $\bk$-subalgebra of $U_q(\mf{g})$ generated by $\mb{e}'_i, k_i$, and $f_i$ for $i \in I$.
We denote by $\tilde{U}_q^+(\g)$ (resp.~$\tilde{U}^0_q(\mf{g})$) the $\bk$-subalgebra of $\tilde{U}_q(\g)$ generated by $\mb{e}'_i$ (resp.~$k_i$) for $i \in I$.
Note that $k_i$ is not invertible in $\tilde{U}_q(\g)$.
\begin{prop} \!{\rm(\!\!\cite{Lu10})} \label{prop:properties of tUqg} 
\mbox{}
\begin{itemize}
	\item[(a)] There exists an isomorphism of $\bk$-algebras between $U_q^+(\g)$ and $\tilde{U}_q^+(\g)$ by $e_i \mapsto \mb{e}_i'$.
		
	\item[(b)] 
	By the multiplication on $\tilde{U}_q(\g)$, we have
	\begin{equation*}
		\tilde{U}_q(\g)\,\cong\, \tilde{U}_q^+(\g) \ot \tilde{U}_q^0(\g) \ot U_q^-(\g)
	\end{equation*}
	as a $\bk$-vector space.

	\item[(c)] For $i, j \in I$, we have \,$\mb{e}'_if_j=q_i^{-\langle h_i,\alpha_j \rangle} f_j\mb{e}'_i +\de_{ij}(1-k_i^2)$.
\end{itemize}
\end{prop}

Let $J$ be the two-sided ideal of $\tilde{U}_q(\g)$ generated by $k_i$ for all $i \in I$.
We define 
\begin{equation} \label{eq:mfU}
	\mf{U} = \tilde{U}_q(\g) / J.
\end{equation}
By abuse of notation, let us write $\mb{e}_i'$ and $f_i$ for the image of $\mb{e}_i'$ and $f_i$ in $\mf{U}$, respectively.
Then we denote by $\mf{U}^+$ (resp.~$\mf{U}^-$) the $\bk$-subalgebra of $\mf{U}$ generated by $\mb{e}_i'$ (resp.~$f_i$) for $i \in I$.

\begin{lem}{\em \!\!\!\cite{Lu10}} \label{lem:realization of Bqg}
There exists an isomorphism of $\bk$-algebras from $\mf{U}$ to $\ms{B}_q(\g)$ such that $\mb{e}_i' \mapsto e_i'$ and $f_i \mapsto f_i$ for $i \in I$.
\end{lem}

In summary, we have
\begin{equation*}
	\xymatrix@R=0.5em @C=-0.45em{
	U_q(\g) & \supset & \tilde{U}_q(\g) \ar@{->>}[drr] &  &  & \\
	& & \ms{B}_ q(\g) & \cong & \mf{U} &:=\tilde{U}_q(\g)/J
	}
\end{equation*}
By using this realization of $\ms{B}_q(\g)$, we shall consider the braid group symmetries in \cite{Lu10} on $\ms{B}_q(\g)$ in the subsequent section.

\begin{rem}
{\em
	The $q$-boson algebra $\ms{B}_q(\mathfrak{sl}_2)$ is related to the {\it $q$-oscillator algebra} \cite{BLZ}.
	More precisely, the $q$-oscillator algebra denoted by $\ms{U}_q^+(\mathfrak{sl}_2)$, is a $\bk$-algebra generated by $e, f, k^{\pm 1}$ with relations
	$$
		ke = q^2 ek, \quad
		kf = q^{-2} fk, \quad
		kk^{-1} = k^{-1}k = 1, \quad
		[e,f] = \frac{k}{q-q^{-1}}.
	$$
	Here we note that the relation involving $[e,f]$ is modified compared to one in $U_q(\mathfrak{sl}_2)$. Then the $\bk$-subalgebra of $\ms{U}_q^+(\mathfrak{sl}_2)$ generated by $f$ and $e' := (q-q^{-1})k^{-1}e$ is isomorphic to $\ms{B}_q(\mathfrak{sl}_2)$. 
	On the other hand, it is known in \cite{BLZ} that there exists a homomorphism from the $\bk$-subalgebra of $U_q(\widehat{\mathfrak{sl}}_2)$ generated by $e_i, k_i^{\pm 1}$ for $i \in I$ (called the {\it Borel subalgebra} of $U_q(\widehat{\mathfrak{sl}}_2)$) to $\ms{U}_q^+(\mathfrak{sl}_2)$, which looks similar to the evaluation morphism $U_q(\widehat{\mathfrak{sl}}_2) \rightarrow U_q(\mathfrak{sl}_2)$ \cite{J} (see also \cite[Remark 2.2 and Proposition 3.8]{Her23}).
}	
\end{rem}

Let
\begin{equation} \label{eq:def of M}
\begin{split}
	M = 
	\frac{\tilde{U}_q(\g)}{\displaystyle \,\,\,\sum_{i \in I} \tilde{U}_q(\g)\mb{e}_i' + \sum_{i \in I} \tilde{U}_q(\g)k_i\,\,\,}
\end{split}
\end{equation}
be the left $\tilde{U}_q(\g)$-module by left multiplication.
Since each $k_i$ acts on $M$ as zero due to 
the relations $k_i \mb{e}_j' = q_i^{\langle h_i, \alpha_j \rangle} \mb{e}_j' k_i$ and $k_i f_j = q_i^{-\langle h_i, \alpha_j \rangle}f_j k_i$ for $i, j \in I$,
we regard $M$ as a $\mf{U}$-module, where 
\begin{equation*}
	{\bf x} \cdot \overline{m} = \overline{x m},
\end{equation*}
where ${\bf x} = x + J \in \mf{U}$ and $\overline{m} \in M$.

\begin{lem} \label{lem:M is isom to Uq- as Bq-mod}
$M$ is isomorphic to $U_q^-(\g)$ as a $\ms{B}_q(\g)$-module.
\end{lem}
\begin{proof}
By Proposition \ref{prop:properties of tUqg}(b)--(c), we have $M \,\cong\, U_q^-(\g)$ as a $\bk$-vector space, so the canonical projection $\iota$ from $U_q^-(\g)$ to $M$ gives an isomorphism of $\bk$-vector spaces. Then it follows from Lemma \ref{lem:negative half is irreducible module over boson algebra} and Proposition \ref{prop:properties of tUqg}(c) that 
the following diagram commutes:
\begin{equation*}
\begin{split}
\xymatrix@C=5em@R=2em{
	U_q^-(\g) \ar@{->}[r]^\iota \ar@{->}[d]_{x} & M \ar@{->}[d]^{\bf x} \\
	U_q^-(\g) \ar@{->}[r]_\iota & M
}
\end{split}
\end{equation*}
where ${\bf x} \in \mf{U}$ is the image of $x \in \ms{B}_q(\g)$ through Lemma \ref{lem:realization of Bqg}, and the vertical arrows denote the actions of $x$ and $\bf x$, respectively.
\end{proof}

\subsection{Braid group symmetries on $q$-boson algebra}
Let $W$ be the Weyl group of $\g$, which is the subgroup of ${\rm GL}(\h^*)$ generated by the simple reflection $s_i$ given by $s_i(\lambda) = \lambda - \langle h_i,\, \lambda \rangle\, \alpha_i$ for $i \in I$ and $\lambda \in \h^*$.
Let $R(w)=\{\,(i_1,\dots,i_\ell) \,\,|\,\, i_j \in I \,\,\, \text{and} \,\, w=s_{i_1}\dots s_{i_\ell}\,\}$ be the set of reduced expressions of $w$, where $\ell$ is the length of $w$.

For $i \in I$, let $T_i$ be the $\bk$-algebra automorphism of $U_q({\mf g})$ in \cite{Lu10}, which satisfies
{\allowdisplaybreaks
\begin{align}
&
T_i(q^h) = q^{s_i(h)}, \label{eq:Ti(qh)} \\
& 
T_i(e_i) = -f_i k_i, \quad\,\,\,\,\,
T_i(e_j) = \sum_{r+s=-a_{ij}} (-1)^r q_i^{-r} e_i^{(s)} e_j e_i^{(r)}\quad (i\neq j), \label{eq:Ti(ej)} \\
&T_i(f_i) = -k_i^{-1}e_i, \quad 
T_i(f_j) = \sum_{r+s=-a_{ij}} (-1)^r q_i^{ r} f_i^{(r)} f_j f_i^{(s)}\quad (i\neq j), \label{eq:Ti(fj)}
\end{align}
}\!\!
for $j \in I$ and $h \in \sfP^\vee \subset \h$, where $s_i(h) = h - \langle h, \alpha_i \rangle h_i$.
Note that $T_i=T''_{i,1}$ in \cite{Lu10}.

\begin{prop}\label{prop:braid}
For $i\in I$, we have $T_i(\tilde{U}_q^+(\mf{g}))\subset \tilde{U}_q(\mf{g})$. \end{prop}
\begin{proof}
By \eqref{eq:Ti(qh)}--\eqref{eq:Ti(fj)}, we have
\begin{equation} \label{eq:Ti acts on bej'}
	T_i(\mb{e}_j') = 
	\begin{cases}
		\left(q_i-q_i^{-1}\right)q_i^2f_i & \text{if $i = j$,} \\
		\displaystyle (-1)^{a_{ij}}\left( q_i-q_i^{-1} \right)^{a_{ij}} \sum_{r+s=-a_{ij}}(-1)^rq_i^{a_{ij}(1-s)-r}q_j^{sa_{ji}}\mb{e}_i'^{(s)}\mb{e}_j'\mb{e}_i'^{(r)} 
		& \text{if $i \neq j$,}
	\end{cases}
\end{equation}
where the divided power of $\mathbf{e}_i'$ $(i \in I)$ is understood as usual, that is, $\mathbf{e}_i'^{(m)}=(\mathbf{e}_i')^m/[m]_i!$.
This completes the proof.
\end{proof}

\begin{cor}\label{cor:braid on q-bosons}
For $i\in I$, we have a homomorphism of $\bk$-algebras
\begin{equation*}
\text{\em \texttt T}_i : \ms{B}^+_q(\mf{g}) \longrightarrow \ms{B}_q(\mf{g}) 
\end{equation*}
such that
\begin{equation*}
\begin{split}
\text{\em \texttt T}_i(e'_i) &=(q_i-q_i^{-1})q_i^2f_i,\\
\text{\em \texttt T}_i(e'_j) &= (-1)^{a_{ij}}( q_i-q_i^{-1} )^{a_{ij}} \sum_{r+s=-a_{ij}}(-1)^rq_i^{a_{ij}(1-s)-r}q_j^{sa_{ji}}{e'_i}^{(s)}e'_j{e'_i}^{(r)} \quad (i \neq j).
\end{split}
\end{equation*}
\end{cor}
\begin{proof}
By Lemma \ref{lem:realization of Bqg} and Proposition \ref{prop:braid}, the map ${\texttt T}_i$ is the map on $\ms{B}_q^+(\g)$ induced from 
\begin{equation} \label{eq:composition Ti}
\begin{split}
\xymatrix{
	\ms{B}^+_q(\mf{g}) \ar@{->}[d]_{\cong} \ar@{->}[rr]^{{\texttt T}_i} & & \ms{B}_q(\mf{g}) \ar@{->}[d]^{\cong} \\
	\tilde{U}^+_q(\mf{g}) \ar@{->}[r]^{T_i} & \tilde{U}_q(\mf{g}) \ar@{->>}[r]^{\,\,\,\pi} & \mf{U}
}
\end{split}
\end{equation}
where the vertical map is the isomorphism of $\bk$-algebras (cf.~Proposition \ref{prop:properties of tUqg}(a)).
This implies our assertion.
\end{proof}

\begin{ex}{\rm
Suppose that $A$ is simply-laced, that is, $q_i=q$ and $a_{ij} \in \{2,0,-1\}$ for all $i,j \in I$. Then 
\begin{equation*}
\begin{split}
\texttt{T}_i(e'_i) &=(q-q^{-1})q^2f_i,\\
\texttt{T}_i(e'_j) &= 
\begin{cases}
 e'_j & \text{if $a_{ij}=0$},\\
 \dfrac{q^{-1}e'_ie'_j - q^{-2} e'_je'_i}{ -(q-q^{-1})} & \text{if $a_{ij}=-1$}.
\end{cases}
\end{split}
\end{equation*}
This is reminiscent of the one introduced by \cite{KKOP21} (up to suitable scalar multiplication).
}
\end{ex}

\begin{cor}\label{cor:braid rel}
 Let $w\in W$ be given with $(i_\ell, \dots, i_1) \in R(w)$. Let $u\in \ms{B}^+_q(\mf{g})$ be given such that $\text{\em \texttt T}_{i_k}\dots \text{\em \texttt T}_{i_1}(u)\in \ms{B}^+_q(\mf{g})$ for all $1\le k\le \ell$. Then $\text{\em \texttt T}_{i_\ell}\dots \text{\em \texttt T}_{i_1}(u)\in \ms{B}^+_q(\mf{g})$ is independent of the choice of a reduced expression of $w$, which we denote by $\text{\em \texttt T}_w(u)$.
\end{cor}
\begin{proof}
Let $\xymatrix@C=2em{\!\!\psi: \ms{B}^+_q(\mf{g}) \ar@{->}[r] & \tilde{U}^+_q(\mf{g})}\!\!$ be the isomorphism of $\bk$-algebras obtained from Proposition \ref{prop:properties of tUqg}, and let $\xymatrix@C=2em{\!\!\tilde{\pi}:  \tilde{U}_q(\mf{g}) \ar@{->}[r] & \ms{B}_q(\mf{g}) }\!\!$ be the composition of the natural projection from $\tilde{U}_q(\mf{g})$ onto $\mf{U}$ (cf.~\eqref{eq:mfU}) with the $\bk$-algebra isomorphism between $\mf{U}$ and $\ms{B}_q({\mf{g}})$ in Lemma \ref{eq:mfU}.
If $T_i(\psi(u))\in \tilde{U}^+_q(\mf{g})$, then 
\begin{equation}\label{eq:conjugation of T_i}
 \texttt{T}_i(u)=\psi^{-1}T_i\psi(u)
\end{equation}
since $\tilde{\pi}\vert_{\tilde{U}^+_q(\mf{g})}=\psi^{-1}$.
By our assumption, we have
\begin{equation} \label{eq:assumption}
T_{i_k}  \dots   T_{i_1}  \psi(u)\in  \tilde{U}^+_q(\mf{g})
\end{equation}
for  all $1\le k\le \ell$.
Hence it follows from \eqref{eq:conjugation of T_i} and \eqref{eq:assumption} that 
\begin{equation} \label{eq:exp of ttTw}
{\texttt T}_{i_\ell}\dots {\texttt T}_{i_1}(u) 
 =\psi^{-1} T_{i_\ell}  \dots   T_{i_1}  \psi(u).
\end{equation}
The right-hand side in \eqref{eq:exp of ttTw} does not depend on the choice of a reduced expression of $w$ \cite{Lu10}.
\end{proof}

\begin{rem} \label{rem:domain of Ti}
{\em 
For $i \in I$, we define
$$
	{}^*U_q^-(\g)[i] = T_i^{-1} \left( U_q^-(\g) \right) \cap U_q^-(\g).
$$
It is known in \cite{Lu10} (cf.~\cite{Sa94}) that there exists an orthogonal decomposition of $U_q^-(\g)$ given by 
$$
	U_q^-(\g) = {}^*U_q^-(\g)[i] \oplus U_q^-(\g)f_i
$$%
with respect to the nondegenerate symmetric bilinear from $(\,\,,\,)_{\rm K}$ in \cite[Proposition 3.4.4]{Kas91}.
Let $\tilde{U}_q(\g)[i]$ (resp.~$\ms{B}_q(\g)[i]$) be the $\bk$-subalgebra of $\tilde{U}_q(\g)$ (resp.~$\ms{B}_q(\g)$) generated by $\tilde{U}_q^+(\g)$ (resp.~$\ms{B}_q^+(\g)$) and ${}^*U_q^-(\g)[i]$.
Then we have $T_i \left( \tilde{U}_q(\g)[i] \right) \subset \tilde{U}_q(\g)$, so one may extend the domain of ${\texttt T}_i$ to be $\ms{B}_q(\g)[i]$.
}
\end{rem}

\section{Extended affine Weyl groups} \label{sec:extended affine Weyl groups}

\subsection{Affine root systems} \label{subsec:Affine root systems}
We follow the convention in \cite{Kac}.
Assume that the generalized Cartan matrix $\sfA$ is of affine type $X_n^{(1)}$, where its Dynkin diagram is as shown in Table \ref{tab:dynkins}.
\begin{table}[H]
	\resizebox{\columnwidth}{!}{%
		\begin{tabular}{|c||c||c|c|}
			\hline
			type & Dynkin diagram & type & Dynkin diagram \\
			\hline \hline
			$A_1^{(1)}$ & 
			\begin{minipage}[c][0.06\textheight]{0.5\linewidth}
			\setlength{\unitlength}{0.17in}
			\begin{picture}(16,3.7)
					\put(7,2.1){\makebox(0,0)[c]{$\circ$}}
					\put(9,2.1){\makebox(0,0)[c]{$\circ$}}
					\put(7.3,2.0){\line(1,0){1.3}}
					\put(7.3,2.2){\line(1,0){1.3}}
					\put(7.0,1.89){\Large $\prec$}
					\put(8.2,1.89){\Large $\succ$}
					\put(7,1.3){\makebox(0,0)[c]{\tiny ${\alpha}_0$}}
					\put(9,1.3){\makebox(0,0)[c]{\tiny ${\alpha}_1$}}
				\end{picture}
			\end{minipage} & & \\ 
			\hline
			$\underset{(n \ge 2)}{A_n^{(1)}}$ & 
			\begin{minipage}[c][0.09\textheight]{0.5\linewidth}
			\setlength{\unitlength}{0.17in}
				\begin{picture}(16,5.5)
					\put(3,2.1){\makebox(0,0)[c]{$\circ$}}
					\put(5.6,2.1){\makebox(0,0)[c]{$\circ$}}
					\put(10.5,2.1){\makebox(0,0)[c]{$\circ$}}
					\put(13,2.1){\makebox(0,0)[c]{$\circ$}}
					\put(8.1,4.1){\makebox(0,0)[c]{$\circ$}}
					\put(3.7,2.1){\line(1,0){1.3}}
					\put(6,2.1){\line(1,0){1.3}}
					\put(8.7,2.1){\line(1,0){1.3}}
					\put(11.2,2.1){\line(1,0){1.3}}
					\put(7.7,3.9){\line(-3,-1){4.3}}
					\put(8.5,3.9){\line(3,-1){4}}

					\put(8,2.05){\makebox(0,0)[c]{$\cdots$}}
					\put(8,3.2){\makebox(0,0)[c]{\tiny $\alpha_0$}}
					\put(3.2,1.1){\makebox(0,0)[c]{\tiny ${\alpha}_1$}}
					\put(5.6,1.1){\makebox(0,0)[c]{\tiny ${\alpha}_2$}}
					\put(10.4,1.1){\makebox(0,0)[c]{\tiny ${\alpha}_{n-1}$}}
					\put(13,1.1){\makebox(0,0)[c]{\tiny ${\alpha}_{n}$}}
				\end{picture}
			\end{minipage} & & \\ 
			\hline
			$\underset{(n \ge 3)}{B_n^{(1)}}$ & 
			\begin{minipage}[c][0.10\textheight]{0.5\linewidth}
			\setlength{\unitlength}{0.17in}
				\begin{picture}(16,3.7)
					\put(2.8,0.7){\makebox(0,0)[c]{$\circ$}}
					\put(2.8,3.3){\makebox(0,0)[c]{$\circ$}}
					\put(5.6,2.1){\makebox(0,0)[c]{$\circ$}}
					\put(10.5,2.1){\makebox(0,0)[c]{$\circ$}}
					\put(13,2.1){\makebox(0,0)[c]{$\circ$}}

					\put(5.3,1.8){\line(-2,-1){2.1}}
					\put(5.3,2.2){\line(-2,1){2.1}}
					\put(6,2.1){\line(1,0){1.3}}
					\put(8.7,2.1){\line(1,0){1.3}}
					\put(11.2,2){\line(1,0){1.3}}
					\put(11.2,2.2){\line(1,0){1.3}}
					\put(11.5,1.88){\Large $\succ$}

					\put(8,2.05){\makebox(0,0)[c]{$\cdots$}}
					\put(2.8,2.5){\makebox(0,0)[c]{\tiny ${\alpha}_0$}}
					\put(2.8,0){\makebox(0,0)[c]{\tiny ${\alpha}_1$}}
					\put(5.6,1.1){\makebox(0,0)[c]{\tiny ${\alpha}_2$}}
					\put(10.4,1.1){\makebox(0,0)[c]{\tiny ${\alpha}_{n-1}$}}
					\put(13,1.1){\makebox(0,0)[c]{\tiny ${\alpha}_{n}$}}
					%
					\put(5.6,2.8){\makebox(0,0)[c]{\tiny $2$}}
					\put(10.4,2.8){\makebox(0,0)[c]{\tiny $2$}}
					\put(13,2.8){\makebox(0,0)[c]{\tiny $2$}}
				\end{picture}
			\end{minipage} & 
			$\underset{(n \ge 3)}{A_{2n-1}^{(2)}}$ & 
			\begin{minipage}[c][0.10\textheight]{0.5\linewidth}
			\setlength{\unitlength}{0.17in}
				\begin{picture}(16,3.7)
					\put(2.8,0.7){\makebox(0,0)[c]{$\circ$}}
					\put(2.8,3.3){\makebox(0,0)[c]{$\circ$}}
					\put(5.6,2.1){\makebox(0,0)[c]{$\circ$}}
					\put(10.5,2.1){\makebox(0,0)[c]{$\circ$}}
					\put(13,2.1){\makebox(0,0)[c]{$\circ$}}

					\put(5.3,1.8){\line(-2,-1){2.1}}
					\put(5.3,2.2){\line(-2,1){2.1}}
					\put(6,2.1){\line(1,0){1.3}}
					\put(8.7,2.1){\line(1,0){1.3}}
					\put(11.2,2){\line(1,0){1.3}}
					\put(11.2,2.2){\line(1,0){1.3}}
					\put(11.5,1.88){\Large $\prec$}

					\put(8,2.05){\makebox(0,0)[c]{$\cdots$}}
					\put(2.8,2.5){\makebox(0,0)[c]{\tiny ${\alpha}_0$}}
					\put(2.8,0){\makebox(0,0)[c]{\tiny ${\alpha}_1$}}
					\put(5.6,1.1){\makebox(0,0)[c]{\tiny ${\alpha}_2$}}
					\put(10.4,1.1){\makebox(0,0)[c]{\tiny ${\alpha}_{n-1}$}}
					\put(13,1.1){\makebox(0,0)[c]{\tiny ${\alpha}_{n}$}}
					%
					\put(5.6,2.8){\makebox(0,0)[c]{\tiny $2$}}
					\put(10.4,2.8){\makebox(0,0)[c]{\tiny $2$}}
				\end{picture}
			\end{minipage}
			\\ 
			\hline
			$\underset{(n \ge 2)}{C_n^{(1)}}$ & 
			\begin{minipage}[c][0.06\textheight]{0.5\linewidth}
			\setlength{\unitlength}{0.17in}
			\begin{picture}(16,3.7)
					\put(3,2.1){\makebox(0,0)[c]{$\circ$}}
					\put(5.6,2.1){\makebox(0,0)[c]{$\circ$}}
					\put(10.5,2.1){\makebox(0,0)[c]{$\circ$}}
					\put(13,2.1){\makebox(0,0)[c]{$\circ$}}
					\put(3.7,2){\line(1,0){1.3}}
					\put(3.7,2.2){\line(1,0){1.3}}
					\put(6,2.1){\line(1,0){1.3}}
					\put(8.7,2.1){\line(1,0){1.3}}
					\put(11.2,2){\line(1,0){1.3}}
					\put(11.2,2.2){\line(1,0){1.3}}
					\put(11.5,1.88){\Large $\prec$}
					\put(4,1.88){\Large $\succ$}

					\put(8,2.05){\makebox(0,0)[c]{$\cdots$}}
					\put(3.2,1.1){\makebox(0,0)[c]{\tiny ${\alpha}_0$}}
					\put(5.6,1.1){\makebox(0,0)[c]{\tiny ${\alpha}_1$}}
					\put(10.4,1.1){\makebox(0,0)[c]{\tiny ${\alpha}_{n-1}$}}
					\put(13,1.1){\makebox(0,0)[c]{\tiny ${\alpha}_{n}$}}
					%
					\put(5.6,2.8){\makebox(0,0)[c]{\tiny $2$}}
					\put(10.4,2.8){\makebox(0,0)[c]{\tiny $2$}}
				\end{picture}
			\end{minipage} & $\underset{(n \ge 2)}{D_{n+1}^{(2)}}$ &
			\begin{minipage}[c][0.06\textheight]{0.5\linewidth}
			\setlength{\unitlength}{0.17in}
			\begin{picture}(16,3.7)
					\put(3,2.1){\makebox(0,0)[c]{$\circ$}}
					\put(5.6,2.1){\makebox(0,0)[c]{$\circ$}}
					\put(10.5,2.1){\makebox(0,0)[c]{$\circ$}}
					\put(13,2.1){\makebox(0,0)[c]{$\circ$}}
					\put(3.7,2){\line(1,0){1.3}}
					\put(3.7,2.2){\line(1,0){1.3}}
					\put(6,2.1){\line(1,0){1.3}}
					\put(8.7,2.1){\line(1,0){1.3}}
					\put(11.2,2){\line(1,0){1.3}}
					\put(11.2,2.2){\line(1,0){1.3}}
					\put(11.5,1.88){\Large $\succ$}
					\put(4,1.88){\Large $\prec$}

					\put(8,2.05){\makebox(0,0)[c]{$\cdots$}}
					\put(3.2,1.1){\makebox(0,0)[c]{\tiny ${\alpha}_0$}}
					\put(5.6,1.1){\makebox(0,0)[c]{\tiny ${\alpha}_1$}}
					\put(10.4,1.1){\makebox(0,0)[c]{\tiny ${\alpha}_{n-1}$}}
					\put(13,1.1){\makebox(0,0)[c]{\tiny ${\alpha}_{n}$}}
					%
				\end{picture}
			\end{minipage} \\ 
			\hline
			$\underset{(n \ge 4)}{D_n^{(1)}}$ &
			\begin{minipage}[c][0.11\textheight]{0.5\linewidth}
				\setlength{\unitlength}{0.17in}
				\begin{picture}(16,3.7)
					\put(2.8,0.7){\makebox(0,0)[c]{$\circ$}}
					\put(2.8,3.3){\makebox(0,0)[c]{$\circ$}}
					\put(5.6,2){\makebox(0,0)[c]{$\circ$}}
					\put(10.4,2){\makebox(0,0)[c]{$\circ$}}
					\put(13.1,3.3){\makebox(0,0)[c]{$\circ$}}
					\put(13.1,0.7){\makebox(0,0)[c]{$\circ$}}
					\put(5.3,1.8){\line(-2,-1){2.1}}
					\put(5.3,2.2){\line(-2,1){2.1}}
					\put(6,2){\line(1,0){1.3}}
					\put(8.7,2){\line(1,0){1.3}}
					\put(10.7,2.2){\line(2,1){2}}
					\put(10.7,1.8){\line(2,-1){2}}

					\put(8,1.95){\makebox(0,0)[c]{$\cdots$}}
					\put(2.8,2.5){\makebox(0,0)[c]{\tiny ${\alpha}_0$}}
					\put(2.8,0.0){\makebox(0,0)[c]{\tiny ${\alpha}_1$}}
					\put(5.6,1){\makebox(0,0)[c]{\tiny ${\alpha}_2$}}
					\put(10.4,1){\makebox(0,0)[c]{\tiny ${\alpha}_{n-2}$}}
					\put(13.1,2.5){\makebox(0,0)[c]{\tiny ${\alpha}_{n-1}$}}
					\put(13.1,0.0){\makebox(0,0)[c]{\tiny ${\alpha}_{n}$}}
					%
					\put(5.6,2.7){\makebox(0,0)[c]{\tiny $2$}}
					\put(10.4,2.7){\makebox(0,0)[c]{\tiny $2$}}
				\end{picture}
			\end{minipage} & &  \\ 
			\hline
			$E_6^{(1)}$ & 
			\begin{minipage}[c][0.15\textheight]{0.5\linewidth}
			\setlength{\unitlength}{0.17in}
				\begin{picture}(16,8.5)
					\put(3,2.1){\makebox(0,0)[c]{$\circ$}}
					\put(5.6,2.1){\makebox(0,0)[c]{$\circ$}}
					\put(8,2.1){\makebox(0,0)[c]{$\circ$}}
					\put(10.5,2.1){\makebox(0,0)[c]{$\circ$}}
					\put(13, 2.1){\makebox(0,0)[c]{$\circ$}}
					\put(8,4.5){\makebox(0,0)[c]{$\circ$}}
					\put(8,7){\makebox(0,0)[c]{$\circ$}}

					\put(8, 4){\line(0, -3){1.3}}
					\put(8, 6.5){\line(0, -3){1.3}}
					\put(3.7,2.1){\line(1,0){1.3}}
					\put(6,2.1){\line(1,0){1.3}}
					\put(8.7,2.1){\line(1,0){1.3}}
					\put(11.2,2.1){\line(1,0){1.3}}
					\put(8.7,7){\makebox(0,0)[c]{\tiny $\alpha_0$}}
					\put(3.2,1.1){\makebox(0,0)[c]{\tiny ${\alpha}_1$}}
					\put(8.7,4.5){\makebox(0,0)[c]{\tiny $\alpha_2$}}
					\put(5.6,1.1){\makebox(0,0)[c]{\tiny ${\alpha}_3$}}
					\put(8,1.1){\makebox(0,0)[c]{\tiny $\alpha_4$}}
					\put(10.4,1.1){\makebox(0,0)[c]{\tiny ${\alpha}_{5}$}}
					\put(13,1.1){\makebox(0,0)[c]{\tiny ${\alpha}_{6}$}}
					%
					\put(7.4,4.5){\makebox(0,0)[c]{\tiny $2$}}
					\put(5.6,2.7){\makebox(0,0)[c]{\tiny $2$}}
					\put(8.5,2.7){\makebox(0,0)[c]{\tiny $3$}}
					\put(10.4,2.7){\makebox(0,0)[c]{\tiny $2$}}
				\end{picture}
			\end{minipage} & & \\ 
			\hline
			$E_7^{(1)}$ & 
			\begin{minipage}[c][0.10\textheight]{0.5\linewidth}
			\setlength{\unitlength}{0.17in}
				\begin{picture}(16,5.5)
					\put(0.7,2.1){\makebox(0,0)[c]{$\circ$}}
					\put(3,2.1){\makebox(0,0)[c]{$\circ$}}
					\put(5.6,2.1){\makebox(0,0)[c]{$\circ$}}
					\put(8,2.1){\makebox(0,0)[c]{$\circ$}}
					\put(10.5,2.1){\makebox(0,0)[c]{$\circ$}}
					\put(13, 2.1){\makebox(0,0)[c]{$\circ$}}
					\put(15.5,2.1){\makebox(0,0)[c]{$\circ$}}
					\put(8,4.5){\makebox(0,0)[c]{$\circ$}}
					\put(8, 4){\line(0, -3){1.3}}
					\put(1.2,2.1){\line(1,0){1.3}}
					\put(3.7,2.1){\line(1,0){1.3}}
					\put(6,2.1){\line(1,0){1.3}}
					\put(8.7,2.1){\line(1,0){1.3}}
					\put(11.2,2.1){\line(1,0){1.3}}
					\put(13.7,2.1){\line(1,0){1.3}}
					\put(0.8,1.1){\makebox(0,0)[c]{\tiny ${\alpha}_0$}}
					\put(3.2,1.1){\makebox(0,0)[c]{\tiny ${\alpha}_1$}}
					\put(8.7,4.5){\makebox(0,0)[c]{\tiny $\alpha_2$}}
					\put(5.6,1.1){\makebox(0,0)[c]{\tiny ${\alpha}_3$}}
					\put(8,1.1){\makebox(0,0)[c]{\tiny $\alpha_4$}}
					\put(10.4,1.1){\makebox(0,0)[c]{\tiny ${\alpha}_{5}$}}
					\put(13,1.1){\makebox(0,0)[c]{\tiny ${\alpha}_{6}$}}
					\put(15.5, 1.1){\makebox(0,0)[c]{\tiny ${\alpha}_{7}$}}
					%
					\put(3.2,2.7){\makebox(0,0)[c]{\tiny $2$}}
					\put(7.5,4.5){\makebox(0,0)[c]{\tiny $2$}}
					\put(5.6,2.7){\makebox(0,0)[c]{\tiny $3$}}
					\put(8.5,2.7){\makebox(0,0)[c]{\tiny $4$}}
					\put(10.4,2.7){\makebox(0,0)[c]{\tiny $3$}}
					\put(13,2.7){\makebox(0,0)[c]{\tiny $2$}}
				\end{picture}
			\end{minipage} & & \\ 
			\hline
			$E_8^{(1)}$ & 
			\begin{minipage}[c][0.10\textheight]{0.5\linewidth}
			\setlength{\unitlength}{0.15in}
				\begin{picture}(16,5.5)
					\put(0.7,2.1){\makebox(0,0)[c]{$\circ$}}
					\put(3,2.1){\makebox(0,0)[c]{$\circ$}}
					\put(5.6,2.1){\makebox(0,0)[c]{$\circ$}}
					\put(8,2.1){\makebox(0,0)[c]{$\circ$}}
					\put(10.5,2.1){\makebox(0,0)[c]{$\circ$}}
					\put(13, 2.1){\makebox(0,0)[c]{$\circ$}}
					\put(15.5,2.1){\makebox(0,0)[c]{$\circ$}}
					\put(18,2.1){\makebox(0,0)[c]{$\circ$}}
					\put(5.6,4.5){\makebox(0,0)[c]{$\circ$}}
					\put(5.6,4){\line(0, -3){1.3}}
					\put(1.2,2.1){\line(1,0){1.3}}
					\put(3.7,2.1){\line(1,0){1.3}}
					\put(6,2.1){\line(1,0){1.3}}
					\put(8.7,2.1){\line(1,0){1.3}}
					\put(11.2,2.1){\line(1,0){1.3}}
					\put(13.7,2.1){\line(1,0){1.3}}
					\put(16.2,2.1){\line(1,0){1.3}}
					\put(0.8,1.1){\makebox(0,0)[c]{\tiny ${\alpha}_1$}}
					\put(3.2,1.1){\makebox(0,0)[c]{\tiny ${\alpha}_3$}}
					\put(6.5,4.5){\makebox(0,0)[c]{\tiny $\alpha_2$}}
					\put(5.6,1.1){\makebox(0,0)[c]{\tiny ${\alpha}_4$}}
					\put(8,1.1){\makebox(0,0)[c]{\tiny $\alpha_5$}}
					\put(10.4,1.1){\makebox(0,0)[c]{\tiny ${\alpha}_6$}}
					\put(13,1.1){\makebox(0,0)[c]{\tiny ${\alpha}_7$}}
					\put(15.5, 1.1){\makebox(0,0)[c]{\tiny ${\alpha}_8$}}
					\put(18, 1.1){\makebox(0,0)[c]{\tiny ${\alpha}_0$}}
					%
					\put(0.8,2.7){\makebox(0,0)[c]{\tiny $2$}}
					\put(3.2,2.7){\makebox(0,0)[c]{\tiny $4$}}
					\put(5,4.5){\makebox(0,0)[c]{\tiny $3$}}
					\put(6.2,2.7){\makebox(0,0)[c]{\tiny $6$}}
					\put(8,2.7){\makebox(0,0)[c]{\tiny $5$}}
					\put(10.4,2.7){\makebox(0,0)[c]{\tiny $4$}}
					\put(13,2.7){\makebox(0,0)[c]{\tiny $3$}}
					\put(15.5,2.7){\makebox(0,0)[c]{\tiny $2$}}
				\end{picture}
			\end{minipage} 
			& & \\ 
			\hline
			$F_4^{(1)}$ & 
			\begin{minipage}[c][0.06\textheight]{0.5\linewidth}
			\setlength{\unitlength}{0.17in}
			\begin{picture}(16,3.9)
					\put(4,2.1){\makebox(0,0)[c]{$\circ$}}
					\put(6,2.1){\makebox(0,0)[c]{$\circ$}}
					\put(8,2.1){\makebox(0,0)[c]{$\circ$}}
					\put(10,2.1){\makebox(0,0)[c]{$\circ$}}
					\put(12,2.1){\makebox(0,0)[c]{$\circ$}}
					\put(4.3,2.1){\line(1,0){1.3}}
					\put(6.3,2.1){\line(1,0){1.3}}
					\put(8.3,2.0){\line(1,0){1.3}}
					\put(8.3,2.2){\line(1,0){1.3}}
					\put(10.3,2.1){\line(1,0){1.3}}
					\put(8.6,1.88){\Large $\succ$}
					\put(4,1.3){\makebox(0,0)[c]{\tiny ${\alpha}_0$}}
					\put(6,1.3){\makebox(0,0)[c]{\tiny ${\alpha}_1$}}
					\put(8,1.3){\makebox(0,0)[c]{\tiny ${\alpha}_2$}}
					\put(10,1.3){\makebox(0,0)[c]{\tiny ${\alpha}_3$}}
					\put(12,1.3){\makebox(0,0)[c]{\tiny ${\alpha}_4$}}
					%
					\put(6,2.7){\makebox(0,0)[c]{\tiny $2$}}
					\put(8,2.7){\makebox(0,0)[c]{\tiny $3$}}
					\put(10,2.7){\makebox(0,0)[c]{\tiny $4$}}
					\put(12,2.7){\makebox(0,0)[c]{\tiny $2$}}
				\end{picture}
			\end{minipage} & $E_6^{(2)}$ &
			\begin{minipage}[c][0.06\textheight]{0.5\linewidth}
			\setlength{\unitlength}{0.17in}
			\begin{picture}(16,3.9)
					\put(4,2.1){\makebox(0,0)[c]{$\circ$}}
					\put(6,2.1){\makebox(0,0)[c]{$\circ$}}
					\put(8,2.1){\makebox(0,0)[c]{$\circ$}}
					\put(10,2.1){\makebox(0,0)[c]{$\circ$}}
					\put(12,2.1){\makebox(0,0)[c]{$\circ$}}
					\put(4.3,2.1){\line(1,0){1.3}}
					\put(6.3,2.1){\line(1,0){1.3}}
					\put(8.3,2.0){\line(1,0){1.3}}
					\put(8.3,2.2){\line(1,0){1.3}}
					\put(10.3,2.1){\line(1,0){1.3}}
					\put(8.6,1.88){\Large $\prec$}
					\put(4,1.3){\makebox(0,0)[c]{\tiny ${\alpha}_0$}}
					\put(6,1.3){\makebox(0,0)[c]{\tiny ${\alpha}_1$}}
					\put(8,1.3){\makebox(0,0)[c]{\tiny ${\alpha}_2$}}
					\put(10,1.3){\makebox(0,0)[c]{\tiny ${\alpha}_3$}}
					\put(12,1.3){\makebox(0,0)[c]{\tiny ${\alpha}_4$}}
					%
					\put(6,2.7){\makebox(0,0)[c]{\tiny $2$}}
					\put(8,2.7){\makebox(0,0)[c]{\tiny $3$}}
					\put(10,2.7){\makebox(0,0)[c]{\tiny $2$}}
					\put(12,2.7){\makebox(0,0)[c]{\tiny $1$}}
				\end{picture}
			\end{minipage}
			 \\ 
			\hline
			$G_2^{(1)}$ & 
			\begin{minipage}[c][0.06\textheight]{0.5\linewidth}
			\setlength{\unitlength}{0.17in}
			\begin{picture}(16,3.7)
					\put(6,2.1){\makebox(0,0)[c]{$\circ$}}
					\put(8,2.1){\makebox(0,0)[c]{$\circ$}}
					\put(10,2.1){\makebox(0,0)[c]{$\circ$}}
					\put(6.3,2.1){\line(1,0){1.3}}
					\put(8.3,1.9){\line(1,0){1.3}}
					\put(8.3,2.1){\line(1,0){1.3}}
					\put(8.3,2.3){\line(1,0){1.3}}
					\put(8.5,1.84){\LARGE $\succ$}
					\put(6,1.3){\makebox(0,0)[c]{\tiny ${\alpha}_0$}}
					\put(8,1.3){\makebox(0,0)[c]{\tiny ${\alpha}_1$}}
					\put(10,1.3){\makebox(0,0)[c]{\tiny ${\alpha}_2$}}
					%
					\put(8,2.7){\makebox(0,0)[c]{\tiny $2$}}
					\put(10,2.7){\makebox(0,0)[c]{\tiny $3$}}
				\end{picture}
			\end{minipage} & $D_4^{(3)}$ & 
			\begin{minipage}[c][0.06\textheight]{0.5\linewidth}
			\setlength{\unitlength}{0.17in}
			\begin{picture}(16,3.7)
					\put(6,2.1){\makebox(0,0)[c]{$\circ$}}
					\put(8,2.1){\makebox(0,0)[c]{$\circ$}}
					\put(10,2.1){\makebox(0,0)[c]{$\circ$}}
					\put(6.3,2.1){\line(1,0){1.3}}
					\put(8.3,1.9){\line(1,0){1.3}}
					\put(8.3,2.1){\line(1,0){1.3}}
					\put(8.3,2.3){\line(1,0){1.3}}
					\put(8.5,1.84){\LARGE $\prec$}
					\put(6,1.3){\makebox(0,0)[c]{\tiny ${\alpha}_0$}}
					\put(8,1.3){\makebox(0,0)[c]{\tiny ${\alpha}_1$}}
					\put(10,1.3){\makebox(0,0)[c]{\tiny ${\alpha}_2$}}
					%
					\put(8,2.7){\makebox(0,0)[c]{\tiny $2$}}
				\end{picture}
			\end{minipage}
			\\ 
			\hline
		\end{tabular}
	}
	\caption{}
	\label{tab:dynkins}
\end{table}
Let us take the index set by $I = \left\{\,0,\,1,\,\dots,\,n\,\right\}$ so that $\mathring{\sfA} = \left(a_{ij}\right)_{i,j \in I \setminus \{ 0 \}}$ is of finite type. Put $\I = I \setminus \{ 0 \}$ for simplicity.

Let
\begin{equation*}
	\sfQ=\bigoplus_{i \in I}\,\Z\alpha_i \text{\,\, and \,\,}
	\sfQ^\vee=\bigoplus_{i \in I}\,\Z h_i,
\end{equation*}
which are the root and coroot lattices for $\g$, respectively,
and set $\sfQ_+=\bigoplus_{i \in I}\Z_+\alpha_i$ and $\sfQ_- = -\sfQ_+$.

Let $\delta \in \h^*$ be an imaginary null root and $K \in \h$ a central element given by
\begin{equation} \label{eq:del and K}
	\delta = \sum_{i \in I} \,a_i \alpha_i \text{\,\, and \,\,}
	K = r_\g \sum_{i \in I} \,a_i^\vee h_i,
\end{equation}
where $a_i$ (resp. $a_i^\vee$) is the (resp.~dual) Kac label for $\sfA$ (resp.~${}^t\sfA$), and $r_\g$ is given by $$
	r_\g = 
	\begin{cases}
		1 & \text{if $X = A, D, E$,} \\
		2 & \text{if $X = B, C, F$,} \\
		3 & \text{if $X = G$.}
	\end{cases}
$$
The numerical labels in Table \ref{tab:dynkins} are the Kac labels for $\sfA$ and $^t\sfA$, where we omit those equal to $1$ (cf.~Remark \ref{rem:compared with Kac}).

Set $\theta = \delta-a_0\alpha_0$ and $\theta^\vee = a_0^{-1} (r_\g^{-1} K - h_0)$.
We denote by $d$ the scaling element.
Let $\{\, \Lambda_i \in \h^* \,\mid\, i \in I \,\}$ (resp.~$\{\, \Lambda_i^\vee \in \h \,\mid\, i \in I \,\}$) be the set of the fundamental weights (resp.~coweights) of $\g$, where $\langle \Lambda_i, h_j \rangle  = \delta_{ij}$, $\langle \alpha_i, \Lambda_j^\vee \rangle = \delta_{ij}$ $(i, j \in I)$, and $d = \Lambda_0^\vee$.
Put
\begin{equation*}
	\sfP = \bigoplus_{i \in I} \Z \Lambda_i + \mathbb{C} \delta \text{\,\, and \,\,}
	\sfP^\vee = \bigoplus_{i \in I} \Z \Lambda_i^\vee + \mathbb{C} K.
\end{equation*}

Let us take a nondegenerate symmetric bilinear $\C$-valued from $(\,\cdot\,,\,\cdot\,)$ on $\h$ following \cite[(6.2.1)]{Kac} with respect to $\sfD$ (see Remark \ref{rem:compared with Kac}), and then $\nu : \h \rightarrow \hd$ is the $\C$-linear isomorphism from $\h$ to $\hd$ induced from $(\,\cdot\,,\,\cdot\,)$, that is, $\langle \nu(h), h' \rangle = (h,h')$ for $h, h' \in \h$. We still denote by $(\,\cdot\,,\,\cdot\,)$ the $\C$-valued bilinear form on $\h^*$ induced from $\nu : \h \rightarrow \h^*$.

Let $\cg$ denote the underlying finite-dimensional simple Lie subalgebra of $\g$ corresponding to $\mr{\sfA}$, and let $\ch$ be its Cartan subalgebra, which is defined by the $\C$-linear span of $\alpha_1, \dots, \alpha_n$.
Let $\mr{\sfQ}=\bigoplus_{i \in \I}\Z\alpha_i$ be the root lattice of $\cg$. Let $\{\,\varpi_i\,|\,i \in \I\,\}$ be the set of the fundamental weights, and let $\mr{\sfP}=\bigoplus_{i \in \I}\Z\varpi_i$ be the weight lattice of $\mr{\g}$.
Note that $\theta$ is equal to the maximal root of $\cg$ by regarding $\mr{\sfP}$ as a sublattice of $\sfP / \Z\delta$.
Similarly, we denote by $\mr{\sfQ}^\vee = \bigoplus_{i \in \I}\Z h_i$ (resp.~$\mr{\sfP}^\vee = \bigoplus_{i \in \I}\Z\varpi_i^\vee$) the coroot (resp.~coweight) lattice of $\cg$.
Let $\mr{\sfDel}^+$ (resp.~$\mr{\sfDel}^-$) be the set of positive (resp.~negative) roots of $\cg$.

Let $\ov{\lambda} \in \ch^*$ denote the orthogonal projection of $\lambda \in \h^*$ onto $\ch^* \subset \h^*$ with respect to $(\,\cdot\,,\,\cdot\,)$, where we also use the same notations for $\h$ and $\ch$.
For $i \in \I$, we have
\begin{equation*}
	\Lambda_i = \ov{\Lambda}_i + a_i^\vee \Lambda_0, \qquad
	\Lambda_i^\vee = \ov{\Lambda_i^\vee} + a_i d 
\end{equation*}
where $\ov{\Lambda_0} = \ov{\Lambda_0^\vee} = 0$. 
Note that $\ov{\Lambda}_i$ and $\ov{\Lambda_i^\vee}$ coincide with $\varpi_i$ and $\varpi_i^\vee$ as elements of $\mr{\h}^*$ and $\mr{\h}$, respectively.

\subsection{Extended affine Weyl groups}
Let $W$ (resp.~$\cW$) be the Weyl group of $\g$ (resp.~$\cg$), where $\cW$ is naturally regarded as a subgroup of $W$.

For $\beta \in \hd$, the endomorphism $t_\beta$ of $\hd$ is defined by
\begin{equation} \label{eq:translation}
	t_\beta (\lambda)
	=
	\lambda + \langle \lambda,\, K \rangle \,\beta - \left( (\lambda, \beta) + \frac{1}{2}(\beta,\beta) \langle \lambda, K \rangle \right) \delta,
\end{equation}
which is called the {\it translation} by $\beta$ \cite[(6.5.2)]{Kac}.
The family of translations satisfies the additive property and it acts on $\sfDel = \sfDel^+ \cup \sfDel^-$.

Let
\begin{equation} \label{eq:Z-lattices}
	\sfM = \nu\left( \Z(\cW \cdot \theta^\vee) \right)
	\, \text{ and } \,\,\,
	\eM = \left\{ \left. \, \beta \in \chd \, \right| \, (\beta,\alpha_i) \in \mathbb{Z} \,\, \text{\,for} \,\, i \in \I \, \right\}
\end{equation}
be the $\Z$-lattices of $\mr{\h}^*$.
The next result is well-known, but we provide its proof for self-containedness.

\begin{prop} \label{prop:Z-lattices for T}
We have
\begin{equation*}
	\sfM = \nu\left(\mr{\sfQ}^\vee\right) \supseteq \mr{\sfQ} \text{\, and \,}
	\eM = \nu\left(\mr{\sfP}^\vee\right) \supseteq \mr{\sfP}.
\end{equation*}
In particular, $\nu(\varpi_i^\vee) = d_i^{-1} \varpi_i$ for $i \in \I$, which forms the $\Z$-basis of $\eM$.
\end{prop}
\begin{proof}
First, the proof of $\sfM = \nu\left(\mr{\sfQ}^\vee\right)$ can be found in \cite[\textsection 6.5]{Kac}.
Note that $\nu(h_i) = d_i^{-1} \alpha_i$ with $d_i^{-1} \in \Q$ \cite[(6.5.2)]{Kac}, so we have $\sfM \supseteq \mr{\sfQ}$.

Second, let us verify $\eM = \nu\left(\mr{\sfP}^\vee\right)$.
Take $\beta \in \eM$. For $i \in \I$, we have
\begin{equation*}
	\left\langle \alpha_i,\, \nu^{-1}(\beta) \right\rangle 
	= (h_i, \nu^{-1}(\beta))d_i
	= (\alpha_i, \beta) \in \Z.
\end{equation*}
Thus $\nu^{-1}(\beta) \in \mr{\sfP}^\vee$, equivalently, $\beta \in \nu\left(\mr{\sfP}^\vee\right)$.
Conversely, let $\beta \in \nu\left(\mr{\sfP}^\vee\right)$. 
For $j \in \I$, we have
\begin{equation*}
	\left( \alpha_j,\, \beta \right) 
	= \left( d_j h_j,\, \nu^{-1}(\beta) \right) 
	= \langle \alpha_j,\, \nu^{-1}(\beta) \rangle \in \Z,
\end{equation*}
so $\beta \in \eM$. 

Finally, we have $d_i\nu(\varpi_i^\vee) = \varpi_i$, since
$
	\langle \nu(\varpi_i^\vee), h_j \rangle 
	= (\varpi_i^\vee, h_j) 
	= \langle \alpha_j, \varpi_i^\vee \rangle d_j^{-1} = \delta_{ji}d_i^{-1}.
$
Hence $\eM \supseteq \mr{\sfP}$. We complete the proof.
\end{proof}
Let $T_{\sfM}$ $\left(\text{resp. $T_{\eM}$}\right)$ be the group of translations consisting of $t_\beta$ \eqref{eq:translation} for $\beta \in \sfM$ (resp.~$\beta \in \eM$).
Note that $\eM$ is defined so that $T_{\eM}$ acts on $\sfDel$, and 
$t_{w(\beta)} = w t_\beta w^{-1}$ for $\beta \in \eM$ and  $w \in \cW$
(cf.~\cite[(6.5.7)]{Kac}).

\begin{rem} \label{rem:compared with Kac}
{\em 
In this paper, we take $\sfD = {\rm diag}(d_i)_{i\in I}$ with $d_i = r_\g a_i^{-1}a_i^\vee$ for $i \in I$, while $d_i = a_i^{-1}a_i^\vee$ in \cite[\text section 6]{Kac}. 
But the formula \eqref{eq:translation} coincides essentially with \cite[(6.5.2)]{Kac}. 
More precisely, let $(\,\cdot\,|\,\cdot\,)$ be the normalized invariant bilinear form in \cite{Kac}, and let $\nu' : \h \rightarrow \h^*$ be the $\mathbb{C}$-linear isomorphism induced from $(\,\cdot\,|\,\cdot\,)$. Set $\eM' = \nu'(\mr{\sfP}^\vee)$. 
For $\beta \in \eM'$, we denote by $t_\beta'$ the translation by $\beta$ as in \eqref{eq:translation} by replacing $(\,\cdot\,,\,\cdot\,)$ and $K$ with $(\,\cdot\,|\,\cdot\,)$ and $K' := r_\g^{-1} K$, respectively (i.e.~as in \cite[(6.5.2)]{Kac}).
	Since $(\,\cdot\,|\,\cdot\,) = r_\g (\,\cdot\,,\,\cdot\,)$ on $\h$ (so, $\nu = r_\g^{-1} \nu'$) and $(\lambda, \nu(\varpi_i^\vee)) = \frac{2(\lambda,\varpi_i)}{(\varpi_i, \varpi_i)} = (\lambda | \nu'(\varpi_i^\vee))$ for $i \in \I$ and $\lambda \in \h^*$, we have $t_{\nu(\varpi_i^\vee)} = t_{\nu'(\varpi_i^\vee)}'$.
	This implies $T_{\eM} = T_{\eM'}$ by Proposition \ref{prop:Z-lattices for T}.
}
\end{rem}

We define
\begin{equation*}
	\eW = \mathring{W} \ltimes T_{\eM},
\end{equation*}
which is called the {\it extended affine Weyl group} of $\g$.
Since it is known that $W$ is isomorphic to the semidirect product $\cW \ltimes T_{\sfM}$ \cite[Proposition 6.5]{Kac},
we have $W \subset \eW$.

We denote by $\mr{\h}_{\mathbb{R}}^*$ the $\mathbb{R}$-linear span of $\alpha_1, \dots, \alpha_n$.
Let
\begin{equation*}
	C_{\rm af} = \left\{ \left. \, \lambda \in \mr{\h}_{\mathbb{R}}^* \, \right| \, (\lambda, \alpha_i) \ge 0 \text{ for $i \in \I$, and } (\lambda, \theta) \le 1 \,  \right\}
\end{equation*}
which is called the {\it affine Weyl chamber} (also called {\it fundamental alcove}).
Let $\mc{T}$ be the subgroup of $\eW$, which stabilizes $C_{\rm af}$.

\begin{prop}[\!\!\!\cite{Bour02}] \label{prop:Weyl}
The subgroup $\mc{T}$ is isomorphic to $\eW /W \simeq T_{\eM} / T_{\sfM}$, and $\eW$ is isomorphic to $W \rtimes \mc{T}$, where each element of $\mc{T}$ is understood as an affine Dynkin diagram automorphism.
\end{prop}

By Proposition \ref{prop:Weyl}, for $\hat{w} \in \eW$,
\begin{equation*}
	\hat{w} = w \tau,
\end{equation*}
where $w \in W$ and $\tau \in \mc{T}$. The length of $\hat{w} \in \eW$ is defined by the length $\ell(w)$ of $w$, which is also denoted by $\ell(\hat{w})$.
Note that $\mc{T} = \{ \, \hat{w} \in \eW \, | \, \ell(\hat{w}) = 0 \, \}$.
For $\hat{w} \in \eW$, we define
$$\sfDel^+(\hat{w})=\sfDel^+\cap \left(\hat{w}\sfDel^-\right) \subset \sfDel^+.$$
Then it is well-known that $\ell(\hat{w})  = |\sfDel^+(\hat{w})|$.

\subsection{Translations $t_r$} \label{subsec:translations}
For $r \in \I$, let us collect some known properties of $t_{-\nu\left(\overline{\Lambda_r^\vee}\right)}$ following \cite{Bour02, Wakimoto,Kac,Hump} (cf.~\cite{Ko,Che}).
We write $\lar = \nu\left( \overline{\Lambda_r^\vee} \right)$ and $\tr = t_{\lar}$ for simplicity.
By Proposition \ref{prop:Weyl}, we have
\begin{equation} \label{eq:translation by fundamentals}
\tr^{-1} = w_r \tau_r
\end{equation}
for some $w_r\in W$ and $\tau_r\in\mc{T}$.
For simplicity, we often omit the subscript $r$ in $\tau_r$ throughout this paper if there is no confusion (cf.~Example \ref{ex:rex of tr in rank 2}).

The following can be found in several literatures, but we also provide its proof for self-containdness.

\begin{prop} \label{prop:Delta+ by tr}
We have
\begin{equation*}
	\sfDel^+\left(w_r\right) = \sfDel^+\left(\tr^{-1}\right) = 
	\left\{ \alpha + k \delta \, \left| \,\, \alpha \in \mr{\sfDel}^+,\, 0 \le k < (\lambda_r,\,\alpha)  \right. \right\}.
\end{equation*}
\end{prop}
\begin{proof}
	Since $\tau\left(\sfDel^{\pm}\right) = \sfDel^{\pm}$ for $\tau \in \mc{T}$, we have $\sfDel^+\left(\tr^{-1}\right) = \sfDel^+(w_r)$.
	Let $\beta \in \sfDel^+\left(\tr^{-1}\right)$ be given.
	Then there exists $\gamma \in \sfDel^-$ such that $\beta = \tr^{-1}\left(\gamma\right) \in \sfDel^+$.
	It is well-known (e.g.~\cite[Proposition 6.3]{Kac}) that $\gamma$ is in $\mr{\sfDel}^-$ or given as $\alpha + \texttt{k}\delta$ for some $\alpha \in \mr{\sfDel}$ and $\texttt{k} < 0$.
	Since $\langle \gamma, K \rangle = 0$ and $\tr^{-1} = t_{-\lar}$, we have by \eqref{eq:translation}
	\begin{equation*}
		\beta = \tr^{-1}(\gamma) = \gamma + \left( \gamma, \lar \right)\delta \in \sfDel^+
	\end{equation*}
	If $\gamma \in \mr{\sfDel}^-$, then $\beta \in \sfDel^-$, which contradicts $\beta \in \sfDel^+$. So we assume that $\gamma = \alpha + \texttt{k}\delta$ for some $\alpha \in \mr{\sfDel}$ and $\texttt{k} < 0$. Then
	\begin{equation*}
		\beta = (\alpha+\texttt{k}\delta) + \left( \alpha+\texttt{k}\delta, \lar \right)\delta = \alpha + \left(\texttt{k} + (\alpha, \lar)\right) \delta.
	\end{equation*}
	Put $k = \texttt{k} + (\alpha, \lar)$.
	Since $\beta \in \sfDel^+$ and $\texttt{k} < 0$, we have $\alpha \in \mr{\sfDel}^+$ and $0 \le k < (\alpha, \lar)$.
	Conversely, let $\beta = \alpha + k\delta$ for $\alpha \in \mr{\sfDel}^+$ and $0 \le k < (\lar, \alpha)$. Then we have
	\begin{equation*}
		\tr(\beta) = \tr(\alpha) + k\delta = \alpha + \left(k-(\alpha,\lar)\right)\delta \in \sfDel^-,
	\end{equation*}
	which implies $\beta \in \sfDel^+ \cap \tr^{-1}(\sfDel^-) = \sfDel^+(\tr^{-1})$.
	This completes the proof.
\end{proof}

The following formula of $\ell(\tr^{-1})$ is obtained directly from Proposition \ref{prop:Delta+ by tr}, which is well-known in \cite[6.10 in \S 6.8]{Kac} (see also \cite{Pap95}).

\begin{cor} \label{cor:length of tr-1}
	The length of $\tr^{-1}$ is given by
	\begin{equation*}
		\ell(\tr^{-1}) = \sum_{\alpha \in \mr{\sfDel}^+} (\lar, \alpha).
	\end{equation*}
\end{cor}

\begin{prop} \label{prop:length of tr}
	For $i \in I$, we have
	\begin{equation*}
		\ell(s_i\tr^{-1}) =
		\begin{cases}
			\ell(\tr^{-1}) + 1 & \text{if $i \neq r$,} \\
			\ell(\tr^{-1}) - 1 & \text{if $i = r$,}
		\end{cases}
		\qquad
		\ell(\tr^{-1}s_i) =
		\begin{cases}
			\ell(\tr^{-1}) + 1 & \text{if $i \neq 0$,} \\
			\ell(\tr^{-1}) - 1 & \text{if $i = 0$.}
		\end{cases}
	\end{equation*}
\end{prop}
\begin{proof}
	For $i \in I$, it is well-known that 
	\begin{equation} \label{eq:reflect by ai}
	s_i \sfDel^\pm = \left( \sfDel^\pm \setminus \{ \pm\alpha_i \} \right) \cup \{ \mp\alpha_i \}.
	\end{equation}
	
	Let us first consider $\ell(s_i\tr^{-1})$.
	If $i \neq 0, r$, then $\tr^{-1}(-\alpha_i) = -\alpha_i$ by \eqref{eq:translation} and, by \eqref{eq:reflect by ai},
	$\sfDel^+(s_i\tr^{-1}) = \sfDel^+(\tr^{-1}) \cup \{ \alpha_i \}$. 
	If $i = 0$, then $\tr^{-1}(\alpha_0+a_r\delta) = \alpha_0$, so $\alpha_0 \notin \sfDel^+(\tr^{-1})$, while $\alpha_0 \in \sfDel^+(s_0\tr^{-1})$.
	As a result, $\ell(s_i\tr^{-1}) = \ell(\tr^{-1})+1$ for $i \neq r$.
	Suppose $i = r$. Since $\tr^{-1}(\alpha_r) = \alpha_r + \delta$, we have $\tr^{-1}(\alpha_r - \delta) = \alpha_r$, so $\alpha_r \in \sfDel^+(\tr^{-1})$. Thus, $\sfDel^+(s_r\tr^{-1}) = \sfDel^+(\tr^{-1}) \setminus \{ \alpha_r \}$ by \eqref{eq:reflect by ai}. This shows $\ell(s_r\tr^{-1}) = \ell(\tr^{-1}) - 1$. 

	Second, we consider $\ell(\tr^{-1}s_i)$.
	If $i = r$, then $\tr^{-1}s_r(-\alpha_r) = \alpha_r+\delta$, so $\ell(\tr^{-1}s_r) = \ell(\tr^{-1})+1$ by \eqref{eq:reflect by ai}. If $i \neq 0, r$, then $\tr^{-1}s_i(-\alpha_i) = \alpha_i$ implies $\ell(\tr^{-1}s_i) = \ell(\tr^{-1})+1$.
	Suppose $i = 0$. Then
	$\tr^{-1}(-\alpha_0) = (a_r-1)\delta + \theta \in \sfDel^+,$
	so $(a_r-1)\delta + \theta \in \sfDel^+(\tr^{-1})$.
	Note that $a_r - 1 \ge 0$ by its definition.
	By \eqref{eq:reflect by ai}, we have 
	\begin{equation*}
		\sfDel^+(\tr^{-1}s_0) = \sfDel^+(\tr^{-1}) \setminus \{ (a_r-1)\delta + \theta \},
	\end{equation*}
	which implies $\ell(\tr^{-1}s_0) = \ell(\tr^{-1})-1$.
\end{proof}

\begin{cor} \label{cor:rex of wr ending stau0}
The element $w_r \in W$ in \eqref{eq:translation by fundamentals} has a  reduced expression $s_{i_1}\dots s_{i_\ell}$ such that $s_{i_1} = s_r$ and $s_{i_\ell} = s_{\tau(0)}$. Furthermore, if $a_r = 1$, then the reduced expression $s_{i_1}\dots s_{i_\ell}$ of $w_r$ is unique up to $2$-braid relations, where a $2$-braid relation means $s_i s_j = s_j s_i$ for $i, j \in I$ such that $a_{ij} = 0$.
\end{cor}

\begin{rem} \label{rem:tr acts on al}
{\em 
In the proof of Proposition \ref{prop:length of tr}, we have
\begin{equation*}
	\tr^{-1}(\alpha_i) = 
	\begin{cases}
		\alpha_i & \text{if $i \in \I \setminus \{ r \}$,} \\
		\alpha_r + \delta & \text{if $i = r$,} \\
		\alpha_0 - a_r\delta & \text{if $i = 0$.}
	\end{cases}
\end{equation*}
In particular, $\alpha_r = \tr^{-1}\left(\alpha_r-\delta\right)$ and $\theta = \tr^{-1}\left(\theta-a_r\delta\right)$, which implies that $\alpha_r,\, \theta \in \sfDel^+(\tr^{-1})$.
}
\end{rem}

\begin{ex} \label{ex:rex of tr in rank 2}
{\em 
Let us illustrate Proposition \ref{prop:Delta+ by tr}, Corollary \ref{cor:length of tr-1}, and Proposition \ref{prop:length of tr} for types $A_2^{(1)}$, $C_2^{(1)}$, and $G_2^{(1)}$.
By \eqref{eq:translation} and Proposition \ref{prop:Weyl},
one may obtain a reduced expression of $\tr^{-1}$ as follows:
\smallskip

\begin{enumerate}[(1)]
	\item Type $A_2^{(1)}$.
	We have
	\begin{equation*}
	\begin{split}
		\tr^{-1} &= 
		\begin{cases}
			s_1 s_2 \tau_1 & \text{if $r = 1$,} \\
			s_2 s_1 \tau_2 & \text{if $r = 2$,}
		\end{cases} \\
		\sfDel^+(\tr^{-1}) &= 
		\begin{cases}
			\left\{ \, \alpha_1,\, \alpha_1 + \alpha_2 \, \right\} & \text{if $r = 1$,} \\
			\left\{ \, \alpha_2,\, \alpha_1 + \alpha_2 \, \right\} & \text{if $r = 2$.}
		\end{cases}
	\end{split}
	\end{equation*}
	Here $\tau_r$ is given by $i \mapsto i-r \mod{3}$ for $r = 1, 2$.
	\smallskip

	\item Type $C_2^{(1)}$.
	We have
	\begin{equation*}
	\begin{split}
		\tr^{-1} &=
		\begin{cases}
			s_1 s_2 s_1 s_0 & \text{if $r = 1$,} \\
			s_2 s_1 s_2 \tau & \text{if $r = 2$,}
		\end{cases} \\ 
		\sfDel^+\left( \tr^{-1} \right) &= 
		\begin{cases}
			\left\{ \alpha_1, 2\alpha_1 + \alpha_2, \alpha_1 + \alpha_2, \delta + 2\alpha_1 + \alpha_2 \right\} & \text{if $r = 1$,} \\
			\left\{ \alpha_2,\, \alpha_1 + \alpha_2,\, 2\alpha_1 + \alpha_2 \right\} & \text{if $r = 2$.}
		\end{cases}
	\end{split}
	\end{equation*}
	Here $\tau$ is given by $\tau(0) = 2$ and $\tau(1) = 1$. 
	\smallskip

	\item Type $G_2^{(1)}$.
	We have
	\begin{equation*}
	\begin{split}
		\tr^{-1} &=
		\begin{cases}
			s_1 s_2 s_1 s_2 s_1 s_0 & \text{if $r = 1$,} \\
			s_2 s_1 s_2 s_1 s_2 s_0 s_1 s_2 s_1 s_0 & \text{if $r = 2$,}
		\end{cases} \\
		\sfDel^+\left( \tr^{-1} \right) &= 
		\begin{cases}
			\left\{\, \alpha_1,\, \alpha_1+\alpha_2,\, 2\alpha_1+3\alpha_2,\, \alpha_1+2\alpha_2,\, \alpha_1+3\alpha_2,\, \delta+2\alpha_1+3\alpha_2 \,\right\} & \text{if $r = 1$,} \\
			\begin{split} 
				& \left\{\, \alpha_2,\, \alpha_1+3\alpha_2,\,\alpha_1+2\alpha_2,\,2\alpha_1+3\alpha_2,\,\alpha_1+\alpha_2 \,\right\} \\
				& \quad\quad \cup \left\{\, \delta+\alpha_1+3\alpha_2,\, \delta+2\alpha_1+3\alpha_2,\, \delta+\alpha_1+2\alpha_2 \,\right\} \\
				& \qquad\qquad \cup \left\{\, 2\delta+\alpha_1+3\alpha_2,\, 2\delta+2\alpha_1+3\alpha_2 \,\right\}
			\end{split}
			& \text{if $r = 2$.}
		\end{cases}
	\end{split}
	\end{equation*}
\end{enumerate}
}
\end{ex}

\begin{rem} \label{eq:rex of wr}
{\em 
	Let us assume that $a_r = 1$ (see Table \ref{tab:dynkins}). We denote by $\tilde{w}_r$ a reduced expression of $w_r$ \eqref{eq:translation by fundamentals}. Then we obtain an explicit description of $\tilde{w}_r$ given by
	{\allowdisplaybreaks
	\begin{gather*}
		\quad\,
		\begin{cases}
			{\bf s}_{(r,\,n)} {\bf s}_{( r-1,\,n-1)} \dots {\bf s}_{( 1,	\,n-r+1)}  & \text{for type $A_n$ with $r \in \I$,} \\
			(1, 2, \dots, n-1, n, n-1, \dots, 2, 1) & \text{for type 	$B_n$ with $r  = 1$,} \\
			\mathbf{s}_n \cdot \mathbf{s}_{n-1} \cdot\, \dots \,\cdot 	\mathbf{s}_1 & \text{for type $C_n$ with $r = n$,} \\
			{\bf s}_{(1,\,n)} \cdot {\bf s}_{(1,\,n-2)}^{-1} & \text{for 	type $D_n$ with $r = 1$,} \\
			\tilde{\bf s}_1 \tilde{\bf s}_2 \dots \tilde{\bf s}_{n-1} & 	\text{for type $D_n$ with $r = n$,} \\
			(1, 3, 4, 5, 6, 2, 4, 5, 3, 4, 2, 1, 3, 4, 5, 6) & \text{for 	type $E_6$ with $r = 1$,} \\
			(6, 5, 4, 3, 1, 2, 4, 3, 5, 4, 2, 6, 5, 4, 3, 1) & \text{for 	type $E_6$ with $r = 6$,} \\
			(7, 6, 5, 4, 3, 1, 2, 4, 3, 5, 4, 2, 6, 5, 4, 3, 1, 7, 6, 5, 4, 3, 2, 4, 5, 6, 7) & \text{for type $E_7$ with $r = 7$,} \\
		\end{cases}
	\end{gather*}\!\!}
	where $\mb{s}_{(i,\,j)} = (i, i+1, \dots, j)$,\, $\mb{s}_k = (n, 	n-1, \dots, n-k+1)$,\, $\tilde{\mb{s}}_k$ is given by
	\begin{equation*}
		\tilde{\bf s}_k =
		\begin{cases}
			s_n {\bf s}_{(k,\, n-2)}^{-1} & \text{if $k$ is odd,} \\
			{\bf s}_{(k,\, n-1)}^{-1} & \text{if $k$ is even,} \\
			s_n & \text{if $n$ is even and $k = n-1$,}
		\end{cases}
	\end{equation*}
	and the dot $\cdot$ means the concatenation of two sequences.
	Here $\mb{s}_{(i,j)}^{-1}$ is understood as the sequence obtained by 	reversing $\mb{s}_{(i,j)}$.
	For type $D_n^{(1)}$, a description of $\tilde{w}_{n-1}$ is obtained from $\tilde{w}_n$ by replacing $n$ with $n-1$.
}
\end{rem}

\section{Unipotent quantum coordinate ring as a module in $\mc{O}$} \label{sec:pseudo-prefundamentals}
\subsection{Unipotent quantum coordinate rings}
Let $\g$ be as in Section \ref{subsec:cartan data and quantum group}.
Let $w \in W$ be given.
For $\wtd{w} = (i_1,\dots,i_\ell)\in R(w)$, we have $\sfDel^+(w)=\{\,\beta_k\,|\,1\le k\le \ell\,\}$ (e.g.~see \cite{Pap94, Pap95}), where 
\begin{equation}\label{eq:root wrt rex}
\beta_k=s_{i_1}\dots s_{i_{k-1}}(\alpha_k)\quad (1\le k\le \ell).
\end{equation}
For $1\le k\le \ell$ and $c\in \Z_+$, let 
\begin{equation} \label{eq:root vector}
F(c\beta_k)=T_{i_1}\dots T_{i_{k-1}}\left(f_{i_k}^{(c)}\right)\in U_q^-(\mf{g})_{-c\beta_k}
\end{equation}
(see \cite{Lu10}).
In particular, when $c = 1$, we call it the {\it root vector} of $\beta_k$.
For ${\bf c}=(c_1,\dots,c_\ell)\in \Z_+^\ell$, put
\begin{equation} \label{eq:PBW monomial}
F\left({\bf c},\wtd{w}\right)=F(c_1\beta_1) \cdots F(c_\ell\beta_\ell),
\end{equation}
which is called a {\it PBW monomial} associated with $\sum_{k=1}^\ell c_k \beta_k$.

Assume that $\sfDel^+(w)$ is linearly ordered by $\beta_1<\dots<\beta_\ell$.
The following $q$-commutation relations for \eqref{eq:PBW monomial} are known \cite{LS91} (cf.~\cite{Ki12}):
\begin{equation}\label{eq:LS formula}
\begin{split}
F\left(c_j\beta_j\right)F\left(c_i\beta_i\right)-q^{-(c_i\beta_i,c_j\beta_j)}F\left(c_i\beta_i\right)F\left(c_j\beta_j\right)
=\sum_{{\bf c}'}f_{{\bf c}'}F\left({\bf c}',\wtd{w}\right)
\end{split}
\end{equation}
for $i<j$ and $c_i,c_j \in \Z_+$, where the sum is over ${\bf c}'=(c'_k)$ such that $c_i\beta_i+c_j\beta_j=\sum_{i\le k \le j}c'_{k}\beta_k$ with $c'_i<c_i$, $c'_j<c_j$ and $f_{{\bf c}'} \in {\bf k}$.

\begin{df} \label{df:Uq-w}
{\em 
For $w \in W$, we denote by $U_q^-(w)$ the vector space over ${\bf k}$ generated by $\left\{\,F\left({\bf c},\wtd{w}\right)\,\left|\,{\bf c}\in \Z_+^\ell\right.\right\}$.
}
\end{df}

\noindent
We remark that $U_q^-(w)$ does not depend on the choice of $\wtd{w}\in R(w)$, and it is the ${\bf k}$-subalgebra of $U_q^-(\mf{g})$ generated by $\left\{\left.F(\beta_k)\,\right|\,1\le k\le \ell\,\right\}$ by \eqref{eq:LS formula}.

Let $F^{\rm up}({\bf c},\wtd{w})$ be the dual of \eqref{eq:PBW monomial} with respect to the Kashiwara bilinear form $(\,\cdot\,\,,\,\,\cdot\,)_{\rm K}$ on $U_q^-(\g)$ \cite[Section 3.4]{Kas91}.
Since $U_q^-(\g)$ is a (twisted) self-dual bialgebra with respect to $(\,\cdot\,\,,\,\,\cdot\,)_{\rm K}$, we regard $F^{\rm up}({\bf c},\wtd{w})$ as an element of $U_q^-(\g)$ by
\begin{equation} \label{eq:dual PBW monomial}
	F^{\rm up}({\bf c},\wtd{w})
	=
	\frac{1}{\left(\,F({\bf c},\wtd{w})\,,\,F({\bf c},\wtd{w})\,\right)_{\rm K}} F({\bf c},\wtd{w}) \in U_q^-(\g).
\end{equation}
Let $\mc{A} = \mathbb{C}\left[q^{\pm 1}\right]$ and let $U_q^-(w)^{\rm up}_{\mc{A}}$ be the $\mc{A}$-lattice generated by $F^{\rm up}({\bf c},\wtd{w})$ for ${\bf c} \in \Z_+^\ell$.
Then the $\bk$-subalgebra $U_q^-(w)$ is called the {\em unipotent quantum coordinate ring}, since $\C\ot_{\mc{A}}U_q^-(w)^{\rm up}_{\mc{A}}$ is isomorphic to the coordinate ring of the unipotent subgroup $N(w)$ of the Kac-Moody group associated to $w$ (see \cite{GLS13, Ki12}).

\subsection{Category $\mc{O}$ and prefundamental modules}
Assume that $\g$ is of untwisted affine type.
Let $U_q'(\g)$ be the {\it quantum affine algebra} associated to $\g$ without the degree operator $q^d$.
It is well-known in \cite{B94, D} that $U_q'(\g)$ is also isomorphic to the $\bk$-algebra generated by
$x_{i,r}^\pm$ ($i\in\I, r\in \Z$), $k_{i}^{\pm 1}$ $(i\in\I)$, $h_{i,r}$ ($i\in\I, r\in \Z\setminus\{0\}$), and $C^{\pm \frac{1}{2}}$ subject to the following relations:
{\allowdisplaybreaks
\begin{gather*}
	\text{$C^{\pm \frac{1}{2}}$ are central with $C^{\frac{1}{2}}C^{-\frac{1}{2}}=C^{-\frac{1}{2}}C^{\frac{1}{2}} = 1$,} \\
	k_ik_j=k_jk_i,\quad k_ik_i^{-1}=k_i^{-1}k_i=1,\\
	k_ih_{j,r}=h_{j,r}k_i,\quad k_ix^{\pm}_{j,r}k_i^{-1}=\qi^{\pm a_{ij}}x^{\pm}_{j,r},\\
	[h_{i,r},h_{j,s}]=\delta_{r,-s}\frac{1}{r}[r a_{ij}]_i\frac{C^r-C^{-r}}{\qi-\qi^{-1}},\\
	[h_{i,r},x^{\pm}_{j,s}]=\pm \frac{1}{r}[r a_{ij}]_i C^{\mp |r|/2}x^{\pm}_{j,r+s},\\
	x^\pm_{i,r+1}x^\pm_{j,s}-\qi^{\pm a_{ij}}x^\pm_{j,s}x^\pm_{i,r+1}
	= \qi^{\pm a_{ij}}x^\pm_{i,r}x^\pm_{j,s+1}-x^\pm_{j,s+1}x^\pm_{i,r},\\
	[x^+_{i,r},x^-_{j,s}]=\delta_{i,j}\frac{C^{(r-s)/2}\psi^+_{i,r+s}-C^{-(r-s)/2}\psi^-_{i,r+s}}{\qi-\qi^{-1}},\\
	\sum_{w\in \mf{S}_m}\sum_{k=0}^m {\small \left[\begin{matrix} m \\ k \end{matrix}\right]}_i x^\pm_{i,r_{w(1)}}\dots x^\pm_{i,r_{w(k)}}x^\pm_{j,s}x^\pm_{i,r_{w(k+1)}}\dots x^\pm_{i,r_{w(m)}}=0\quad (i\neq j),
\end{gather*}
}
\noindent \!\!where $r_1, \dots, r_m$ is any sequence of integers with $m=1-a_{ij}$, $\mf{S}_m$ denotes the group of permutations on $m$ letters, and $\psi^\pm_{i,r}$'s are the elements determined by the following identity of formal power series in $z$:
\begin{equation}\label{eq:psi generators}
	\sum_{r=0}^\infty \psi^\pm_{i,\pm r} z^{\pm r} = k_i^{\pm 1}\exp\left( \pm(\qi-\qi^{-1}) \sum_{s=1}^\infty h_{i,\pm s} z^{\pm s} \right).
\end{equation}
\smallskip

Let $U_q(\bo)$ be the subalgebra of $U_q'(\g)$ generated by $e_i$ and $k_i^{\pm 1}$ for $0 \le i \le n$ and $C^{\pm \frac{1}{2}}$.

\begin{rem} \label{rem: presentation of Borel subalgebra}
{\em 	
Note that $U_q(\bo)$ is a Hopf subalgebra of $U_q'(\g)$ and it is isomorphic to the $\bk$-algebra with generators $e_i$ and $k_i^{\pm 1}$ for $0 \le i \le n$ satisfying the relations in the Drinfeld-Jimbo's presentation of $U_q'(\g)$ except for the relations involving $f_i$'s (e.g.~see \cite[\S 4.21]{Jan}).
}
\end{rem}

Let $\mf{t}$ be the subalgebra of $U_q(\bo)$ generated by $k_i^{\pm 1}$ for $i\in\I$, and let $\mf{t}^\ast = (\bk^\times)^{\I}$ be the set of maps from $\I$ to $\bk^\times$, which is a group under pointwise multiplication.

Let $U_q'(\g)^\pm$ (resp.~$U_q'(\g)^0$) be the subalgebras of $U_q'(\g)$ generated by $x_{i,r}^{\pm}$ for $i\in\I$ and $r\in \Z$ $\big( \text{resp. $k_i^{\pm 1}$, $\psi^\pm_{i,\pm r}$ for $i\in\I$, $r>0$ and $C^{\pm \frac{1}{2}}$} \big)$. 
If we put $U_q(\bo)^+ =U_q'(\g)^+\cap U_q(\bo)$ and $U_q(\bo)^0 =U_q'(\g)^0\cap U_q(\bo)$, then 
it follows from \cite{B94,B94b} (cf.~\cite{BCP, HJ}) that
$U_q(\bo)^+=\langle\,x_{i,r}^+\,\rangle_{i\in\I, r\ge 0}$ and
$U_q(\bo)^0=\langle\,\psi_{i,r}^+, k_i^{\pm 1}, C^{\pm \frac{1}{2}} \,\rangle_{i\in\I, r>0}$.

\begin{rem}
	{\em
		Throughout this paper, we assume that $C^{\pm \frac{1}{2}}$ acts trivially on a $U_q(\bo)$-module.
	}
\end{rem}

Let $V$ be a $U_q(\bo)$-module. For $\omega \in {\mf t}^*$, we define the weight space of $V$ with weight $\omega$ by
\begin{equation} \label{eq:weight space wrt t}
	V_{\omega} = \{\, v \in V \, | \, k_i v = \omega(i)v \  (i \in \I) \,\}.
\end{equation}
We say that $V$ is {\em ${\mf t}$-diagonalizable} if $V = \oplus_{\omega \in {\mf t}^*} V_{\omega}$.

A series ${\bf \Psi} = (\Psi_{i, m})_{i \in \I, m \ge 0}$ of elements in $\bk$ such that $\Psi_{i, 0} \neq 0$ for all $i \in \I$ is called an {\em $\ell$-weight}.
We often identify ${\bf \Psi} = (\Psi_{i, m})_{m \ge 0}$ with ${\bf \Psi} = (\Psi_i(z))_{i \in \I}$, a tuple of formal power series, where
\begin{equation*}
	\Psi_i(z) = \sum_{m \ge 0} \Psi_{i,m} z^m.
\end{equation*}
We denote by ${\mf t}_{\ell}^*$ the set of $\ell$-weights. Since $\Psi_i(z)$ is invertible, ${\mf t}_{\ell}^*$ is a group under multiplication. Let $\varpi : {\mf t}_{\ell}^* \longrightarrow {\mf t}^*$ be the surjective morphism defined by $\varpi(\Psi)(i) = \Psi_{i,0}$ for $i\in \I$.

For ${\bf \Psi} \in {\mf t}_{\ell}^*$, we define the {\em $\ell$-weight space} of $V$ with $\ell$-weight ${\bf \Psi}$ by
\begin{equation*}
	V_{\bf \Psi} = \left\{ v \in V \, \left| \text{ there exist } p \in \mathbb{Z}_+ \text{ such that } (\psi_{i,m}^+ - \Psi_{i,m})^p v = 0 \text{ for all } i \in I \text{ and } m \ge 0 \right. \right\}.
\end{equation*}
For ${\bf \Psi} \in {\mf t}_{\ell}^*$, we say that $V$ is of {\em highest $\ell$-weight $\bf \Psi$} if there exists a non-zero vector
$v \in V$ such that
\vskip 2mm
\begin{center}
	(i) $V = U_q(\bo)v$ \,\, (ii) $e_i v = 0$ for all $i \in \I$ \,\, (iii) $\psi_{i,m}^+ v = \Psi_{i, m}v$ for $i \in \I$ and $m \ge 0$.
\end{center}
\vskip 2mm
A non-zero vector $v \in V$ is called a {\it highest weight vector of $\ell$-weight ${\bf \Psi}$} if it satisfies the conditions (ii) and (iii).
There exists a unique irreducible $U_q(\bo)$-module of highest $\ell$-weight $\bf \Psi$, which we denote by $L({\bf \Psi})$.

\begin{df}{\em \!\!\!\cite[Definition 3.7]{HJ}} \label{df:prefundamental}
	{\em
	For $r \in I_0$ and $a \in \bk^{\times}$, let $L_{r, a}^{\pm}$ be an irreducible $U_q(\bo)$-module of highest $\ell$-weight ${\bf \Psi}$ with
	\begin{equation*} 
		\Psi_i(z) =
		\left\{
		\begin{array}{ll}
			(1-az)^{\pm1} & \text{if $i=r$},      \\
			1             & \text{if $i \neq r$}.
		\end{array}
		\right.
	\end{equation*}
	The $U_q(\bo)$-module $L_{r, a}^-$ (resp.~$L_{r, a}^+$) is called the {\em negative} (resp.~{\em positive}) {\em prefundamental module}.
	}
\end{df}

We define a map $\widetilde{\,\,\,} : \mr{\sfP} \rightarrow {\mf t}^*$ given by $\widetilde{\varpi}_i(j) = q_i^{\delta_{ij}}$ for $i, j \in \I$.
Note that $\widetilde{\alpha}_i \in {\mf t}^*$ is given by $\widetilde{\alpha}_i(j)= q_i^{a_{ij}}$ for $i, j \in \I$.
We define a partial order $\le$ on ${\mf t}^*$ by
$\omega' \le \omega$  if and only if $\omega'\omega^{-1}$ is a product of $\widetilde{\alpha}_i^{-1}$'s.
For $\lambda \in {\mf t}^*$, put $D(\lambda) = \{\, \omega \in {\mf t}^* \, | \, \omega \le \lambda \,\}$.

\begin{df}{\em \!\!\!\cite[Definition 3.8]{HJ}} \label{df:category O}
	{\em
		Let $\mathcal{O}$ be the category of $U_q(\bo)$-modules $V$ such that
		\begin{enumerate}
			\item[(i)] $V$ is ${\mf t}$-diagonalizable,
			\item[(ii)] $\dim V_{\omega} < \infty$ for all $\omega \in {\mf t}^*$,
			\item[(iii)] there exist $\lambda_1, \dots, \lambda_s \in {\mf t}^*$ such that the weights of $V$ are in $\bigcup_{j=1}^s D(\lambda_j)$.
		\end{enumerate}
	}
\end{df}

\noindent
The category $\mc{O}$ is closed under taking finite direct sums,  quotients, and finite tensor products of objects in $\mc{O}$.

The simple objects in $\mc{O}$ are characterized in terms of tuples ${\bf \Psi} = (\Psi_i(z))_{i \in \I}$ of rational functions regular and non-zero at $z = 0$, called {\it Drinfeld rational fractions}, which can be regarded as a natural extension of {\it Drinfeld polynomials} \cite{CP95a, CP95b}, as follows:
\begin{thm}{\em \!\!\!\cite[Theorem 3.11]{HJ}} \label{thm:characterization of simples in O}
	For ${\bf \Psi} \in {\mf t}_{\ell}^*$, $L({\bf \Psi})$ is in the category $\mc{O}$ if and only if $\Psi_i(z)$ is rational for all $i\in \I$.
\end{thm}

\begin{rem} \label{rem:prefundamentals are building blocks of O}
{\em 
For ${\bf \Psi}, {\bf \Psi}'\in {\mf t}_{\ell}^*$, 
it follows from the formulas of $\Delta(\Psi_{i,\pm k}^{\pm})$ and $\Delta(x_{i,k}^+)$ \cite{Dami00} (see also \cite[Theorem 2.6]{HJ}) that $L({\bf \Psi}{\bf \Psi}')$ is a subquotient of $L({\bf \Psi})\ot L({\bf \Psi}')$.
Since we have by \cite[Corollary 4.8, Corollary 5.1]{HJ} that $L_{r,a}^{\pm} \in \mc{O}$ for all $r \in I$ and $a \in \bk^{\times}$, we conclude that any irreducible $U_q(\bo)$-module in $\mc{O}$ is a subquotient of a tensor product of prefundamental modules and one-dimensional modules.
	Here a one-dimensional module has a trivial action of $e_0, \dots e_n$ (see \cite[Section 3.2]{HJ} and \cite[Definition 3.6]{FH}).
}
\end{rem}

\subsection{$U_q(\bo)$-module structure on $U_q(w)$}

Let $\hat{w} \in \eW$ be given, where $\hat{w}$ can be written as $\hat{w} = s_{i_1} s_{i_2} \cdots s_{i_\ell} \tau$ with $\tau \in \mc{T}$ by Proposition \ref{prop:Weyl}.
Then we define 
\begin{equation} \label{eq:def of Tw}
	T_{\hat{w}} = T_{i_1} \cdots T_{i_\ell} \tau \in {\rm Aut}(U_q(\g)),
\end{equation}
where $\tau$ is understood as an automorphism of $U_q(\g)$ \cite{B94}.
Note that $T_{\hat{w}}$ is independent of the choice of the reduced expression of $\hat{w}$.
For simplicity, we set $T_\lambda = T_{t_\lambda}$ in the case of $\hat{w} = t_\lambda$ for $\lambda \in \eM$.

\begin{lem}\!\!\!{\em \cite[Lemma 3.2]{B94}} \label{lem:action of Tlar}
	\,$T_{-\lar}(e_i) = e_i$ for $i \in \I \setminus \{r\}$.
\end{lem}

Let us consider the $\ms{B}_q(\g)$-module $M$ \eqref{eq:def of M}, and let
$$
	\iota : \!\!\xymatrix@C=2em{U_q^-(\g) \ar@{->}[r] & M}
$$ 
be the canonical projection from $U_q^-(\g)$ to $M$, which is an isomorphism of $\ms{B}_q(\g)$-modules by Lemma \ref{lem:M is isom to Uq- as Bq-mod}.
For $w \in W$, we define
\begin{equation} \label{eq:def of Mw}
	M_w = \iota\left( U_q^-(w) \right) \subset M.
\end{equation}
We may identify $U_q^-(w)$ with $M_w$.
For simplicity, we denote $\overline{x} \in M_w$ by $x$ if there is no confusion (cf.~\eqref{eq:def of M}).

\begin{lem} \label{lem:closedness of Mr}
	$M_w$ is closed under the action of $\mb{e}_i'$ for all $i\in I$.
\end{lem}	
\begin{proof}
Let $A_w$ be the ${\bf k}$-linear subspace of $U_q^-(\g)$ spanned by all elements $x \in U_q^-(\g)$ such that
$$
e_{i_1}' \dots e_{i_\ell}' x = 0
$$
for any $\beta \in \Delta^+ \cap w\Delta^+$ and any sequence $(i_1, \dots, i_\ell) \in I^\ell$ satisfying $\sum_{k=1}^\ell \alpha_{i_k} = \beta$, where $\ell = {\rm ht}(\beta)$.
Since \cite[Theorem 2.20]{KKOP18} says $A_q(n(w)) = A_w$ in the convention of \cite{KKOP18}, we have  
\begin{equation*}
	U_q^-(w) = *\left( A_w \right).
\end{equation*}
By Remark \ref{rem:eistar}, for any $x \in U_q^-(\g)$, we have that $x \in U_q^-(w)$ if and only if 
$$
e_{i_1}^\star \dots e_{i_\ell}^\star x = 0
$$
for any $\beta \in \Delta^+ \cap w\Delta^+$ and any sequence $(i_1, \dots i_\ell) \in I^\ell$ such that $\sum_{k=1}^\ell \alpha_{i_k} = \beta$.
Since $e_j^\star$ commutes with $e_i'$,  we have 
$$
 e_{i_1}^\star \dots e_{i_\ell}^\star (e_i' x) =   e_i'( e_{i_1}^\star \dots e_{i_\ell}^\star  x) = 0
$$ 
for any $x \in U_q^-(w) $ and any sequence $(i_1, \dots i_\ell) \in I^\ell$  with $\sum_{k=1}^\ell \alpha_{i_k} \in \sfDel^+ \cap w\sfDel^+$.
Therefore the element $e_i'(x)$ is contained in $U_q(w)$ for any $x\in U_q(w)$ and $i\in I$.
This proves that $M_w$ in \eqref{eq:def of Mw} is closed under the action of $\mb{e}_i'$ for all $i\in I$.
\end{proof}

\begin{rem}
{\em 
The convention for quantum unipotent coordinate rings $A_q(n(w))$ used in \cite{KKKO18, KKOP18} is different from the one in this paper, where $A_q(n(w))$ in \cite{KKOP18} is generated by the dual root vectors defined by using $T_i^* := * \circ T_i \circ *$.
}
\end{rem}

For $r \in \I$, let us write $M_r = M_{w_r}$ for simplicity, where $w_r$ is given in \eqref{eq:translation by fundamentals}.
For $r \in \I$, let
\begin{equation*} 
	x_0 = T_{-\lar}(e_0).
\end{equation*}

\begin{lem} \label{lem:description of x0}
	We have
	$$x_0 = y_0 q_0^{-\tr^{-1}(h_0)} = q_0^{-\tr^{-1}(h_0)} q_0^2 y_0$$
	for some $y_0 \in U_q^-(w_r)$ such that $e_r'(y_0) = 0$.
\end{lem}

\begin{proof} 
	By Corollary \ref{cor:rex of wr ending stau0},
	one can take a reduced expression of $w_r$, say $s_{i_1} \dots s_{i_\ell}$, with $i_1 = r$ and $i_\ell = \tau(0)$.
	By \eqref{eq:Ti(ej)} and \eqref{eq:def of Tw}, we compute
	\begin{equation} \label{eq:compute x_0}
	\begin{split}
		x_0 &= T_{i_1} \dots T_{i_\ell} \tau (e_0) \\
			&= T_{i_1} \dots T_{i_{\ell-1}}(-f_{\tau(0)}k_{\tau(0)}) \\
			&= -T_{i_1}\dots T_{i_{\ell-1}}(f_{\tau_r(0)})
			T_{i_1}\dots T_{i_{\ell-1}}(k_{\tau_r(0)})
	\end{split}
	\end{equation}
	Put $y_0 = -T_{i_1} \dots T_{i_{\ell-1}}(f_{\tau(0)})$.
	By Definition \ref{df:Uq-w},
	we have $y_0 \in U_q^-(w_r)$.
	Since $i_1 = r$, we have by Proposition \ref{prop:length of tr} and \cite[Proposition 40.1.3]{Lu10} 
	\begin{equation*}
		T_r^{-1}\left(y_0\right) \in U_q^-(s_rw_r) \subset U_q^-(\g).
	\end{equation*} 
	Thus it follows from \cite[Proposition 38.1.6]{Lu10} (cf.~\cite[Lemma 4.13]{Ki12})\footnote{In \cite{Lu10, Ki12}, the $q$-derivation $e_i'$ is denoted by ${}_ir$.} that $e_r'(y_0) = 0$.
	On the other hand, we have by \eqref{eq:Ti(qh)}
	\begin{equation*}
	\begin{split}
		T_{i_1} \dots T_{i_{\ell-1}} (k_{\tau(0)})  
		= T_{i_1} \dots T_{i_{\ell-1}} \left( q^{d_{\tau(0)}h_{\tau(0)}} \right)
		= q^{d_{\tau(0)} s_{i_1} \dots s_{i_{\ell-1}}(h_{\tau(0)})}
		= q_0^{-\tr^{-1}(h_0)}
	\end{split}
	\end{equation*}
	where $d_0 = d_{\tau(0)}$ and $i_\ell = \tau(0)$.
	Note that we also have $y_0 q_0^{-\tr^{-1}(h_0)} = q_0^{-\tr^{-1}(h_0)} q_0^2 y_0$ from \eqref{eq:compute x_0} by using the defining relation $q^h f_i= q^{-\langle h, \alpha_i \rangle} f_i q^h$.
	We complete the proof.
\end{proof}%

By Lemma \ref{lem:realization of Bqg}, Corollary \ref{cor:braid on q-bosons}, and Remark \ref{rem:domain of Ti}, ${\texttt T}_i$ is also understood as a $\bk$-algebra homomorphism from $\mf{U}^+[i]$ to $\mf{U}$, where $\mf{U}^+[i]$ is the $\bk$-subalgebra of $\mf{U}$ generated by $\mf{U}^+$ and ${}^*U_q^-(\g)[i]$.
By replacing $T_i$ with ${\texttt T}_i$ in \eqref{eq:def of Tw}, we denote by ${\texttt T}_{-\lar}$ the map associated to $T_{-\lar}$.
\vskip 1mm

Put
\begin{equation} \label{eq:def of mbx0}
	\mb{x}_0 = {\texttt T}_{-\lar} (\mb{e}_0').
\end{equation}

\begin{prop} \label{prop:action of mbx0}
	$\mb{x}_0$ is a well-defined element in $\mf{U}^-$, and 
	$$\mb{x}_0 \cdot \overline{m} = -(q_{\tau(0)}-q_{\tau(0)}^{-1})q_{\tau(0)}^2\overline{y_0m}$$ 
	for $\overline{m} \in M_r$, where $y_0$ is given in Lemma \ref{lem:description of x0}.
\end{prop}
\begin{proof}
	It follows from \eqref{eq:Ti acts on bej'} and Corollary \ref{cor:rex of wr ending stau0} that 
	\begin{equation*}
		\text{\bf x}_0 = {\texttt T}_{-\lar}(\text{\bf e}_0') = (q_{\tau(0)}-q_{\tau(0)}^{-1})q_{\tau(0)}^2 {\texttt T}_{w_r s_{\tau(0)}}(f_{\tau(0)}) \in \mf{U}^-.
	\end{equation*}
	Here ${\texttt T}_{w_r s_{\tau(0)}}(f_{\tau(0)}) \in \mf{U}^-$ by the argument of Corollary \ref{cor:braid rel}.
\end{proof}

The following is a direct consequence of Corollary \ref{cor:braid on q-bosons} (cf.~\eqref{eq:composition Ti}), Corollary \ref{cor:braid rel}, Lemma \ref{lem:action of Tlar}, and Lemma \ref{lem:description of x0} (cf.~\eqref{eq:compute x_0}).

\begin{cor} \label{cor:description of x0 and Tla}
	\mbox{}
	\begin{itemize}
		\item[(1)] We have $\mb{e}_r'(\mb{x}_0 \cdot \overline{m}) = q_0^{a_{0r}} \mb{x}_0 \cdot \overline{\mb{e}_r' m}$ for $\overline{m} \in M_r$.
		
		\item[(2)] For $i \in \I \setminus \{ r \}$, we have {\em ${\texttt T}_{-\lar}(\mb{e}_i') = \mb{e}_i'$}.
	\end{itemize}
\end{cor}

For $i\in I$ and $a \in \bk^\times$, we define $\bk$-linear operators on $M_r$ as follows: 
\smallskip
\begin{equation} \label{eq:Borel actions}
\kaction{i}(u) = 
\begin{cases}
q^{\left(\al_i,{\rm wt}(u)\right)}u \quad &\text{if $i \in \I$}\,,  \\
q^{-\left(\theta,{\rm wt}(u)\right)}u  \quad &\text{if $i=0$}\,,
\end{cases}
\qquad
\eaction{i}(u) = 
\begin{cases}
	 \mb{e}_i'(u)  & \text{ if } i \in \I,  \\
	a\mb{x}_0 u.  & \text{ if } i = 0,
\end{cases}
\end{equation}
\smallskip
\noindent
where $u \in M_r$ is homogeneous.
Note that $\eaction{i}$ is well-defined by Lemma \ref{lem:closedness of Mr} and Proposition \ref{prop:action of mbx0} with \eqref{eq:LS formula}.
\smallskip

Now, we are in position to state our first main result of this paper, which gives a family of $U_q(\bo)$-modules in the category $\mc{O}$ for all untwisted affine types.

\begin{thm} \label{thm:1st main}
There exists a $\bk$-algebra homomorphism
\begin{equation*}
	\rho :
	\xymatrix@R=0em
	{
		U_q(\bo) \ar@{->}[r] & {\rm End}_\bk(M_r) \\
		e_i \ar@{|->}[r] & \eaction{i} \\
		k_i \ar@{|->}[r] & \kaction{i}
	}
\end{equation*}
Hence $M_r$ becomes a $U_q(\bo)$-module, which we denote by $M_{r,a}$.
Moreover, $M_{r,a}$ belongs to $\mc{O}$.
\end{thm}
\begin{proof}
Let us recall Remark \ref{rem: presentation of Borel subalgebra}.
For $a \in \Z_{\le 0}$ and $i\in I$, let us write the {\it quantum Serre relation} (cf.~Section \ref{subsec:cartan data and quantum group}) as follows:
\begin{equation*}
R_{l}(i;X,Y) := \sum_{k=0}^{1-l} (-1)^k \qbn{1-l}{k}_i  X^{ k } Y X^{ 1-l-k },
\end{equation*}
where $X$ and $Y$ are symbols. 
For $i, j \in I$, let us check the following relation:
\begin{equation*} 
	R_{a_{ij}}(i ; \eaction{i}, \eaction{j}) = 0. 
\end{equation*}

{\it Case 1}. $i, j \in I_0$.
The relation in this case follows from Lemma \ref{lem:negative half is irreducible module over boson algebra} and Lemma \ref{lem:realization of Bqg}.
\smallskip

{\it Case 2}. $i = 0$ and $j \in I_0 \setminus \{ r\}$. 
By Proposition \ref{prop:properties of tUqg}, \eqref{eq:def of mbx0}, and Corollary \ref{cor:description of x0 and Tla}(2), we have
\begin{equation*}
\begin{split}
	R_{a_{0j}}(0 ; \eaction{0} , \eaction{j}) &   
	= R_{a_{0j}}(0 ; a{\texttt T}_{-\lar}(\mb{e}_0') , \mb{e}_j')  
	= a^{1-a_{0j}} {\texttt T}_{-\lar} \left( R_{a_{0j}}(0 ; \mb{e}_0' , \mb{e}_j') \right) = 0, \\
	R_{a_{j0}}(j ; \eaction{j}, \eaction{0})  &
	= R_{a_{j0}}(j ; \mb{e}_j', a{\texttt T}_{-\lar}(\mb{e}_0'))
	= a {\texttt T}_{-\lar} \left( R_{a_{j0}}(j ; \mb{e}_j', \mb{e}_0') \right) = 0.
\end{split}
\end{equation*}

{\it Case 3}. $i = 0$ and $j = r$.
As $\bk$-linear operators, it follows from Lemma \ref{lem:realization of Bqg} and Corollary \ref{cor:description of x0 and Tla}(1) that 
\begin{equation} \label{eq:er e0 commute}
	\eaction{r} \eaction{0} = a\mb{e}_r' \mb{x}_0 = aq_0^{a_{0r}} \mb{x}_0 \mb{e}_r' = q_0^{a_{0r}} \eaction{0} \eaction{r}
\end{equation}
(cf.~\cite[Remark 3.11]{JKP23}).
Note that it is known in \cite[(3.2.8)]{Kas91} that 
\begin{gather}
\prod_{k=0}^{n-1} (1 - q_i^{2k} z) = \sum_{k=0}^{n} (-1)^k q_i^{(n-1)k} \qbn{n}{k}_i z^k, \label{eq: qbn1} \\ 
\prod_{k=0}^{n-1} (1 - q_i^{-2k} z) = \sum_{k=0}^{n} (-1)^k q_i^{-(n-1)k} \qbn{n}{k}_i z^k, \label{eq: qbn2}
\end{gather}
where $z$ is an indeterminate.
Put $b = 1-a_{0r}$ and $c = 1-a_{r0}$.
By \eqref{eq:er e0 commute}, we have
\begin{align*}
R_{a_{0r}}(0 ; \eaction{0}, \eaction{r})	
= \left( \sum_{k=0}^b (-1)^k  q_0^{-a_{0r}k} \qbn{b}{k}_0  \right)  \eaction{r} \eaction{0}^{ b }
= \left( \prod_{k=0}^{-a_{0r}} (1 - q_0^{2k} ) \right) \eaction{r} \eaction{0}^{ b } = 0, \\
R_{a_{r0}}(r ; \eaction{r}, \eaction{0})
= \left( \sum_{k=0}^c (-1)^k  q_0^{a_{0r}k} \qbn{c}{k}_0  \right) \eaction{0} \eaction{r}^{ c }
= \left( \prod_{k=0}^{-a_{0r}} (1 - q_0^{-2k} ) \right) \eaction{0} \eaction{r}^{ c } = 0,
\end{align*}
where we apply \eqref{eq: qbn1} (resp.~\eqref{eq: qbn2}) by putting $z = 1$, $n = b$ (resp.~$n = c$), and $i = 0$.
\smallskip

It is straightforward to check the remaining relations for $\kaction{i}$ and $\eaction{i}$. We leave the details to the reader (cf.~\cite[Theorem 3.10]{JKP23}).

Finally, it follows from \eqref{eq:def of Mw} that the $U_q(\bo)$-module $M_r$ is $\mathfrak{t}$-diagonalizable, where each weight space of $M_r$ is finite-dimensional.
Also, all weights of $M_r$ are in $D(\overline{1})$, where $\overline{1}$ is the map from $I$ to $\bk^\times$ by $i \mapsto 1$ for all $i \in I$.
Hence, $M_r$ is an object of $\mc{O}$ (see Definition \ref{df:category O}).
\end{proof}

\begin{rem} \label{rem:compare with previous work}
{\em 
Let us consider the case of $a_r = 1$ (see~Table \ref{tab:dynkins}).
By Remark \ref{rem:tr acts on al}, we have
\begin{equation*}
	\theta = \delta - \alpha_0 = \tr^{-1}(-\alpha_0) = w_r\left(-\alpha_{\tau(0)}\right) \in \sfDel^+(w_r).
\end{equation*}
By Corollary \ref{cor:rex of wr ending stau0} (cf.~Remark \ref{eq:rex of wr}), a reduced expression of $w_r$ is unique up to $2$-braid relations (cf.~\cite[Lemma 3.2]{JKP23}).
In the proof of Proposition \ref{prop:action of mbx0}, we have
\begin{equation*} 
	\mb{x}_0 
	= \left(q_{\tau(0)}-q_{\tau(0)}^{-1}\right) q_{\tau(0)}^2\, {\texttt T}_{w_rs_{\tau(0)}} \left(f_{\tau(0)}\right),
\end{equation*}
and hence $\mb{x}_0$ coincides with the one in \cite[(3.17)]{JKP23} up to scalar multiplication (for types $A$ and $D$).

More precisely, let $F(\beta_\ell)$ be the root vector with respect to $w_r = s_{i_1}\dots s_{i_\ell}$ (cf.~\eqref{eq:root vector}), where $\beta_\ell = \theta$ by Remark \ref{eq:rex of wr}.
By \cite[Section 2.2]{Le04} (see also \cite[Lemma 2.12]{Ki12}) and \cite[Proposition 38.2.3]{Lu10}, we compute
\begin{equation*}
	\left( F(\beta_\ell), F(\beta_\ell) \right)_{\rm K} = 
	\prod_{i \in \I} \left( 1-q_i^2 \right)^{a_i}.
\end{equation*}
Then we have (cf.~\eqref{eq:dual PBW monomial})
\begin{equation*}
	F^{\rm up}(\beta_\ell) = \frac{ 1 }{\displaystyle \left(q_{\tau(0)}-q_{\tau(0)}^{-1}\right) q_{\tau(0)}^2 \prod_{i \in \I} \left( 1-q_i^2 \right)^{a_i} } \left( \mb{x}_0 \cdot \overline{1} \right).
\end{equation*}
Hence, the $U_q(\bo)$-action in \eqref{eq:Borel actions} recovers the one in \cite[(4.19)]{JKP23} up to scalar multiplication.
In this sense, Theorem \ref{thm:1st main} can be viewed as a generalization of \cite[Theorem 4.20]{JKP23} for all untwisted affine types.
}
\end{rem}

\begin{prop} \label{prop:ar ge 2 cases}
For $r \in \I$ with $a_r \ge 2$,
the $U_q(\bo)$-module $M_{r,a}$ is not irreducible.
\end{prop}
\begin{proof}
Let us consider a subset of $\sfDel^+(w_r)$ given by
\begin{equation*}
	\sfDel^+(w_r)_{\rm max} = \left\{ \, \alpha + (a_r-1)\delta \in \sfDel^+(w_r) \, \left| \,  \alpha \in \mr{\sfDel}^+ \text{ with } (\lar, \alpha) = a_r \right. \right\}.
\end{equation*}
Let $U_q^-(w_r)'$ be the $\bk$-linear subspace of $U_q^-(w_r)$ spanned by monomials in PBW vectors $F(\gamma)$ for $\gamma \in \sfDel^+(w_r)_{\rm max}$.
For $i \in \I$, we assume $e_i' F(\gamma) \neq 0$ for $\gamma \in \sfDel^+(w_r)_{\rm max}$, where $\gamma = \alpha + (a_r-1)\delta$ for some $\alpha \in \mr{\sfDel}^+$ with $(\lar, \alpha) = a_r$.
By Lemma \ref{lem:closedness of Mr},  we know $e_i' F(\gamma) \in U_q^-(w_r)$, so it is written as a linear combination of monomials in PBW vectors associated to $\sfDel^+(w_r)$, say
\begin{equation} \label{eq:ei Fgamma}
\begin{split}
	e_i' F(\gamma) 
	= \sum_{{\bf c} \in \Z_+^\ell} \eta_{\bf c} F({\bf c}, \tilde{w}_r).
\end{split}
\end{equation} 
for $\eta_{\bf c} \in \bk$.
Suppose $\eta_{\bf c} \neq 0$ for some ${\bf c} = (c_k) \in \Z_+^\ell$.
Then we have
$$
	-(\alpha-\alpha_i)-(a_r-1)\delta = {\rm wt}\left( e_i' F(\gamma) \right) = {\rm wt}\left( F({\bf c}, \tilde{w}_r) \right) = -\left( \sum_{k=1}^\ell c_k \mr{\beta}_k \right)-\left( \sum_{k=1}^\ell c_k L_k \right)\delta,
$$
where $\beta_k = \mr{\beta}_k + L_k\delta \in \sfDel^+(w_r)$ with $\mr{\beta}_k \in \mr{\sfDel}^+$ and $0 \le L_k < (\lar, \mr{\beta}_k)$ by Proposition \ref{prop:Delta+ by tr}. 
Therefore $\alpha-\alpha_i = \sum_{k=1}^\ell c_k \mr{\beta}_k$ and $a_r - 1 = \sum_{k=1}^\ell c_k L_k$. Since $0 \le L_k < (\lar, \mr{\beta}_k)$, we put $(\lar, \mr{\beta}_k) = L_k + r_k$ for $r_k \ge 1$.

Since $(\lar, \alpha) = a_r$, we have $(\lar, \alpha-\alpha_i) = a_r - \delta_{ir}$, where $\delta_{ir} = 1$ if $i = r$, otherwise $\delta_{ir} = 0$.
Then 
\begin{equation} \label{eq:ar-1}
	(a_r-1) + \sum_{k=1}^\ell c_k r_k
	= \sum_{k=1}^\ell c_k (L_k + r_k)
	= \sum_{k=1}^\ell c_k (\lar, \mr{\beta}_k)
	= (\lar, \alpha-\alpha_i)
	= a_r - \delta_{ir}
\end{equation}
Let $N$ be the number of distinct $k$'s on which $c_k \neq 0$ in the given ${\bf c} = (c_k)_{1 \le k \le \ell}$. Note that $N \ge 1$. 
	If $N > 1$, then $\sum_{k=1}^\ell c_k r_k \ge 2$, but it is impossible due to \eqref{eq:ar-1}.
This implies $N = 1$, and then $F({\bf c}, \tilde{w}_r) = F(\gamma-\alpha_i)$. 
Hence we have $\eta_{\bf c} F(\gamma-\alpha_i) = e_i' F(\gamma) \in U_q^-(w_r)$ for some $\eta_{\bf c} \in \bk^\times$, so $\gamma - \alpha_i \in \sfDel^+(w_r)_{\rm max}$.
By \eqref{eq:derivation ei'} and \eqref{eq:LS formula}, this implies that $U_q^-(w_r)'$ is closed under the action of $e_i'$ for all $i \in \I$.

On the other hand, the weight of $y_0$ in Lemma \ref{lem:description of x0}
is $\alpha_0 - a_r \delta = -\theta - (a_r-1)\delta \in -\sfDel^+(w_r)_{\rm max}$ by the definition of $x_0$.
Thus, $U_q^-(w_r)'$ is closed under the action given by left multiplication of $y_0$ by \eqref{eq:LS formula}.

Set $M_{r,a}' = \iota\left( U_q^-(w_r)' \right)$.
Then our assertion so far implies that $M_{r,a}'$ is a non-trivial proper $U_q(\bo)$-submodule of $M_{r,a}$ by \eqref{eq:def of Mw} and Proposition \ref{prop:action of mbx0}.
\end{proof}

\begin{ex}
{\em 
Let us consider $M_{1,a}$ for type $C_2^{(1)}$.
By Example \ref{ex:rex of tr in rank 2},
in this case, we have
$$M_{1,a}' = \bk\,\langle \mb{x}_0^n \, | \, n \in \Z_+ \rangle.$$
Observe
\begin{equation*}
	\eaction{i} \mb{x}_0 = 
	\begin{cases}
		\mb{0} & \text{if $i = 1, 2$,} \\
		a\mb{x}_0^2 & \text{if $i = 0$,}
	\end{cases}
\end{equation*} 
where $\eaction{1} \mb{x}_0 = \mb{e}_1' \mb{x}_0 = \mb{0}$ by Corollary \ref{cor:description of x0 and Tla}(1), and $\eaction{2} \mb{x}_0 = \mb{e}_2' \mb{x}_0 = \mb{0}$ by Lemma \ref{lem:closedness of Mr} and ${\rm wt}\left(\eaction{2}\mb{x}_0\right) = -\theta + \alpha_2 - \delta = -2\alpha_1 - \delta \notin \sfDel^-(w_1)$.
Thus $M_{1,a}'$ is a non-trivial proper submodule of $M_{1,a}$.
}
\end{ex}

\section{Prefundamental modules} \label{sec:prefundamentals}
\subsection{Character formulas}
For a $U_q(\bo)$-module $V$ in $\mc{O}$, it has a weight space decomposition with respect to $\left\{\, k_i \,|\, i \in \I\,\right\}$, that is, 
\begin{equation*}
	V = \bigoplus_{\omega \in {\mf t}^*} V_{\omega},
\end{equation*}
where $V_{\omega}$ is given as in \eqref{eq:weight space wrt t}.
For $\omega \in {\mf t}^*$, we also write $V_\beta = V_\omega$ for $\beta \in \mr{\sfP}$ with $\widetilde{\beta} = \omega$.
Let $\Z[\![ e^\beta ]\!]_{\beta \in \mr{\sfP}}$ be the ring of formal power series in variables $e^\beta$ for $\beta \in \mr{\sfP}$ with $e^\beta \cdot e^\gamma = e^{\beta+\gamma}$.
Then we define
\begin{equation*}
		{\rm ch}(V) = \sum_{\beta \in \mr{\sfP}} \left( \dim V_\beta \right) e^{\beta} \in \Z[\![ e^\beta ]\!]_{\beta \in \mr{\sfP}},
\end{equation*}
which is called the {\it ordinary $\mr{\g}$-character} of $V$.
We often call it the {\it character} of $V$ if there is no confusion.

\begin{thm} \label{thm:main2}
For $r \in \I$, we have
\begin{equation} \label{eq:product formula}
\begin{split}
	{\rm ch}\left( M_r \right) = \prod_{\beta \in \sfDel^+\left(\tr^{-1} \right) \cap \mr{\sfDel}^+} \frac{1}{ (1-e^{-\ov{\beta}} )^{[\beta]_r}},
\end{split}
\end{equation}
where $[\beta]_r \in \mathbb{Z}_+$ is given by $\overline{\beta} = \sum_{s \in \I}\, [\beta]_s \alpha_s$.
Furthermore, we have ${\rm ch} \left( M_r \right) = {\rm ch} \left( L_{r, a}^\pm \right)$ for any $a \in \mathbb{C}^\times$.
\end{thm}
\begin{proof}
By Proposition \ref{prop:Delta+ by tr}, \eqref{eq:product formula} follows from Definition \ref{df:Uq-w} and Proposition \ref{prop:Delta+ by tr}.
By \cite{HJ,Lee,Neg} (cf.~\cite{MY14}), the character of $L_{r, 1}^-$ is equal to the right-hand side of \eqref{eq:product formula}.
Furthermore, it is known in \cite{HJ} that
\begin{equation*}
	{\rm ch} (L_{r, 1}^-) = {\rm ch} (L_{r, a}^-) = {\rm ch} (L_{r, a}^+)
\end{equation*}
for any $a \in \mathbb{C}^\times$.
As a result,
we conclude ${\rm ch} \left( M_r \right) = {\rm ch} \left( L_{r, a}^\pm \right)$. This completes the proof.
\end{proof}

\begin{df}
{\em 
For $r \in \I$, we say that $r$ is {\it cominuscule} if $a_r = 1$, that is, $(\lar,\alpha) = 1$ for $\alpha \in \mr{\sfDel}_+$ with $(\lar,\alpha) \neq 0$.
We say that $L_{r,a}^\pm$ are {\it cominuscule} if $r$ is cominuscule.
}
\end{df}

\begin{rem}
{\em 
By Proposition \ref{prop:Delta+ by tr}, we have $\Delta^+(w_r) \subset \mr{\sfDel}_+$ if $r \in \I$ is cominuscule. 
Hence, when $r \in \I$ is cominuscule, one may regard $U_q^-(w_r)$ as a $\bk$-subalgebra of $U_q^-(\cg)$.
}
\end{rem}

\subsection{Realization of cominuscule negative prefundamental modules} \label{subsec:realization of L-}
Let $a \in \bk^\times$ be given. 
For coincidence with \cite{JKP23} (see Remark \ref{rem:compare with previous work}),
we renormalize the $0$-action \eqref{eq:Borel actions} by 
\begin{equation} \label{eq:normalized 0-action}
	\eaction{0} (u) := a \left[ \frac{ 1 }{\left(q_{\tau(0)}-q_{\tau(0)}^{-1}\right) q_{\tau(0)}^2 \prod_{i \in \I} \left( 1-q_i^2 \right)^{a_i} }  \mb{x}_0 \right] u = a {\texttt x}_0 u,
\end{equation}
and then we denote by $\rho_{r,a}$ the representation of $U_q(\bo)$ on $M_r$.
Note that $\rho_{r,a}$ coincides with $\rho_{r,a}^-$ in \cite[Theorem 4.20]{JKP23}.
We denote by $M_{r,a}$ the $U_q(\bo)$-module corresponding to $\rho_{r,a}$. 

We state our second main result in this paper, which is an extension of \cite[Theorem 4.22]{JKP23} to all cominuscule $r \in \I$ (see~Table \ref{tab:dynkins}).

\begin{thm} \label{thm:main3}
For cominuscule $r \in \I$ and $a \in \bk^\times$,
there exists $\eta_r \in \bk^\times$ such that as a $U_q(\bo)$-module,
\begin{equation*}
	M_{r,a} \cong L_{r,a\eta_r}^-,
\end{equation*}
where $\eta_r$ is given by 
\begin{equation} \label{eq:eta of negative prefundamental modules}
	\eta_r = \begin{cases}
 		(-1)^n q^{-n-1}(q-q^{-1}) o(r) & \text{for type $A_n^{(1)}$,} \\
 		-q^{-2n-2}(q^{-1}+q)(1-q^{-4}) o(1) & \text{for type $B_n^{(1)}$,} \\
 		-q^{-n-1}(1-q^{-4}) o(n) & \text{for type $C_n^{(1)}$,} \\
 		-q^{-2n+2}(q-q^{-1}) o(r) & \text{for type $D_n^{(1)}$,} \\
 		-q^{-11}(1-q^{-2}) o(r) & \text{for type $E_6^{(1)}$,} \\
 		-q^{-17}(1-q^{-2}) o(7) & \text{for type $E_7^{(1)}$.}
	\end{cases}
\end{equation}
\end{thm}
\begin{proof}
The proof is given in Section \ref{subsec:hw}.
\end{proof}

\begin{rem}
{\em 	
We have ${\rm ch}\left(M_{r,a}\right) = {\rm ch}(L_{r,b}^-)$ by Theorem \ref{thm:main2} for all $r \in \I$, $a, b \in \mathbb{C}^\times$. However, it follows from Proposition \ref{prop:ar ge 2 cases} that $M_{r,a}$ is not irreducible for $r \in \I$ with $a_r \ge 2$, so $M_{r,a}$ cannot be isomorphic to $L_{r,a}^-$ (up to a shift of spectral parameter).
}
\end{rem}

\begin{cor}
	For $r \in \I$ and $a \in \bk^\times$, the $U_q(\bo)$-module $M_{r,a}$ with respect to \eqref{eq:Borel actions} and \eqref{eq:normalized 0-action} is irreducible if and only if $r$ is cominuscule.
\end{cor}
\begin{proof}
	It follows from Proposition \ref{prop:ar ge 2 cases} and Theorem \ref{thm:main3}.
\end{proof}

\begin{rem} \label{rem:Mw}
{\em 
		By Lemma \ref{lem:closedness of Mr} and \cite[Lemma 3.4.2]{Kas91},
		the $\bk$-linear space $M_w$ $(w \in \eW)$ \eqref{eq:def of Mw} has a $U_q(\bo)$-module structure by the actions of $\mb{e}_i'$ and $k_i$ for $i \in I$, where the action of $k_i$ is given as in \eqref{eq:Borel actions} (i.e.~we define the $0$-action by $\mb{e}'_0$, not multiplication of $\xz$).
		Then $M_w \in \mc{O}$ by definition, but it is not irreducible in general, in particular, in the case where $w = w_r$. Thus it seems necessary to modify the action of $\mb{e}_0$ on $M_r$ as in \eqref{eq:Borel actions}.
}	
\end{rem}

\begin{rem} \label{rem:remark for main result}
{\em 
It is shown in \cite[Section 5]{HJ} that $L_{r,a}^+$ is a $\sigma$-twisted dual of $L_{r,a}^-$ as a representation of the asymptotic algebra, where $\sigma$ is an isomorphism of $\bk$-algebras from $U_q(\g)$ to $U_{q^{-1}}(\g)$. Hence, for cominuscule $r \in \I$, one can also obtain a realization of $L_{r,a}^+$ by using Theorem \ref{thm:main3} following \cite{HJ}. 
Note that the authors use $S^{-1}$ \eqref{eq:antipode} to give a $U_q(\bo)$-module structure on the graded dual space of $L_{r,a}^-$ in the reason of \cite[Remark 3.19]{HJ}.
On the other hand, it is known in \cite{JKP23} for types $A_n^{(1)}$ and $D_n^{(1)}$ with cominuscule $r \in \I$ that there exists another way to give a $U_q(\bo)$-modules structure on $M_{r}$ directly, which is isomorphic to $L_{r,a}^+$ up to shift of spectral parameter.
This can be extended to all cominuscule $r$, which will be discussed in Section \ref{subsec:realization of L+}.
}
\end{rem}

\subsection{Proof of Theorem \ref{thm:main3}} \label{subsec:hw}
The rest of this section is devoted to the proof of Theorem \ref{thm:main3} and Remark \ref{rem:remark for main result}.

To prove Theorem \ref{thm:main3}, we shall compute explicitly the $\ell$-highest weight of a non-zero scalar multiple of $1 \in M_{r,a}$, denoted by ${\bf 1}$. 
Note that the parameter $\eta_r \in \bk^\times$ in Theorem \ref{thm:main3} is determined after we finish to compute the highest $\ell$-weight of $M_{r,a}$.

Let $2\rho$ be the sum of positive roots of $\mathring{\mf g}$.
Take a reduced expression $t_{2\rho}=s_{i_1} \cdots s_{i_N}\in W$ and define a doubly infinite sequence
\begin{equation} \label{eq:doubly infinite sequence}
	\dots, i_{-2}, \,i_{-1}, \,i_0, \,i_1, \, i_2 \dots,
\end{equation}
by setting $i_k = i_{k\, ({\rm mod}\, N)}$ for $k \in \mathbb{Z}$.
Then \eqref{eq:doubly infinite sequence} gives a convex order $\prec$ (cf.~\cite[Definition 2.1]{MT18}) on the set of positive roots of $\mf{g}$:
\begin{equation}\label{eq:convex order}
\beta_0 \prec \beta_{-1} \prec \dots \prec \delta \prec \cdots \prec \beta_2 \prec \beta_1, 
\end{equation}
where $\beta_k$ is given by
\begin{equation*}
	\beta_k = 
	\begin{cases} 
	s_{i_0}s_{i_{-1}}\cdots s_{i_{k+1}} (\alpha_{i_k}) & \textrm{if $k \leqq 0$}, \\
	s_{i_1}s_{i_2}\cdots s_{i_{k-1}}(\alpha_{i_k}) & \textrm{if $k > 0$}.
	\end{cases}
\end{equation*}

For $k\in\Z$, we define the root vector $\E_{\beta_k}$ by
\begin{equation} \label{eq:real root vectors}
\E_{\beta_k} = 
\left\{ 
\begin{array}{ll}
T_{i_0}^{-1} T_{i_{-1}}^{-1} \cdots T_{i_{k+1}}^{-1}(e_{i_k}) & \textrm{if $k \leqq 0$}, \\
T_{i_1} T_{i_2} \cdots T_{i_{k-1}}(e_{i_k}) & \textrm{if $k > 0$}.
\end{array}
\right.
\end{equation}
We define the root vector $\F_{\beta_k}$ in the same way with $e_{i_k}$ replaced by $f_{i_k}$ in \eqref{eq:real root vectors}.
Recall that $\E_{\beta_k} \in U_q^+ (\g)$ and $\F_{\beta_k} \in U_q^- (\g)$ for $k \in \mathbb{Z}$ (see \cite[Proposition 40.1.3]{Lu10}).
In particular, if $\beta_k = \alpha_i$ for some $i \in I$, then $\E_{\beta_k} = e_i$ and $\F_{\beta_k} = f_i$ (cf. \cite[Corollary 4.3]{MT18}).

\begin{lem} {\em (\!\!\cite[Lemma 4.2]{JKP23})} \label{lem:independence of E}
	For $i \in I$ and $k > 0$, the root vectors $\E_{k\delta\pm\alpha_i}$ and $\F_{k\delta\pm\alpha_i}$ are independent of the choice of a reduced expression of $t_{2\rho}$.
	\qed
\end{lem}

Let $o : I \,\rightarrow\, \{ \pm 1 \}$ be a map such that $o(i) = -o(j)$ whenever $a_{ij} < 0$.
Recall that
\begin{equation}\label{eq:Drinfeld generators in terms of affine root vectors}
	\psi_{i, k}^+ = o(i)^k (q_i-q_i^{-1})C^{+ \frac{k}{2}}k_i
	\left( \E_{k\delta-\alpha_i}\E_{\alpha_i} - q_i^{-2}\E_{\alpha_i} \E_{k\delta-\alpha_i} \right)
\end{equation}
for $i \in I$ and $k > 0$ (\!\!\cite{BN},\,\,cf.~\cite[Proposition 1.2]{BCP}).
We remark that the right-hand side of \eqref{eq:Drinfeld generators in terms of affine root vectors} is also independent of the choice of a reduced expression of $t_{2\rho}$ due to Lemma \ref{lem:independence of E}.
\smallskip

The following lemma together with \eqref{eq:Drinfeld generators in terms of affine root vectors} enables us to compute the action of $\psi_{i, k}^+$ for $i \in I$ and $k \in \mathbb{Z}_{>0}$ on a $U_q({\mf b})$-module.

\begin{lem}{\em (cf. \cite[Lemma 4.3]{JKP23})} \label{lem:inductive formula}
For $i \in I$ and $k \in \mathbb{Z}_{>0}$, we have
\begin{equation*}
\begin{split}
\E_{(k+1)\delta-\alpha_i} &= -\frac{1}{q_i+q_i^{-1}}
\Big(\,
\E_{\delta-\alpha_i} \E_{\alpha_i}\E_{k\delta-\alpha_i} - q_i^{-2} \E_{\alpha_i}\E_{\delta-\alpha_i}\E_{k\delta-\alpha_i} \\
& \qquad \qquad \quad\,\,\,- \E_{k\delta -\alpha_i} \E_{\delta-\alpha_i} \E_{\alpha_i} + q_i^{-2} \E_{k\delta - \alpha_i} \E_{\alpha_i} \E_{\delta - \alpha_i}\,
\Big)\,.
\end{split}
\end{equation*}
\end{lem}
\begin{proof}
	The proof is almost identical with the one of \cite[Lemma 4.3]{JKP23}) by using \cite[Lemma 4.2]{B94}.
	Note that $\wtd{\psi}_{i, k} = \E_{k\delta-\alpha_i}\E_{\alpha_i} - q_i^{-2}\E_{\alpha_i} \E_{k\delta-\alpha_i}$ in the case of non-symmetric types.
\end{proof}

For $i \in I$, we define the $\bk$-linear subspace $U_q^+(\delta-\alpha_i)$ of $U_q^+(\g)_{\delta-\alpha_i}$ by
\begin{equation*}
	U_q^+(\delta-\alpha_i) = \sum_{\beta} U_q^+(\g)_{\beta} \cdot \E_{\delta-\alpha_i-\beta},
\end{equation*}
where the sum is over $\beta \in \Delta^+ \setminus \left\{ \theta - \alpha_i \right\}$.
The following formula plays an important role in the subsequent computation of the highest $\ell$-weight of $M_{r,a}$.
\begin{lem} \label{lem:crucial formulas of affine root vectors}
Assume that $\g$ is of type $X_n^{(1)}$ and $r$ is cominuscule.
Then we have
\begin{equation} \label{eq:description of root vector}
	\E_{\delta-\alpha_r} - \mathbf{E}_{\delta-\alpha_r} \in U_q^+(\delta-\alpha_r),
\end{equation}
where $\mb{E}_{\delta-\alpha_r} \in U_q^+(\g)$ is given by
{\allowdisplaybreaks
\begin{gather*}
	\mb{E}_{\delta-\alpha_r} = 
	\begin{cases}
		\displaystyle (-q^{-1})^{n-1} \left( e_{r+1} \dots e_{n-1} e_n \right) \left( e_{r-1} \dots e_2 e_1 \right) e_0 & \text{if $X = A$ and $r \in \I$,} \\
		\displaystyle q^{-2n} (q^{-1}+q) \left( e_2 e_3 \dots e_n \right) \left( e_n e_{n-1} \dots e_2 \right) e_0 & \text{if $X = B$ and $r = 1$,} \\
		\displaystyle \left( \frac{q^{-1}}{q^{-1}+q} \right)^{n-1} \left(e_{n-1}\right)^2 \left(e_{n-2}\right)^2 \dots \left(e_1\right)^2 e_0 & \text{if $X = C$ and $r = n$,} \\
		q^{-2n+4} \left( e_2 e_3 \dots e_{n-2} e_{n-1} \right) \left( e_n e_{n-2} \dots e_3 e_2 \right) e_0	& \text{if $X = D$ and $r = 1$,} \\
		q^{-2n+4} \left( e_{n-2} e_{n-3} \dots e_2 e_1 \right) \left( e_n e_{n-2} \dots e_3 e_2 \right) e_0	& \text{if $X = D$ and $r = n-1$,} \\
		q^{-2n+4} \left( e_{n-2} e_{n-3} \dots e_2 e_1 \right) \left( e_{n-1} e_{n-2} \dots e_3 e_2 \right) e_0	& \text{if $X = D$ and $r = n$,} \\
		q^{-10} \left( e_3 e_4 e_5 e_6 e_2 e_4 e_5 e_3 e_4 e_2 \right) e_0	& \text{if $X_n = E_6$ and $r = 1$,} \\
		q^{-10} \left( e_5 e_4 e_3 e_1 e_2 e_4 e_3 e_5 e_4 e_2 \right) e_0	& \text{if $X_n = E_6$ and $r = 6$,} \\
		q^{-16} \left( e_6 e_5 e_4 e_2 e_3 e_1 e_4 e_5 e_6 e_3 e_4 e_5 e_2 e_4 e_3 e_1 \right) e_0	& \text{if $X_n = E_7$ and $r = 7$.}
	\end{cases}
\end{gather*}}
\end{lem}
\begin{proof}
The formulas \eqref{eq:description of root vector} for $X = A,\, D$ are already proven in \cite[Lemma 4.7 and Lemma 4.12]{JKP23}.
We refer the reader to \cite[Lemma 6.5]{Liu} for the proof of \eqref{eq:description of root vector} when $X = C$.
Note that the reduced expression of $w_r$ in \cite[Remark 6.3]{Liu} coincides with the one in Remark \ref{eq:rex of wr} up to $2$-braid relations.
\smallskip

Let us verify two remaining cases $X = B,\, E$. For simplicity, let us write 
\begin{equation*}
a,b]_q := ab - q^{-1} ba.
\end{equation*}

{\it Case 1}. $X = B$.
Set $x_0 := T_{n-1} \dots T_2 (e_1)$ and $x := T_1 T_2 \dots T_{n-1} T_n (x_0)$.
It is straightforward to check
\begin{equation*}
	x_0 = e_{n-1}, e_{n-2}]_q, e_{n-3}]_q, \dots, e_2]_q, e_1]_q
\end{equation*}
(cf.~\cite[Lemma 4.6]{JKP23}).
By \cite[\S 39.2.3]{Lu10}, we have
\begin{equation*}
\begin{split}
	T_{n-1} T_n (e_{n-1}) &= T_n^{-1} \left( T_n T_{n-1} T_n (e_{n-1}) \right)
	= e_{n-1} e_n^{(2)} - q^{-2} e_n e_{n-1} e_n + q^{-4} e_n^{(2)}e_{n-1}. 
\end{split}
\end{equation*}
Then we compute
\begin{equation*}
\begin{split}
	& T_1 T_2 \dots T_{n-2} T_{n-1} T_n \left(\, e_{n-1}, e_{n-2}]_q, e_{n-3}]_q, \dots, e_2]_q, e_1]_q \,\right) \\
	&= T_1 T_2 \dots T_{n-2} T_{n-1} T_n (e_{n-1}), e_{n-1}]_q, e_{n-2}]_q, \dots, e_3]_q, e_2]_q \\
	&= T_1 T_2 \dots T_{n-2} \left( e_{n-1} e_n^{(2)} - q^{-2} e_n e_{n-1} e_n + q^{-4} e_n^{(2)}e_{n-1} \right) , e_{n-1}]_q, e_{n-2}]_q, \dots, e_3]_q, e_2]_q.
\end{split}
\end{equation*}
Hence, we have
\begin{equation*}
	x = e_1, e_2]_q,e_3]_q,\dots]_q,e_{n-1}]_q,e_n^{(2)}]_q, e_{n-1}]_q, e_{n-2}]_q, \dots, e_3]_q, e_2]_q + X,
\end{equation*}
where $X$ is the sum of the remaining monomials $e_{i_1}\dots e_{i_k}$ such that $e_{i_k} \neq e_1$.
In this case, the Dynkin diagram automorphism $\tau_1$ of $B_n^{(1)}$ is given by 
$$\tau_1(1) = 0,\quad \tau_1(0) = 1, \quad \tau_1(k) = k$$ 
for $2 \le k \le n$.
Since $\E_{\beta-\alpha_1} = \tau_1(x)$, we prove the desired formula by the above formula of $x$.
\vskip 2mm

{\it Case 2}. $X = E$.
First, we prove the case of $n = 6$ with $r = 1$.
One check
\begin{equation*}
\begin{split}
	    y &= T_1T_3T_4T_5T_6 T_2T_4T_5 T_3T_4T_2 T_1T_3T_4T_5(e_6) \\
	      &= e_1, e_3 ]_q, e_4 ]_q, e_2 ]_q, e_5 ]_q, e_4 ]_q, e_3 ]_q, e_6]_q, e_5]_q, e_4 ]_q, e_2 ]_q.
\end{split}
\end{equation*}
In this case, the Dynkin diagram automorphism $\tau$ is given by 
$$\tau(0) = 6,\, \tau(1) = 0,\, \tau(2) = 5,\, \tau(3) = 2,\, \tau(4) = 4,\, \tau(6) = 1,\, \tau(5) = 3.$$
Since $\E_{\delta-\alpha_1} = \tau(y)$, the desired formula follows from the above formula of $y$.
The proof for the case of $n=6$ with $r=6$ is similar.
\smallskip

Second, we prove the case of $n=7$ with $r=7$.
One check
\begin{equation*}
\begin{split}
	z & = T_7T_6T_5T_4T_3T_1 T_2T_4T_3 T_5T_4T_2 T_6T_5T_4T_3T_1 T_7T_6T_5T_4T_3T_2T_4T_5T_6(e_7) \\
	  &\,\, = e_7, e_6]_q, e_5]_q, e_4]_q, e_2]_q, e_3]_q, e_4]_q, e_5]_q, e_1]_q, e_3]_q, e_4]_q, e_6]_q, e_5]_q, e_2]_q, e_4]_q, e_3]_q, e_1]_q.
\end{split}
\end{equation*}
In this case, the Dynkin diagram automorphism $\tau$ is given by 
$$
\tau(0)=7,\, \tau(1)=6,\, \tau(2)=2,\, \tau(3)=5,\, \tau(4)=4,\, \tau(5)=3,\, \tau(6)=1,\, \tau(7)=0.
$$
Hence, the desired formula follows from the above formula of $z$.
\end{proof}

Let us recall \eqref{eq:Borel actions} and \eqref{eq:normalized 0-action}.
Following \cite{JKP23}, it is not difficult to obtain inductively the following formulas from \cite[Lemma 3.6]{JKP23}, Lemma \ref{lem:inductive formula}, Lemma \ref{lem:crucial formulas of affine root vectors}, Corollary \ref{cor:derivation on root vectors in B}, and Corollary \ref{cor:derivation on root vectors in C} (cf.~Table \ref{tab:dynkins}).

\begin{prop} \label{prop:action of affine root vectors for rho-}
Assume that $\g$ is of type $X_n^{(1)}$ and $r$ is cominuscule.
Then we have
{\allowdisplaybreaks
\begin{align*}
	\E_{k\delta-\alpha_i}(1) &=
	\begin{cases}
		\displaystyle (-1)^{kn-1}q^{-k(n+1)+2}(q-q^{-1})^{k-1}a^k f_r & \text{if $X = A$ and $i = r$,} \\
		\displaystyle (-1)^{k-1} \left( q^{-2n} (q^{-1}+q) a \right)^k q^{-2(k-1)} (1-q^{-4})^{k-1} f_1 & \text{if $X = B$ and $i = 1$,} \\
		\displaystyle (-1)^{k-1} \left( q^{-n+1}a \right)^k q^{-2(k-1)} (1-q^{-4})^{k-1} f_n & \text{if $X = C$ and $i = n$,} \\
		\displaystyle (-1)^{k-1} q^{-2k(n-1)+2} (q-q^{-1})^{k-1}a^k f_r & \text{if $X = D$ and $i = r$,} \\
		\displaystyle (-1)^{k-1}\left( aq^{-10} \right)^k q^{-k+1} (1-q^{-2})^{k-1} f_r & \text{if $X_n = E_6$ and $i = r$,} \\
		\displaystyle (-1)^{k-1} \left( aq^{-16} \right)^k q^{-k+1} (1-q^{-2})^{k-1} f_7 & \text{if $X_n = E_7$ and $i = 7$,} \\
		\displaystyle 0 & \text{if $i \in \I \setminus \{ r \}$,}
	\end{cases}
	\\
	\E_{\delta-\alpha_i}(f_r) &=
	\begin{cases}
		\displaystyle (-1)^{n-1} q^{-n+1}a f_r^2 & \text{if $X = A$ and $i = r$,} \\
		\displaystyle q^{-2n} (q^{-1}+q)a f_1^2 & \text{if $X = B$ and $i = 1$,} \\
		\displaystyle q^{-n+1}a f_n^2 & \text{if $X = C$ and $i = n$,} \\
		\displaystyle q^{-2n+4}a f_r^2 & \text{if $X = D$ and $i = r$,} \\
		\displaystyle aq^{-10} f_r^2 & \text{if $X_n = E_6$ and $i = r$,} \\
		\displaystyle aq^{-16} f_7^2 & \text{if $X_n = E_7$ and $i = 7$,}\\
		\displaystyle 0 & \text{if $i \in \I \setminus \{ r \}$.}
	\end{cases}
\end{align*}\!\!}
\end{prop}

\begin{rem} \label{rem:action of ei on dual vector}
{\em 
\mbox{}
\begin{enumerate}[(1)]
	\item For types $A_n^{(1)}$, $D_n^{(1)}$, and $E_{6,7}^{(1)}$ with cominuscule $r \in \I$, 
	let us take a reduced expression $s_{i_1}\dots s_{i_\ell}$ of $w_r$ such that $\beta_1 = \alpha_r$ and $\beta_\ell = \theta$ as in Remark \ref{rem:compare with previous work}.
	Then we have
	\begin{equation*} 
		\quad 
		F^{\rm up}\left( \beta_k \right) \in \iota_{\varpi_r}\left( G^{\rm up}(\varpi_r) \right), 
		\quad
		e_i'\left( F^{\rm up}\left( \beta_k \right) \right) =
		\begin{cases}
			F^{\rm up}\left( \beta_k - \alpha_i \right) & \text{if $(\alpha_i, \beta_k) = 1 + \delta_{i,r}$\,,} \\
			0 & \text{otherwise,}
		\end{cases}
	\end{equation*}
	where $\iota_{\varpi_r}$ is the embedding from the irreducible highest weight $U_q(\cg)$-module $V(\varpi_r)$ associated to $\varpi_r$ into $U_q^-(\cg)$, and $G^{\rm up}(\varpi_r)$ is the dual canonical basis of $V(\varpi_r)$ (see \cite[Section 3.1]{JKP23} for more details and references therein). 
	This is proved in \cite[Lemma 3.6]{JKP23} for types $A$ and $D$, but the proof also holds for type $E$ with help of Proposition \ref{prop:length of tr} (or Remark \ref{eq:rex of wr}), since we use only some properties of dual canonical basis of $U_q^-(\cg)$ in the proof of \cite[Lemma 3.6]{JKP23}, which are available for general types \cite{Lu10,Ki12}.
	\smallskip
	
	\item For types $B_n^{(1)}$ and $C_n^{(1)}$ with cominuscule $r \in \I$, we also have
	\begin{equation*}
		F^{\rm up}\left( \beta_k \right) \in \iota_{\varpi_r}\left( G^{\rm up}(\varpi_r) \right).
	\end{equation*}
	However, we cannot apply the argument in the proof of \cite[Lemma 3.6]{JKP23} to obtain a formula of the action of $e_i'$ on $F^{\rm up}(\beta_k)$.
	Alternatively, we use an explicit realization of $F^{\rm up}(\beta_k)$ in terms of quantum shuffle algebra \cite{Le04}, and then we yield similar formulas in Corollary \ref{cor:derivation on root vectors in B} and Corollary \ref{cor:derivation on root vectors in C} for types $B$ and $C$, respectively (cf.~\cite[Proposition 3.14]{JKU24}).
	\smallskip

	\item 
	Thanks to (1) and (2), the crystal graph of $V(\varpi_r)$ is helpful to compute the formulas of Proposition \ref{prop:action of affine root vectors for rho-}.
	For the convenience of the reader, we provide the crystal graph of $V(\varpi_r)$ for type $E_6$ (resp.~$E_7$) \cite{Jang22}, which coincides with the following enumeration of $\sfDel^+(w_r)$ with respect to \eqref{eq:root wrt rex} up to shift of weights, where we denote by $\substack{\texttt{b} \\ \texttt{acdef}}$ (resp.~$\substack{\texttt{b}\,\,\, \\ \texttt{acdefg}}$) the positive root $\texttt{a}\alpha_1 + \texttt{b}\alpha_2 + \texttt{c}\alpha_3 + \texttt{d}\alpha_4 + \texttt{e}\alpha_5 + \texttt{f}\alpha_6 \in \mr{\sfDel}^+$ (resp.~$\texttt{a}\alpha_1 + \texttt{b}\alpha_2 + \texttt{c}\alpha_3 + \texttt{d}\alpha_4 + \texttt{e}\alpha_5 + \texttt{f}\alpha_6 + \texttt{g}\alpha_7 \in \mr{\sfDel}^+$), and	two positive roots are connected by a single edge if their difference is equal to $\pm \alpha_i$ for some $i \in \I$.
	\smallskip

	\begin{enumerate}[]
	\item {\it Type $E_6$ with $r = 6$}. Note that $\sfDel^+(w_1)$ is obtained from $\sfDel^+(w_6)$ by replacing $6 \leftrightarrow 1$ and $5 \leftrightarrow 3$.
	\begin{equation*} 
	\resizebox{0.85\columnwidth}{!}
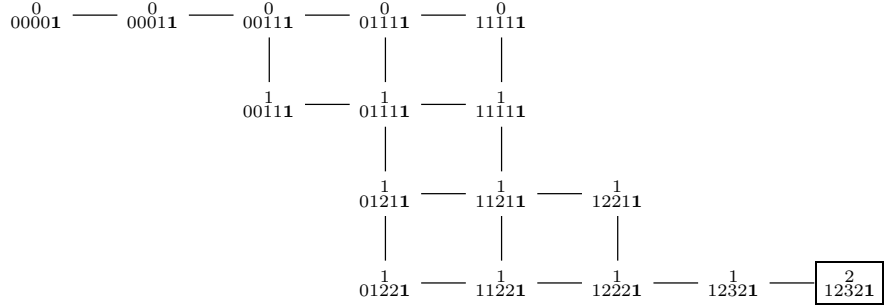
{%
	\begin{tikzpicture}[scale=1, baseline=(current  bounding  box.center)]
	  \draw[-] (0.3,0) -- (0.8,0); 
	  \draw[-] (1.6,0) -- (2.1,0); 
	  \draw[-] (2.9,0) -- (3.4,0); 
	  \draw[-] (4.2,0) -- (4.7,0); 
	  
	  \draw[-] (2.5,-0.25) -- (2.5,-0.75); 
	  \draw[-] (3.8,-0.25) -- (3.8,-0.75); 
	  \draw[-] (5.1,-0.25) -- (5.1,-0.75); 
	  
	  \draw[-] (2.9,-1) -- (3.4,-1); 
	  \draw[-] (4.2,-1) -- (4.7,-1); 
	  
	  \draw[-] (3.8, -1.25) -- (3.8, -1.75); 
	  \draw[-] (5.1, -1.25) -- (5.1, -1.75); 
	  
	  \draw[-] (4.2, -2) -- (4.7, -2); 
	  \draw[-] (5.5, -2) -- (6, -2); 
	  
	  \draw[-] (3.8, -2.25)-- (3.8, -2.75);  
	  \draw[-] (5.1, -2.25)-- (5.1, -2.75);  
	  \draw[-] (6.4, -2.25)-- (6.4, -2.75);  
	  
	  \draw[-] (4.2, -3) -- (4.7, -3); 
	  \draw[-] (5.5, -3) -- (6, -3); 
	  \draw[-] (6.8, -3) -- (7.3, -3); 
	  \draw[-] (8.1, -3) -- (8.6, -3); 
	  
	  \node at (-0.1,0) {\scalebox{0.8}{${}_{\substack{0 \\ 0000{\bf 1}}}$}}; 
	  \node at (1.2,0) {\scalebox{0.8}{${}_{\substack{0 \\ 0001{\bf 1}}}$}}; 
	  \node at (2.5,0) {\scalebox{0.8}{${}_{\substack{0 \\ 0011{\bf 1}}}$}}; 
	  \node at (3.8,0) {\scalebox{0.8}{${}_{\substack{0 \\ 0111{\bf 1}}}$}}; 
	  \node at (5.1,0) {\scalebox{0.8}{${}_{\substack{0 \\ 1111{\bf 1}}}$}}; 
	  
	  \node at (2.5, -1) {\scalebox{0.8}{${}_{\substack{1 \\ 0011{\bf 1}}}$}}; 
	  \node at (3.8, -1) {\scalebox{0.8}{${}_{\substack{1 \\ 0111{\bf 1}}}$}}; 
	  \node at (5.1, -1) {\scalebox{0.8}{${}_{\substack{1 \\ 1111{\bf 1}}}$}}; 
	  
	  \node at (3.8, -2) {\scalebox{0.8}{${}_{\substack{1 \\ 0121{\bf 1}}}$}}; 
	  \node at (5.1, -2) {\scalebox{0.8}{${}_{\substack{1 \\ 1121{\bf 1}}}$}}; 
	  \node at (6.4, -2) {\scalebox{0.8}{${}_{\substack{1 \\ 1221{\bf 1}}}$}}; 
	  
	  \node at (3.8, -3) {\scalebox{0.8}{${}_{\substack{1 \\ 0122{\bf 1}}}$}}; 
	  \node at (5.1, -3) {\scalebox{0.8}{${}_{\substack{1 \\ 1122{\bf 1}}}$}}; 
	  \node at (6.4, -3) {\scalebox{0.8}{${}_{\substack{1 \\ 1222{\bf 1}}}$}}; 
	  \node at (7.7, -3) {\scalebox{0.8}{${}_{\substack{1 \\ 1232{\bf 1}}}$}}; 
	  \node at (9, -3) {\scalebox{0.8}{${}_{\boxed{\substack{2 \\ 1232{\bf 1}}}}$}}; 
	\end{tikzpicture}}%
	\end{equation*}
	
	\item {\it Type $E_7$ with $r = 7$}.
	\begin{equation*} 
	\resizebox{0.85\columnwidth}{!}{%
	\quad
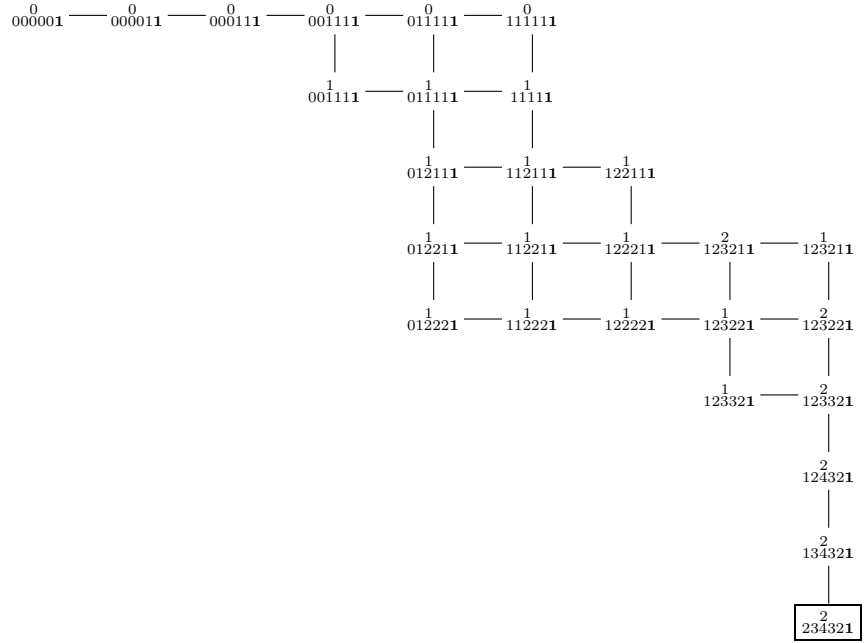
	\begin{tikzpicture}[baseline=(current  bounding  box.center)]
	  \draw[-] (-1, 0) -- (-0.5, 0);
	  \draw[-] (0.3,0) -- (0.8,0); 
	  \draw[-] (1.6,0) -- (2.1,0); 
	  \draw[-] (2.9,0) -- (3.4,0); 
	  \draw[-] (4.2,0) -- (4.7,0); 
	  
	  \draw[-] (2.5,-0.25) -- (2.5,-0.75); 
	  \draw[-] (3.8,-0.25) -- (3.8,-0.75); 
	  \draw[-] (5.1,-0.25) -- (5.1,-0.75); 
	  
	  \draw[-] (2.9,-1) -- (3.4,-1); 
	  \draw[-] (4.2,-1) -- (4.7,-1); 
	  
	  \draw[-] (3.8, -1.25) -- (3.8, -1.75); 
	  \draw[-] (5.1, -1.25) -- (5.1, -1.75); 
	  
	  \draw[-] (4.2, -2) -- (4.7, -2); 
	  \draw[-] (5.5, -2) -- (6, -2); 
	  
	  \draw[-] (3.8, -2.25)-- (3.8, -2.75);  
	  \draw[-] (5.1, -2.25)-- (5.1, -2.75);  
	  \draw[-] (6.4, -2.25)-- (6.4, -2.75);  
	  
	  \draw[-] (4.2, -3) -- (4.7, -3); 
	  \draw[-] (5.5, -3) -- (6, -3); 
	  \draw[-] (6.8, -3) -- (7.3, -3); 
	  \draw[-] (8.1, -3) -- (8.6, -3); 
	  
	  \draw[-] (3.8, -3.25)-- (3.8, -3.75);  
	  \draw[-] (5.1, -3.25)-- (5.1, -3.75);  
	  \draw[-] (6.4, -3.25)-- (6.4, -3.75);  
	  \draw[-] (7.7, -3.25) -- (7.7, -3.75); 
	  \draw[-] (9, -3.25) -- (9, -3.75); 
	  
	  \draw[-] (4.2, -4) -- (4.7, -4); 
	  \draw[-] (5.5, -4) -- (6, -4); 
	  \draw[-] (6.8, -4) -- (7.3, -4); 
	  \draw[-] (8.1, -4) -- (8.6, -4); 
	  
	  \draw[-] (7.7, -4.25) -- (7.7, -4.75);  
	  \draw[-] (9, -4.25) -- (9, -4.75); 
	  
	  \draw[-] (8.1, -5) -- (8.6, -5); 
	  
	   \draw[-] (9, -5.25) -- (9, -5.75); 
	   \draw[-] (9, -6.25) -- (9, -6.75); 
	   \draw[-] (9, -7.25) -- (9, -7.75); 
	  
	  
	  \node at (-1.4,0) {\scalebox{0.8}{${}_{\substack{0\,\,\, \\ 00000{\bf 1}}}$}}; 
	  \node at (-0.1,0) {\scalebox{0.8}{${}_{\substack{0\,\,\, \\ 00001{\bf 1}}}$}}; 
	  \node at (1.2,0) {\scalebox{0.8}{${}_{\substack{0\,\,\, \\ 00011{\bf 1}}}$}}; 
	  \node at (2.5,0) {\scalebox{0.8}{${}_{\substack{0\,\,\, \\ 00111{\bf 1}}}$}}; 
	  \node at (3.8,0) {\scalebox{0.8}{${}_{\substack{0\,\,\, \\ 01111{\bf 1}}}$}}; 
	  \node at (5.1,0) {\scalebox{0.8}{${}_{\substack{0\,\,\, \\ 11111{\bf 1}}}$}}; 
	
	  \node at (2.5, -1) {\scalebox{0.8}{${}_{\substack{1\,\,\, \\ 00111{\bf 1}}}$}}; 
	  \node at (3.8, -1) {\scalebox{0.8}{${}_{\substack{1\,\,\, \\ 01111{\bf 1}}}$}}; 
	  \node at (5.1, -1) {\scalebox{0.8}{${}_{\substack{1\,\,\, \\ 1111{\bf 1}}}$}}; 
	  
	  \node at (3.8, -2) {\scalebox{0.8}{${}_{\substack{1\,\,\, \\ 01211{\bf 1}}}$}}; 
	  \node at (5.1, -2) {\scalebox{0.8}{${}_{\substack{1\,\,\, \\ 11211{\bf 1}}}$}}; 
	  \node at (6.4, -2) {\scalebox{0.8}{${}_{\substack{1\,\,\, \\ 12211{\bf 1}}}$}}; 
	  
	  \node at (3.8, -3) {\scalebox{0.8}{${}_{\substack{1\,\,\, \\ 01221{\bf 1}}}$}}; 
	  \node at (5.1, -3) {\scalebox{0.8}{${}_{\substack{1\,\,\, \\ 11221{\bf 1}}}$}}; 
	  \node at (6.4, -3) {\scalebox{0.8}{${}_{\substack{1\,\,\, \\ 12221{\bf 1}}}$}}; 
	  \node at (7.7, -3) {\scalebox{0.8}{${}_{\substack{2\,\,\, \\ 12321{\bf 1}}}$}}; 
	  \node at (9, -3) {\scalebox{0.8}{${}_{\substack{1\,\,\, \\ 12321{\bf 1}}}$}}; 
	  \node at (3.8, -4) {\scalebox{0.8}{${}_{\substack{1\,\,\, \\ 01222{\bf 1}}}$}}; 
	  
	  \node at (5.1, -4) {\scalebox{0.8}{${}_{\substack{1\,\,\, \\ 11222{\bf 1}}}$}}; 
	  \node at (6.4, -4) {\scalebox{0.8}{${}_{\substack{1\,\,\, \\ 12222{\bf 1}}}$}}; 
	  \node at (7.7, -4) {\scalebox{0.8}{${}_{\substack{1\,\,\, \\ 12322{\bf 1}}}$}}; 
	  \node at (9, -4) {\scalebox{0.8}{${}_{\substack{2\,\,\, \\ 12322{\bf 1}}}$}}; 
	  \node at (7.7, -5) {\scalebox{0.8}{${}_{\substack{1\,\,\, \\ 12332{\bf 1}}}$}}; 
	  \node at (9, -5) {\scalebox{0.8}{${}_{\substack{2\,\,\, \\ 12332{\bf 1}}}$}}; 
	  \node at (9, -6) {\scalebox{0.8}{${}_{\substack{2\,\,\, \\ 12432{\bf 1}}}$}}; 
	  \node at (9, -7) {\scalebox{0.8}{${}_{\substack{2\,\,\, \\ 13432{\bf 1}}}$}}; 
	  \node at (9, -8) {\scalebox{0.8}{${}_{\boxed{\substack{2\,\,\, \\ 23432{\bf 1}}}}$}}; 
	\end{tikzpicture}}%
	\end{equation*}
\end{enumerate}
\noindent
where the boxed positive root is equal to $\theta$.
Similarly, when $r$ is cominuscule, the enumeration of $\sfDel^+(w_r)$ for types $A_n$ and $D_n$ under the same convention as above coincides with the crystal graph associated to $\varpi_r$ \cite{Kw13, JK19} up to shift of weights.
Note that the crystal graph of $V(\varpi_r)$ for types $B_n$ and $C_n$ can be realized from the crystal of $U_q^-(\cg)$ \cite{Kas91}, where $\cg$ is of type $A_{2n-1}$, by using the similarity of crystals \cite{Kas96} in terms of $\sfDel^+(w_r)$ (see~\cite{Kw18b}).
\end{enumerate}
}
\end{rem}

We denote by ${\bf 1}$ a non-zero scalar multiple of $1 \in M_{r,a}$.
Let us recall \eqref{eq:Borel actions} and \eqref{eq:normalized 0-action}.
Finally, we are ready to compute the $\ell$-weight of ${\bf 1}\in M_{r,a}$.

\begin{lem} \label{lem:l-highest weight}
For cominuscule $r \in \I$ and $a \in \bk^\times$, 
there exists $\eta_r \in \bk^\times$ such that the $\ell$-weight $\left( \Psi_i(z) \right)_{i \in I}$ of ${\bf 1} \in M_{r,a}$ is given as follows:
\begin{equation} \label{eq:hw l-wt of Lra}
		\Psi_i(z) = 
			\begin{cases}
				\displaystyle
				\frac{1}{1-a\eta_rz} & \text{if $i = r$,}\\
				1 & \text{if $i \neq r$,}
			\end{cases}
	\end{equation}
\end{lem}
\begin{proof}
Set $V = M_{r,a}$ for simplicity.
Since $r$ is cominuscule,
by Proposition \ref{prop:action of affine root vectors for rho-}, there exist $b \in \bk^\times$ and $c \in \bk$ such that 
\begin{equation} \label{eq:initial computation}
	\E_{\alpha_r}\E_{\delta-\alpha_r} {\bf 1} = b {\bf 1} \quad \text{ and } \quad
	\E_{2\delta-\alpha_r} {\bf 1} = c \E_{\delta-\alpha_r} {\bf 1}.
\end{equation}
Note that $\E_{\delta-\alpha_r} {\bf 1} \neq 0$, which spans the weight space $V_{\overline{\alpha}_r^{-1}}$ of $V$.
Thanks to Lemma \ref{lem:independence of E}, one can find an element $h \in U_q(\bo)^0$ such that
\begin{equation} \label{eq:important identity}
	\left[ h,\, \E_{k\delta-\alpha_r} \right] = \E_{\left( k+1 \right)\delta-\alpha_r}
\end{equation}
for $k \in \Z_{>0}$ (see \cite[Lemma 4.2]{B94} and \cite[Proposition 1.2]{BCP}).
Since ${\bf 1}$ is a simultaneous eigenvector of $U_q(\bo)^0$, there exists $d \in \bk$ such that $h {\bf 1} = d {\bf 1}$. 
By \eqref{eq:important identity}, we have
\begin{equation*}
	h \E_{\delta-\alpha_r} {\bf 1} 
	= \E_{2\delta-\alpha_r} {\bf 1} + \E_{\delta-\alpha_r} h {\bf 1}
	= (c+d)\E_{\delta-\alpha_r} {\bf 1},
\end{equation*}
that is, the element $h$ acts on $V_{\overline{\alpha}_r^{-1}}$ as the scalar multiplication by $c+d$.
\smallskip

We prove that for $k \in \Z_{>0}$, 
\begin{equation} \label{eq:action of Ekd-ar on 1}
	\E_{k\delta - \alpha_r} {\bf 1} = c^{k-1} \E_{\delta-\alpha_r} {\bf 1}.
\end{equation}
Suppose that the equality holds for $k$. Then it follows from \eqref{eq:important identity} and induction hypothesis that
\begin{equation*}
\begin{split}
	\E_{(k+1)\delta-\alpha_r}{\bf 1}
	= h \E_{k\delta-\alpha_r}{\bf 1} - \E_{k\delta-\alpha_r}h {\bf 1} 
	= (c+d)c^{k-1} \E_{\delta-\alpha_r}{\bf 1} - dc^{k-1} \E_{\delta-\alpha_r}{\bf 1}
	= c^k\E_{\delta-\alpha_r}{\bf 1}.
\end{split}
\end{equation*}
This prove \eqref{eq:action of Ekd-ar on 1}.
By \eqref{eq:Drinfeld generators in terms of affine root vectors} and \eqref{eq:action of Ekd-ar on 1}, we compute
\begin{equation*}
	\Psi_{r,k}^+({\bf 1})
	= o(r)^k (q_r-q_r^{-1}) (-q_r^{-2}) k_r \E_{\alpha_r} \E_{k\delta-\alpha_r}({\bf 1})
	= - b(q_r^{-1}-q_r^{-3}) o(r)^k c^{k-1} {\bf 1}.
\end{equation*}
Thus, the $r$-th component of the $\ell$-weight of ${\bf 1}$ is given by
\begin{equation} \label{eq:r-comp of weight of 1}
	\Psi_r(z) = 1 - b(q_r^{-1} - q_r^{-3}) \sum_{k=1}^\infty o(r)^k c^{k-1} z^k = 1 - \frac{b(q_r^{-1} - q_r^{-3})o(r) z}{1 - o(r)c z}.
\end{equation}

On the other hand, let $i \neq r$. 
By Proposition \ref{prop:Delta+ by tr}, we know that $-\alpha_i$ is not a weight of $V$. 
Therefore, we conclude that $\E_{\delta-\alpha_i}{\bf 1} = 0$. This implies $\Psi_i(z) = 1$ by \eqref{eq:Drinfeld generators in terms of affine root vectors} and Lemma \ref{lem:inductive formula}.
\smallskip

The remaining task is to determine whether \eqref{eq:r-comp of weight of 1} is of the form in \eqref{eq:hw l-wt of Lra} with respect to $\rho$.
Since we have by \eqref{eq:r-comp of weight of 1}
\begin{equation*}
	\Psi_r(z) = \frac{1-\left( c + b(q_r^{-1}-q_r^{-3}) \right) o(r)z}{1-o(r)c z},
\end{equation*}
it is enough to show that $c = b (q_r^{-3} - q_r^{-1})$. 
It follows from the following formulas which can be obtained by using Proposition \ref{prop:action of affine root vectors for rho-} (cf.~\cite[(4.17)]{JKP23}):
{\allowdisplaybreaks
\begin{align*}
	\E_{\alpha_r} \E_{\delta-\alpha_r} {\bf 1} &=
	\begin{cases}
		a(-q^{-1})^{n-1}  {\bf 1} & \text{if $X = A$,} \\
		aq^{-2n}(q^{-1}+q) {\bf 1} & \text{if $X = B$,} \\
		a q^{-n+1} {\bf 1} & \text{if $X = C$,} \\
		aq^{-2n+4} {\bf 1} & \text{if $X = D$,} \\
		aq^{-10} {\bf 1} & \text{if $X_n = E_6$,} \\
		aq^{-16} {\bf 1} & \text{if $X_n = E_7$,} \\
	\end{cases}
	\\
	\E_{2\delta-\alpha_r}{\bf 1} &= 
	\begin{cases}
		(-1)^{n-1} aq^{-n+1}(q^{-3}-q^{-1}) \E_{\delta-\alpha_r}{\bf 1} & \text{if $X = A$,} \\
		aq^{-2n}(q^{-1}+q) (q_1^{-3} - q_1^{-1}) \E_{\delta-\alpha_1}{\bf 1}  & \text{if $X = B$,} \\
		aq^{-n+1}(q_n^{-3}-q_n^{-1}) \E_{\delta-\alpha_n}{\bf 1} & \text{if $X = C$,} \\
		aq^{-2n+4}(q^{-3}-q^{-1})  \E_{\delta-\alpha_r}{\bf 1} & \text{if $X = D$,} \\
		aq^{-10} (q^{-3}-q^{-1})  \E_{\delta-\alpha_r}{\bf 1} & \text{if $X_n = E_6$,} \\
		aq^{-16} (q^{-3}-q^{-1})  \E_{\delta-\alpha_7}{\bf 1} & \text{if $X_n = E_7$.}
	\end{cases}
\end{align*}\!\!}
Hence, we complete the proof.
\end{proof}

\begin{proof}[Proof of Theorem \ref{thm:main3}]
By Lemma \ref{lem:l-highest weight}, there exists $c_r \in \bk^\times$ such that the maximal element ${\bf 1} \in M_{r,a}$ has the $\ell$-weight $\left(\Psi_i(z)\right)_{i \in \I}$ of the following form
\begin{equation*} 
\begin{split}
	\Psi_i(z) = 
	\begin{cases}
		\displaystyle \frac{1}{1-a\eta_r z} & \text{if $i = r$,} \\
		\displaystyle 1 & \text{otherwise.}
	\end{cases}
\end{split}
\end{equation*}
Let $N$ be the $U_q(\bo)$-submodule of $M_{r,a}$ generated by ${\bf 1}$.
By Remark \ref{rem:prefundamentals are building blocks of O},
$L_{r,a\eta_r}^-$ is a maximal quotient of $N$, but it follows from Theorem \ref{thm:main2} that $M_{r,a}$ is isomorphic to $L_{r,a\eta_r}^-$ as a $U_q(\bo)$-module.
\end{proof}

\subsection{Realization of cominuscule positive prefundamental modules} \label{subsec:realization of L+}
Let $\texttt{x}_0$ be given as in \eqref{eq:normalized 0-action}.
Following \cite[Section 3.3]{JKP23}, we define $\bk$-linear operators on $M_r$ by 
\smallskip
\begin{equation} \label{eq:another Borel actions}
\kaction{i}(u) = 
\begin{cases}
q^{\left(\al_i,{\rm wt}(u)\right)}u \quad &\text{if $i \neq 0$}\,,  \\
q^{-\left(\theta,{\rm wt}(u)\right)}u  \quad &\text{if $i=0$}\,,
\end{cases}
\qquad
\eaction{i} (u) = 
\begin{cases}
	 \mb{e}_i'(u)  & \text{ if } i \neq 0,  \\
	aq^{-\left( \theta, {\rm wt}(u) \right)} u \texttt{x}_0.  & \text{ if } i = 0,
\end{cases}
\end{equation}
where $\texttt{x}_0$ is given as in \eqref{eq:normalized 0-action}.
Note that the $\bk$-linear operators $\eaction{i}$ for $i \in I$ are well-defined due to \eqref{eq:LS formula} and Lemma \ref{lem:closedness of Mr}.
When $r \in \I$ is cominuscule, \eqref{eq:another Borel actions} coincides with \cite[(3.20)]{JKP23} for types $A$ and $D$ (cf.~Remark \ref{rem:compare with previous work}).
\smallskip

The following theorem is an extension of \cite[Theorem 3.10 and Theorem 4.17]{JKP23} to all cominuscule $r \in \I$.

\begin{thm} \label{thm:main3-1}
For cominuscule $r \in \I$, the map
\begin{equation*}
	\rho :
	\xymatrix@R=0em
	{
		U_q(\bo) \ar@{->}[r] & {\rm End}_\bk(M_r) \\
		e_i \ar@{|->}[r] & \mb{e}_i \\
		k_i \ar@{|->}[r] & \mb{t}_i
	}
\end{equation*}
is a representation of $U_q(\bo)$ on $M_r$.
Moreover, the $U_q(\bo)$-module $M_r$ is isomorphic to $L_{r,-a\eta_r}^+$.
\end{thm}
\begin{proof}
First, we check the quantum Serre relations for $\left\{ \, \mb{e}_i \, | \, i \in I \,\right\}$.
It is enough to verify the relations involving $\eaction{0}$.
Let us fix a homogenous element $u \in M_r$ of weight $\beta$.
\smallskip

{\it Case 1}. 
Suppose that $(\alpha_i, \theta) = 0$.
In this case, we have
\begin{equation*}
\begin{split}
	\eaction{i} \eaction{0} (u) 
	= aq^{-(\theta,\beta)}\left\{ \mb{e}_i'(u)\texttt{x}_0 + q^{(\beta,\alpha_i)} u \mb{e}_i'(\texttt{x}_0) \right\} = aq^{-(\theta,\beta+\alpha_i)} \mb{e}_i'(u)\texttt{x}_0 = \eaction{0} \eaction{i} (u).
\end{split}
\end{equation*}
Here $\mb{e}_i'(\texttt{x}_0) = 0$ follows from \cite[Lemma 3.8]{JKP23} for types $ADE$ (cf.~Remark \ref{rem:action of ei on dual vector}),  Corollary \ref{cor:derivation on root vectors in B} for type $B$, and Corollary \ref{cor:derivation on root vectors in C} for type $C$.
\smallskip

{\it Case 2}. 
Suppose that $(\alpha_i, \theta) \neq 0$.
In this case, we check only the following relation which occurs in the case of $a_{i0} = -2$:
\begin{equation} \label{eq:quantum Serre with four terms}
	\eaction{i}^3 \eaction{0} - (q_i^{-2} + 1 + q_i^2)\eaction{i}^2 \eaction{0}\eaction{i} + (q_i^{-2} + 1 + q_i^2)\eaction{i}\eaction{0}\eaction{i}^2 - \eaction{0}\eaction{i}^3=0,
\end{equation}
since the other relations (in which $a_{i0} = a_{0i} = -1$) can be shown by the similar way as in the proof of \cite[Theorem 3.10]{JKP23} with the help of Corollary \ref{cor:derivation on root vectors in B} for type $B$. 
In fact, the above relation occurs only when $X=C$ and $q_i = q$.
Set $s = (\theta, \beta)$ and $t = (\beta, \alpha_i)$.
By using Corollary \ref{cor:derivation on root vectors in C}, we have
\begin{equation*}
\begin{split}
	(\eaction{i})^3 \eaction{0} (u) &=
	aq^{-s} (\mb{e}_i')^3(u)\texttt{x}_0 + \left( aq^{-s+t-4} + aq^{-s+t-2} + aq^{-s+t} \right) (\mb{e}_i')^2(u) \mb{e}_i'(\texttt{x}_0) \\
	& \qquad\qquad\qquad\quad\,\,\,\, + \left( aq^{-s+2t-4} + aq^{-s+2t-2} + aq^{-s+2t} \right) \mb{e}_i'(u) (\mb{e}_i')^2(\texttt{x}_0), \\
	\eaction{i}^2 \eaction{0} \eaction{i} (u) &=
	aq^{-s+2} (\mb{e}_i')^3(u)\texttt{x}_0 + (aq^{-s+t-2} + aq^{-s+t} (\mb{e}_i')^2(u) \mb{e}_i'(\texttt{x}_0)  \\
	& \qquad\qquad\qquad\qquad\,\,\,\, + q^{-s+2t-2} \mb{e}_i'(u)(\mb{e}_i')^2(\texttt{x}_0), \\
	\eaction{i} \eaction{0} \eaction{i}^2 (u) &=
	aq^{-s+4} (\mb{e}_i')^3(u)\texttt{x}_0 + aq^{-s+t}(\mb{e}_i')^2(u) \mb{e}_i'(\texttt{x}_0),\quad
	\eaction{0} \eaction{i}^3 (u) = aq^{-s+6}(\mb{e}_i')^3(u)\texttt{x}_0.
\end{split}
\end{equation*}
This proves \eqref{eq:quantum Serre with four terms}. 
By {\it Case 1}--{\it Case 2}, we prove our first assertion.
\smallskip

Second, we compute the highest $\ell$-weight of ${\bf 1} \in M_r$ with respect to \eqref{eq:another Borel actions}.
Following the proof of Lemma \ref{lem:l-highest weight}, it is enough to determine 
\begin{equation*}
	\E_{k\delta-\alpha_i}(1) \quad \text{ and } \quad
	\E_{\delta-\alpha_i}(f_r).
\end{equation*}
By Lemma \ref{lem:crucial formulas of affine root vectors}, we have
{\allowdisplaybreaks
\begin{align}
&	\E_{k\delta-\alpha_i}(1) =
	\begin{cases}
		\displaystyle a (-q^{-1})^{n-1} f_r & \text{if $X = A$, $i = r$, and $k = 1$,} \\
		\displaystyle a q^{-2n} (q^{-1} + q) f_1 & \text{if $X = B$, $i = 1$, and $k = 1$,} \\
		\displaystyle a q^{-n+1} f_n & \text{if $X = C$, $i = n$, and $k = 1$,} \\
		\displaystyle a q^{-2n+4} f_r & \text{if $X = D$, $i = r$, and $k = 1$,} \\
		\displaystyle a q^{-10} f_r & \text{if $X_n = E_6$, $i = r$, and $k = 1$,} \\
		\displaystyle a q^{-16} f_7 & \text{if $X_n = E_7$, $i = 7$, and $k = 1$,} \\
		\displaystyle 0 & \text{otherwise,}
	\end{cases} \label{eq:Edel-ali on 1} \\
&	\E_{\delta-\alpha_i}(f_r) =
	\begin{cases}
		\displaystyle a (-q^{-1})^{n-1} q^2 f_r^2 & \text{if $X = A$ and $i = r$,} \\
		\displaystyle a q^{-2n+2} (q^{-1} + q) f_1^2 & \text{if $X = B$ and $i = 1$,} \\
		\displaystyle a q^{-n+5} f_n^2 & \text{if $X = C$ and $i = n$,} \\
		\displaystyle a q^{-2n+6} f_r^2 & \text{if $X = D$ and $i = r$,} \\
		\displaystyle a q^{-8} f_r^2 & \text{if $X_n = E_6$ and $i = r$,} \\
		\displaystyle a q^{-14} f_7^2 & \text{if $X_n = E_7$ and $i = 7$,}\\
		\displaystyle 0 & \text{otherwise.}
	\end{cases} \label{eq:Edel-ali on fr}
\end{align}\!\!\!}
By \eqref{eq:Edel-ali on 1} and \eqref{eq:Edel-ali on fr}, we have $c = 0$ in \eqref{eq:initial computation}. 
Since the scalar $b$ in \eqref{eq:initial computation} can be determined explicitly by \eqref{eq:Edel-ali on 1} and \eqref{eq:Edel-ali on fr} (in fact, a scalar multiple of $a \in \bk^\times$),
it follows from \eqref{eq:r-comp of weight of 1} that 
$\Psi_r(z) = 1 + a\eta_r z$, where $\eta_r$ is given as in \eqref{eq:eta of negative prefundamental modules}.
Combining it with Theorem \ref{thm:main2}, we prove our second assertion.
\end{proof}

\begin{rem}
{\em 
In \cite{P}, the author develops a functor $\mc{F}$ from a category of $U_{q^{-1}}(\bo)$-modules to that of $U_q(\bo)$-modules such that $\mc{F}(L_{r,a}^{\pm, q^{-1}})$ are also prefundamental modules, where there exists $\gamma \in \mathbb{C}^\times$ satisfying $\mc{F}(L_{r,a}^{\pm,q^{-1}}) \cong L_{r,a\gamma}^\mp$ for all $i \in \I$ and $a \in \mathbb{C}^\times$. Here $L_{r,a}^{\pm,q^{-1}}$ are prefundamental modules over $U_{q^{-1}}(\bo)$.
Although the $U_q(\bo)$-actions of $\mc{F}(L_{r,a}^{\pm,q^{-1}})$ following \cite{P} seems to be different from \eqref{eq:another Borel actions}, the parameter $\eta_r$ \eqref{eq:eta of negative prefundamental modules} plays a similar role with the parameter $\gamma$. So one may expect a connection between our work and \cite{P}, but it is not yet clear to us.
}	
\end{rem}

\appendix
\section{Quantum shuffle algebras and root vectors}  \label{appendixA}

\noindent
Let us summarize the description of dual root vectors for types $B_n$ and $C_n$ following \cite{Le04}.

\subsection{Quantum shuffle algebras}
Let $\mathcal{A}=\{\,w_1<\dots<w_n\,\}$ be a linearly ordered set of alphabets, and let $\mathcal{W}$ be the set of words with alphabets in $\mathcal{A}$. 
We put $w[i_1,\dots,i_k]=w_{i_1}\dots w_{i_k}$ for $i_1,\dots,i_k$.
We regard $\mathcal{W}$ as the ordered set with respect to the following lexicographic order:
\begin{equation*}
	w[i_1, \dots, i_a] < w[j_1, \dots, j_b]
\end{equation*}
if there exists an $r$ such that $w_{i_r} < w_{j_r}$ and $w_{i_s} = w_{j_s}$ for $s < r$, or if $a < b$ and $w_{i_s} = w_{j_s}$ for $1 \le s \le a$.
For a word $w=w[i_1,\dots,i_k] \in \mc{W}$, we say that $w$ is a {\it Lyndon word} if $w<w[i_{s},i_{s+1}\dots,i_k]$ for $2\le s\le k$. Let $l\in \mc{W}$ be a Lyndon word. Write $l=l_1l_2$ such that $l_1\neq l$ and $l_1$ is a Lyndon word of maximal length. Then $l_2$ is also a Lyndon word.
We call $l=l_1l_2$ a {\it costandard factorization} of $l$.
For ${\bf w} = w[i_1, \dots, i_s] \in \mathcal{W}$, we denote by $|{\bf w}|$ the weight $\alpha_{i_1} + \dots + \alpha_{i_s}$. 
\smallskip

Let $\GL$ be the set of {\it good} Lyndon words.
Then we have a bijection 
\begin{equation}\label{eq:good Lyndon to positive root}
\begin{split}
 \xymatrixcolsep{2pc}\xymatrixrowsep{0pc}\xymatrix{
  \GL \ \ar@{->}[r] &\ \mathring{\Delta}^+ \\
  w[i_1,\dots,i_r] \ \ar@{|->}[r] &\ \alpha_{i_1}+\dots+\alpha_{i_r}
 }.
\end{split}
\end{equation}
where $\mathring{\Delta}^+$ is the set of positive roots for $\cg$.
For $\beta\in \mathring{\Delta}^+$, 
let $l(\beta)$ be the word in $\GL$ corresponding to $\beta$ under \eqref{eq:good Lyndon to positive root}.
We denote by $\prec$ the linear order on $\mathring{\Delta}^+$ induced from that on $\GL$ under \eqref{eq:good Lyndon to positive root}.

\begin{rem}
{\em 
	Note that the $\prec$ coincides with the convex order induced from the reduced expression of the longest element $w_0$ obtained from Remark \ref{eq:rex of wr} up to $2$-braid relations.
}
\end{rem}

Let $\mathcal{F}$ be the free $\bk$-algebra generated by $\mathcal{A}$.
We denote by $\cdot$ the multiplication of $\mc{F}$, which is given by the concatenation on words. 
We define 
\begin{equation*} 
	(x\cdot i) * (y\cdot j) = (x * (y\cdot j))\cdot i + q^{-(|x\cdot i|,|j|)} ((x\cdot i)*y)\cdot j,
\end{equation*}
and $x * w[]  = w[] * x = x$, for homogenous $x, y \in \mc{F}$ and $i, j \in \mc{A}$.
Then $(\mathcal{F}, *)$ is an associative $\bk$-algebra, which is called the {\it quantum shuffle algebra}.
Note that there exists an $\bk$-algebra embedding $\Phi$ from $U_q^-(\cg)$ into $(\mathcal{F}, *)$.
\smallskip

From now on, we identify $U_q^-(\cg)$ with its image of $\Phi$, denoted by $\mathcal{U}$. For $\beta \in \mathring{\Delta}^+$, we still denote by $F^{\rm up}(\beta)$ the image of the dual root vector $F^{\rm up}(\beta) \in U_q^-(\cg)$ (see \cite[Section 5]{Le04} for more detail).
If there is no confusion, then we omit the $\cdot$\,, that is, write $x \cdot y = xy$ for homogeneous $x, y \in \mc{F}$.
For homogenous ${\bf w} \in \mc{F}$, we also denote by $|{\bf w}|$ its weight in an obvious manner.
\smallskip

For each $i \in \I$, the derivation $e_i'$ is interpreted as an $\bk$-linear endomorphism of $\mathcal{F}$, denoted by $\mathbf{e}_i'$, by 
\begin{equation} \label{eq:derivation in F}
	\mathbf{e}_i' \left( w[i_1, \dots, i_k] \right) = \delta_{i,i_k} w[i_1, \dots, i_{k-1}] \text{ and } 
	\mathbf{e}_i' \left( w[] \right) = 0,
\end{equation}
and then $\Phi(e_i'u) = \mathbf{e}_i'\Phi(u)$ for $u \in U_q^-(\cg)$.
\smallskip

\begin{rem} \label{rem:our setting vs. Leclerc's setting}
{\em 
In \cite{Le04}, the author chose the braid group symmetry on $U_q^+(\mathring{\mf g})$ as $T_{i, -1}'$ with $q = v^{-1}$, while we have took $T_{i, 1}''$ on $U_q^-(\mathring{\mf g})$ with $q = v$ (see \cite[Section 37.1.3]{Lu10}). 
It is known in \cite[Section 37.2.4]{Lu10} (cf.~\cite{Sa94, Ki12}) that the following diagram commutes:
\begin{align*} 
\begin{gathered}
\xymatrixcolsep{2pc}\xymatrixrowsep{2pc} 
\xymatrix{   
	U_q(\mathring{\mf g}) \ar@{->}[r]^{\omega} \ar@{->}[d]_{T_{i,1}''} & U_q(\mathring{\mf g}) \ar@{->}[r]^{{}^-} & U_q(\mathring{\mf g}) \ar@{->}[d]^{T_{i,-1}'} \\
	U_q(\mathring{\mf g}) \ar@{->}[r]^{\omega} & U_q(\mathring{\mf g}) \ar@{->}[r]^{{}^-} & U_q(\mathring{\mf g})
} 
\end{gathered}
\end{align*}
where $\omega$ is the automorphism of $U_q(\mathring{\mf g})$ such that $\omega(e_i) = f_i$, $\omega(f_i)=e_i$ and $\omega(k_i) = k_i^{-1}$ for all $i \in \I$. Hence one can calculate the (dual) PBW basis of $U_q^-(\cg)$ by using \cite{Le04}.
}
\end{rem}

\subsection{Root vectors of type $B_n$}
We take $\mathcal{A} = \I$ with the linear oder given by
\begin{equation} \label{eq:co-standard order}
	1 < 2 < \dots < n.
\end{equation}
We remark that \eqref{eq:co-standard order} is reversed to the one in \cite[Section 8.2]{Le04}, so we compute the dual root vectors with respect to \eqref{eq:co-standard order}.
The set $\GL$ is given by
\begin{equation*}
\begin{split}
	\GL &= \left\{ w[i,\dots,j] \, | \, 1\le i \le j \le n \right\} \\ 
		&\,\,\,\quad \scalebox{1.2}{$\cup$} \left\{ w[i,i+1,\dots,n] w[n,n-1,\dots,j] \, | \, 1\le i < j \le n \right\}.
\end{split}
\end{equation*}

\begin{prop} {\rm (cf.~\cite[Corollary 6.7]{CHW})} \label{prop:dual root vectors in B}
For $\beta \in \mathring{\Delta}^+$, one has
\begin{equation*}
	F^{\rm up}(\beta) =
	\begin{cases}
		w[i,\dots,j] & \text{if $l(\beta) = w[i,\dots,j]$,} \\
		(q^{-1} + q)\,w[i,i+1,\dots,n]\, w[n,n-1,\dots,j] & \text{otherwise.}
	\end{cases}
\end{equation*}
\end{prop}
\begin{proof}
In the case of $l(\beta) = w[i,\dots,j]$, it follows from the induction on $j-i$ that we have $F^{\rm up}(\beta) = w[i, \dots, j]$.
Let us consider the case of $l(\beta) = w[i,i+1,\dots,n] w[n,n-1,\dots,j]$.
Let 
$${\bf w} = w[i,i+1,\dots,n] w[n,n-1,\dots,j].$$
It follows from \cite[Theorem 5]{Le04} that ${\bf w} \in \mc{U}$.
Clearly, $\max({\bf w}) = l(\beta)$.
Then let us write ${\bf w}$ as a linear combination of $F^{\rm up}(\beta)$ for $\beta \in \mathring{\Delta}^+$ as follows:
\begin{equation*}
	{\bf w} = \sum_{\substack{{\bf u} \le l(\beta) \\ |{\bf u}| = |{\bf w}|}} \lambda_{\bf u} F^{\rm up}({|{\bf u}|}).
\end{equation*}
But it is known in \cite[Corollary 27]{Le04} that $l(\beta)$ is the smallest good word of weight $\beta$.
Note that it follows from \cite[Theorem 4.29]{Ki12} that each dual root vector is contained in the dual canonical basis.
By \cite[Theorem 40]{Le04}, we conclude that ${\bf w} = \lambda_{l(\beta)} F^{\rm up}(\beta)$ and $\lambda_{l(\beta)}^{-1} = (q^{-1}+q)$ by calculating the coefficient of $l(\beta)$ in $F^{\rm up}(\beta)$.
\end{proof}

\begin{cor} \label{cor:derivation on root vectors in B}
Let $i \in \I$. Assume that $e_i' (F^{\rm up}(\beta)) \neq 0$. 
Then we have
\begin{equation*}
	e_i' (F^{\rm up}(\beta)) = F^{\rm up}(\beta - \alpha_i).
\end{equation*}
\end{cor}
\begin{proof}
It follows directly from \eqref{eq:derivation in F} and Proposition \ref{prop:dual root vectors in B}.
\end{proof}

\subsection{Root vectors of type $C_n$}
We take $\mathcal{A} = \I$ with the linear oder given by
\begin{equation*}
	n < n-1 < \dots < 1,
\end{equation*}
which corresponds to the standard order in \cite{Le04} under the identification of $k$ with $n-k+1$ for $1 \le k \le n$.
The set $\GL$ is given by 
\begin{equation*}
\begin{split}
	\GL &= \left\{ w[j,j-1, \dots, i] \, | \, 1 \le i \le j \le n \right\} \\
		& \,\,\,\quad \scalebox{1.2}{$\cup$} \left\{ w[n,n-1,\dots,j, n-1,n-2,\dots,k] \, | \, 1 \le j \le k < n \right\}.
\end{split}
\end{equation*}

\begin{prop}\!\!\!{\em \cite[Lemma 54]{Le04}\,} \label{prop:dual root vectors in C}
For $\beta \in \mathring{\Delta}^+$, one has
\begin{equation*}
	F^{\rm up}(\beta) = 
	\begin{cases}
		w[j,j-1, \dots, i] & \text{if $l(\beta) = w[j,j-1, \dots, i]$}, \\
		w[n] \left( w[n-1, n-2, \dots, j] * w[n-1, n-2, \dots, k] \right) & \text{otherwise.}
	\end{cases}
\end{equation*}
\end{prop}

\begin{cor} \label{cor:derivation on root vectors in C}
Let $i \in \I$. Assume that $e_i' (F^{\rm up}(\beta)) \neq 0$. 
Then we have
\begin{equation*}
	e_i' (F^{\rm up}(\beta)) = 
	\begin{cases}
		(q^{-1} + q)\,F^{\rm up}(\beta - \alpha_i) & \text{if $\beta-2\alpha_i \in \mathring{\Delta}^+$,} \\
		F^{\rm up}(\beta - \alpha_i) & \text{otherwise.}
	\end{cases}
\end{equation*}
\end{cor}
\begin{proof}
We only prove the case of $\beta - 2\alpha_i \in \mathring{\Delta}^+$, since the other case follows immediately from Proposition \ref{prop:dual root vectors in C}.
Assume that $l(\beta) = w[n,n-1,\dots,j, n-1,n-2,\dots,k]$.
Then we have $i = j = k$ (otherwise, $\beta - 2\alpha_i \notin \mathring{\Delta}^+$).
Since $\mathbf{e}_i'$ is a $\bk$-linear derivation of $\mathcal{U}$ with respect to $*$, we have 
\begin{equation} \label{eq:ei Fbeta when beta-2al in Delta+ in C}
\begin{split}
	\mathbf{e}_i' \left( F^{\rm up}(\beta) \right)
	&= w[n] \left( w[n-1, n-2, \dots, i+1]*w[n-1,n-2,\dots,i] \right. \\
	& \quad + \left. q^{-1} w[n-1,n-2,\dots,i]*w[n-1,n-2,\dots,i+1] \right)
\end{split}
\end{equation}
Let us take $\beta_k \in \mathring{\Delta}^+$ $(k=1,2)$ so that
\begin{equation*}
	l(\beta_k) = 
	\begin{cases}
		w[n-1, n-2, \dots, i+1] & \text{if $k = 1$,} \\
		w[n-1,n-2,\dots,i] & \text{if $k = 2$,}
	\end{cases}
\end{equation*}
with $\beta_1 \prec \beta_2$.
By Proposition \ref{prop:dual root vectors in C}, we have $F^{\rm up}(\beta_k) = l(\beta_k)$ for $k = 1, 2$. 
Then the desired formula follows from \eqref{eq:ei Fbeta when beta-2al in Delta+ in C} and the dual formula of \eqref{eq:LS formula} (see \cite[Theorem 4.27]{Ki12}).
\end{proof}

\providecommand{\bysame}{\leavevmode\hbox to3em{\hrulefill}\thinspace}
\providecommand{\MR}{\relax\ifhmode\unskip\space\fi MR }
\providecommand{\MRhref}[2]{%
  \href{http://www.ams.org/mathscinet-getitem?mr=#1}{#2}
}
\providecommand{\href}[2]{#2}

\end{document}